\documentclass[11pt]{article}
\usepackage{latexsym,amsfonts,amssymb,amsmath,amsthm}
\usepackage{graphicx}

\usepackage[usenames,dvipsnames]{color}
\usepackage{ulem}

\usepackage{color}

\begin{document}

\newtheorem{tm}{Theorem}[section]
\newtheorem{prop}[tm]{Proposition}
\newtheorem{defin}[tm]{Definition}
\newtheorem{coro}[tm]{Corollary}
\newtheorem{lem}[tm]{Lemma}
\newtheorem{assumption}[tm]{Assumption}
\newtheorem{rk}[tm]{Remark}
\newtheorem{nota}[tm]{Notation}
\numberwithin{equation}{section}

\newcommand{\stk}[2]{\stackrel{#1}{#2}}
\newcommand{\dwn}[1]{{\scriptstyle #1}\downarrow}
\newcommand{\upa}[1]{{\scriptstyle #1}\uparrow}
\newcommand{\nea}[1]{{\scriptstyle #1}\nearrow}
\newcommand{\sea}[1]{\searrow {\scriptstyle #1}}
\newcommand{\csti}[3]{(#1+1) (#2)^{1/ (#1+1)} (#1)^{- #1
 / (#1+1)} (#3)^{ #1 / (#1 +1)}}
\newcommand{\RR}[1]{\mathbb{#1}}

\newcommand{\rd}{{\mathbb R^d}}
\newcommand{\ep}{\varepsilon}
\newcommand{\rr}{{\mathbb R}}
\newcommand{\alert}[1]{\fbox{#1}}
\newcommand{\eqd}{\sim}
\def\p{\partial}
\def\R{{\mathbb R}}
\def\N{{\mathbb N}}
\def\Q{{\mathbb Q}}
\def\C{{\mathbb C}}
\def\l{{\langle}}
\def\r{\rangle}
\def\t{\tau}
\def\k{\kappa}
\def\a{\alpha}
\def\la{\lambda}
\def\De{\Delta}
\def\de{\delta}
\def\ga{\gamma}
\def\Ga{\Gamma}
\def\ep{\varepsilon}
\def\eps{\varepsilon}
\def\si{\sigma}
\def\Re {{\rm Re}\,}
\def\Im {{\rm Im}\,}
\def\E{{\mathbb E}}
\def\P{{\mathbb P}}
\def\Z{{\mathbb Z}}
\def\D{{\mathbb D}}
\newcommand{\ceil}[1]{\lceil{#1}\rceil}

\title{Global existence and asymptotic behavior  of classical solutions to a parabolic-elliptic chemotaxis system with  logistic source on $\mathbb{R}^{N}$}

\author{
Rachidi Bolaji Salako and Wenxian Shen  \\
Department of Mathematics and Statistics\\
Auburn University\\
Auburn University, AL 36849\\
U.S.A. }

\date{}
\maketitle

\begin{abstract}
In the current paper, we consider the following parabolic-elliptic semilinear Keller-Segel model on $\mathbb{R}^{N}$,$$
\begin{cases}u_{t}=\nabla\cdot(\nabla{u}-\chi u\nabla{v})+au-bu^{2}, \quad x\in\mathbb{R}^N,\,\,{t}>0,\\{0}=(\Delta-I)v+u,\quad x\in\mathbb{R}^N,\,\, t>0,\end{cases}$$where $\chi>0,\ a\geq0,\  b>0$ are constant real numbers and $N$ is a positive integer.  We first prove the local existence and uniqueness of classical solutions $(u(x,t;u_0),v(x,t;u_0))$ with $u(x,0;u_0)=u_0(x)$ for various initial functions $u_0(x)$. Next, under some conditions on the constants $a, b, \chi$ and the dimension $N$, we prove the global existence and boundedness of classical solutions $(u(x,t;u_0),v(x,t;u_0))$ for given initial functions $u_0(x)$.  Finally, we investigate the asymptotic behavior of the global solutions with strictly positive initial functions or nonnegative compactly supported initial functions. Under some conditions on the constants $a, b, \chi$ and the dimension $N$, we show that for every strictly positive initial function $u_0(\cdot)$,$$\lim_{t\to\infty}\sup_{x\in\mathbb{R}^N}\big[|u(x,t;u_0)-\frac{a}{b}|+|v(x,t;u_0)-\frac{a}{b}|\big]=0,$$ and that for every nonnegative initial function $u_0(\cdot)$ with non-empty and compact support ${\rm{supp}}(u_0)$, there are $0<c_{\rm{low}}^*(u_0)\leq{c}_{\rm{up}}^*(u_0)<\infty$ such that $$\lim_{t\to\infty}\sup_{|x|\leq{ct}}\big[|u(x,t;u_0)-\frac{a}{b}|+|v(x,t;u_0)-\frac{a}{b}|\big]=0\quad \forall\,\,{0}<c<c_{\rm{low}}^*(u_0)
 $${and}$$\lim_{t\to\infty}\sup_{|x|\geq{ct}}\big[u(x,t;u_0)+v(x,t;u_0)\big]=0\quad\forall\,\,{c}>c_{\rm{up}}^*(u_0).$$
\end{abstract}



\section{Introduction and the Statements of Main results}
The movements of many mobile species are influenced by certain chemical substances. Such movements are referred to
chemotaxis.
The origin of  chemotaxis models was introduced by Keller and Segel  (see \cite{KeSe1}, \cite{KeSe2}). The following is a general Keller-Segel model
for the time evolution of both the density $u(x,t)$ of a mobile species and the density $v(x,t)$ of a  chemoattractant,
\begin{equation}\label{IntroEq0}
\begin{cases}
u_{t}=\nabla\cdot (m(u)\nabla u- \chi(u,v)\nabla v) + f(u,v),\quad   x\in\Omega,\,\,\, t>0 \\
\tau v_t=\Delta v + g(u,v),\quad  x\in\Omega,\, \,\, t>0
\end{cases}
\end{equation}
complemented with certain boundary condition on $\p\Omega$ if $\Omega$ is bounded, where $\Omega\subset \R^N$ is an open domain, $\tau\ge 0$ is a non-negative constant linked to the speed of diffusion of the chemical, the function $\chi(u,v)$ represents  the sensitivity with respect to chemotaxis,  and the functions $f$ and $g$ model the growth of the mobile species and the chemoattractant, respectively.

In the last two decades, considerable progress has been made in  the analysis of various particular cases of (\ref{IntroEq0}) on both bounded and unbounded domains. Among the central problems are the existence  of nonnegative solutions of (\ref{IntroEq0}) which are globally defined in time or blow up at a finite time and  the asymptotic  behavior of time global solutions. The features of solutions of (\ref{IntroEq0}) depend on the geometric properties of the functions $m(u)$,  $\chi(u,v)$,
$f(u,v)$, and $g(u,v)$.

When $\tau>0$, \eqref{IntroEq0} is referred to as the parabolic-parabolic semilinear Keller-Segel model. In this case, when \eqref{IntroEq0} is coupled with Neumann boundary condition on bounded domain, several results have been established for different choices of the functions $m(u)$,  $\chi(u,v)$,
$f(u,v)$, and $g(u,v)$. For example when $\tau=1$, $m(u)=1$, $\chi(u,v)=\chi u$, $g(u,v)=u-v$,  $f(u,v)=u(a-bu)$,  and $\frac{b}{\chi}$ is sufficiently large,  it is shown in \cite{win_JDE2014} that  unique global classical solution exists for  every nonnegative initial data $(u_{0},v_0)\in C^{0}(\overline{\Omega})\times W^{1,\infty}(\Omega)$ and that the constant solution $(\frac{a}{b},\frac{a}{b})$ is asymptotically stable. See also \cite{OTYM2002}, \cite{win_CPDE2010} for the study of boundedness and global existence
of classical solutions when $b$ is large. When $b$ is small, among others, Lankeit in \cite{lankeit_eventual} proved the existence of at least one global weak solution with given initial functions.
   The reader is referred to \cite{BBTW} for a recent survey.

In the present paper, we restrict ourselves to  the case that $\tau=0$, which is supposed to model the situation when the chemoattractant diffuses very quickly. System \eqref{IntroEq0} with $\tau=0$ reads as
\begin{equation}\label{IntroEq0-1}
\begin{cases}
u_{t}=\nabla\cdot (m(u)\nabla u- \chi (u,v)\nabla v) + f(u,v),\quad  x\in\Omega,\,\,\, t>0 \\
0=\Delta v + g(u,v),\quad  x\in\Omega,\, \,\, t>0
\end{cases}
\end{equation}
complemented with certain boundary condition on $\p\Omega$ if $\Omega$ is bounded.

Global existence and asymptotic behavior of solutions of \eqref{IntroEq0-1} on bounded domain $\Omega$ complemented with
Neumann boundary conditions,
\begin{equation}
\label{neumann-cond}
\frac{\p u}{\p n}=\frac{\p v}{\p n}=0\ \,\, \text{for}\ \,\, x\in\p \Omega,
\end{equation}
has been studied in many papers. For example, in \cite{TeWi}, the authors studied \eqref{IntroEq0-1}+\eqref{neumann-cond} with
 $m(u)\equiv 1$, $\chi(u,v)=\chi u$, $f(u,v)=au-bu^2$, which is referred to as the logistic source in literature,
 and $g(u,v)=u-v$, where $\chi$, $a$, and $b$ are positive constants. Among others,  the following are proved in \cite{TeWi},

 \medskip

 \noindent $\bullet$ If either $N\le 2$ or $b>\frac{N-2}{N}\chi$, then for any initial data $u_0\in C^{0,\alpha}(\bar\Omega)$ ($\alpha\in (0,1)$) with $u_0(x)\ge 0$, \eqref{IntroEq0-1}+\eqref{neumann-cond} possesses a  unique bounded global classical solution $(u(x,t;u_0),v(x,t;u_0))$ with
 $u(x,0;u_0)=u_0(x)$.

 \smallskip

 \noindent $\bullet$ If $b>2\chi$, then for any  $u_0\in C^{0,\alpha}(\bar\Omega)$ with $u_0(x)\ge 0$ and $u_0(x)\not \equiv 0$,
 $$
 \lim_{t\to\infty} \big[ \|u(\cdot;t;u_0)-\frac{a}{b}\|_{L^\infty(\Omega)}+\|v(\cdot,t;u_0)-\frac{a}{b}\|_{L^\infty(\Omega)}\big]=0.
 $$

 \noindent It should be pointed out  that, for the above choices of $m(u)$, $\chi(u,v)$, $f(u,v)$, and $g(u,v)$,  when $N\ge 3$ and $b\le \frac{N-2}{N}\chi$,  it  remains open whether for any
 given initial data $u_0\in C^{0,\alpha}(\bar\Omega)$,  \eqref{IntroEq0-1}+\eqref{neumann-cond}
 possesses a global classical solution $(u(x,t;u_0),v(x,t;u_0))$ with
 $u(x,0;u_0)=u_0(x)$, or whether finite-time blow-up occurs for some initial data.  The works \cite{lankeit_exceed},  \cite{win_JMAA_veryweak},  \cite{win_JNLS} should be mentioned along this direction.
 It is shown in \cite{lankeit_exceed}, \cite{win_JNLS} that in presence of suitably weak logistic dampening
(that is, small $b$)  certain transient growth phenomena do occur for some
initial data. It is shown in \cite{win_JMAA_veryweak} that if we keep the choices of $m(u)$ and $\chi(u,v)$ as above and let
 $f(u,v)=au-bu^\kappa$ with suitable
$\kappa<2$ (for instance, $\kappa=3/2$) and $g(u,v)=u-\frac{1}{|\Omega|}\int_\Omega u(x)dx$, then  finite-time blow-up is possible.

The reader is referred to \cite{BBTW}, \cite{DiNa}, \cite{GaSaTe}, \cite{WaMuZh}, \cite{Win}, \cite{win_jde}, \cite{win_JMAA_veryweak}, \cite{win_arxiv}, \cite{win_JNLS}, \cite{YoYo}, \cite{ZhMuHuTi},  and references
 therein for other studies of \eqref{IntroEq0-1} on bounded domain with Neumann or Dirichlet  boundary conditions and with $f(u,v)$ being logistic type source function or $0$ and  $m(u)$, $\chi(u,v)$, and $g(u,v)$ being  various kinds of functions.

 There are also several studies of \eqref{IntroEq0-1} when $\Omega$ is the whole space $\R^N$ and $f(u,v)=0$ (see \cite{DiNaRa},\cite{KKAS} \cite{NAGAI_SENBA_YOSHIDA}, \cite{SuKu}, \cite{Sug}). For example,
in the  case of   $m(u)\equiv 1$, $\chi(u,v)=\chi u$,  $f(u,v)=0$, and $g(u,v)=u-v$, where  $\chi$ is a positive constant, it is known that blow-up occurs if
either N=2 and the total initial population mass is large enough, or $N\ge 3$ (see  \cite{BBTW}, \cite{DiNaRa},  \cite{NAGAI_SENBA_YOSHIDA} and  references therein).  However, there is little study of \eqref{IntroEq0-1} when $\Omega=\R^N$
and $f(u,v)\not =0$.

The objective of this paper is to investigate the local/global existence and asymptotic behavior of positive solutions of \eqref{IntroEq0-1} when
$\Omega=\R^N$ and $f(u,v)=au-bu^2$ is a logistic  source function, where $a$ and $b$ are positive constants. We further restrict ourselves to the choices
 $m(u)\equiv 1$, $\chi(u,v)=\chi u$,
 and $g(u,v)=u-v$, where $\chi$ is  positive constant. System \eqref{IntroEq0-1} with these choices on $\R^N$ reads as
\begin{equation}\label{IntroEq1}
\begin{cases}
u_{t}=\Delta u - \nabla\cdot  (\chi u\nabla v) + u(a-bu),\quad  x\in\R^N,\,\,\, t>0 \\
0=\Delta v + u-v,\quad x\in\R^N,\,\,\, t>0.
\end{cases}
\end{equation}

We first investigate the local existence of solutions of \eqref{IntroEq1} for various given initial functions $u_0(x)$. Note that, due to biological interpretations, only nonnegative initial functions will be of interest. We call $(u(x,t),v(x,t))$ a {\it  classical solution} of \eqref{IntroEq1} on
$ [0,T)$ if  $u,v\in C(\R^N\times [0,T))\cap  C^{2,1}(\R^N\times (0,T))$ and satisfies \eqref{IntroEq1} for
$(x,t)\in\R^N\times (0,T)$ in the classical sense. A classical solution $(u(x,t),v(x,t))$  of \eqref{IntroEq1} on
$ [0,T)$ is called {\it nonnegative} if $u(x,t)\ge 0$ and $v(x,t)\ge 0$ for all $(x,t)\in\R^N\times [0,T)$.  A {\it global classical solution}   of \eqref{IntroEq1} is a classical solution  on
$ [0,\infty)$.

Let
\begin{equation}
\label{unif-cont-space}
C_{\rm unif}^b(\R^N)=\{u\in C(\R^N)\,|\, u(x)\,\,\text{is uniformly continuous in}\,\, x\in\R^N\,\, {\rm and}\,\, \sup_{x\in\R^N}|u(x)|<\infty\}
\end{equation}
equipped with the norm $\|u\|_\infty=\sup_{x\in\R^N}|u(x)|$.  For given $0<\nu<1$ and $0<\theta<1$, let
\begin{equation}
\label{holder-cont-space}
C^{\nu}_{\rm unif}(\R^N)=\{u\in C_{\rm unif}^b(\R^N)\,|\, \sup_{x,y\in\R^N,x\not =y}\frac{|u(x)-u(y)|}{|x-y|^\nu}<\infty\}
\end{equation}
equipped with the norm $\|u\|_{\infty,\nu}=\sup_{x\in\R^N}|u(x)|+\sup_{x,y\in\R^N,x\not =y}\frac{|u(x)-u(y)|}{|x-y|^\nu}$, and
\begin{align*}
&C^{\theta}((t_1,t_2),C_{\rm unif}^\nu(\R^N))\\
&=\{ u(\cdot)\in C((t_1,t_2),C_{\rm unif}^{\nu}(\R^N))\,|\, u(t)\,\, \text{is locally H\"{o}lder continuous with exponent}\,\, \theta\}.
\end{align*}

We have the following result on the local existence and uniqueness of solution of \eqref{IntroEq1} for initial data belonging to $C^{b}_{\rm unif}(\R^N)$.

\begin{tm} \label{Local existence1}
For any $u_0 \in C_{\rm unif}^{b}(\R^N)$ with  $u_0 \geq 0$,  there exists $T_{\max}^\infty(u_0) \in (0,\infty]$  such that \eqref{IntroEq1}  has a unique non-negative classical solution $(u(x,t;u_0),v(x,t;u_0))$ on $[0,T_{\max}^\infty(u_0))$ satisfying that $\lim_{t\to 0^{+}}u(\cdot,t;u_0)=u_0$ in the
$C_{\rm unif}^b(\R^N)$-norm,
\begin{equation}
\label{local-1-eq1}
u(\cdot,\cdot;u_0) \in C([0, T_{\max}^\infty(u_0) ), C_{\rm unif}^b(\R^N) )\cap C^1((0,T_{\max}^\infty(u_0)),C_{\rm unif}^b(\R^N))
\end{equation}
and
\begin{equation}
\label{local-1-eq2}
u(\cdot,\cdot;u_0),\,\, \partial_{x_i} u(\cdot,\cdot),\,\, \partial^2_{x_i x_j} u(\cdot,\cdot),\,\, \partial_t u(\cdot,\cdot;u_0)\in C^\theta((0,T_{\max}^\infty(u_0)),C^\nu_{\rm unif}(\R^N))
\end{equation}
for all $i,j=1,2,\cdots,N$, $0<\theta\ll 1$, and  $0<\nu\ll 1$.
Moreover, if $T_{\max}^\infty(u_0)< \infty,$ then
$\limsup_{t \to T_{\max}^\infty(u_0)}   \left\| u(\cdot,t;u_0) \right\|_{\infty}  =\infty.$
\end{tm}

For given $p\ge 1$ and $\alpha \in(0,\ 1)$,  let $X^\alpha$ be the fractional power space of $I-\Delta$ on $X=L^p(\R^N)$. We
obtain the following results on the local existence and uniqueness of solutions of \eqref{IntroEq1} for $u_{0}\in X^{\alpha}$.

 \begin{tm} \label{Local existence2}
 Assume that $p>N$ and $\alpha\in(\frac{1}{2},1)$.
 For every nonnegative $u_{0}\in X^{\alpha}$,  there is a positive number $T_{\max}^\alpha(u_0)\in (0,\infty]$ such that \eqref{IntroEq1} has a unique nonnegative classical solution $(u(x,t;u_0),v(x,t;u_0))$ on $ \mathbb{R}^{N}\times [0,T_{\max}^\alpha(u_0))$ satisfying that
 $\lim_{t\to 0+}u(\cdot,t;u_0)=u_0$ in the $X^\alpha$-norm,
 \begin{equation}
 \label{local-2-eq1}
 u(\cdot,\cdot;u_0) \in C([0, T_{\max}^\alpha(u_0) ), X^\alpha) \cap C^1((0,T_{\max}^\alpha(u_0)),L^p(\R^N)),
  \end{equation}

  \begin{equation}
  \label{local-2-eq2}
 u(\cdot,\cdot;u_0)\in C((0,T_{\max}^\alpha(u_0)),X^\beta)\cap  C^1((0,T_{\max}^\alpha(u_0)),C_{\rm unif}^b(\R^N)),
 \end{equation}
 and
 \begin{equation}
 \label{local-2-eq3}
u(\cdot,\cdot;u_0),\,\,  \partial_{x_i} u(\cdot,\cdot;u_0),\,\, \partial^2_{x_ix_j} u(\cdot,\cdot;u_0),\,\, \partial_t u(\cdot,\cdot;u_0)\in C^\theta ((0,T_{\max}^\alpha(u_0)),C^\nu_{\rm unif}(\R^N))
 \end{equation}
 for all $0\le \beta<1$, $i,j=1,2,\cdots$, $0<\theta\ll 1$, and $0<\nu\ll 1$.
 Moreover, if   $T_{\max}^\alpha(u_0) <+\infty$, then
 $\lim_{t\to T_{\max}^\alpha(u_0))}\|u(\cdot,t;u_0)\|_{X^{\alpha}} =\infty.$
\end{tm}

Since  $X^{\alpha}\subset C^{b}_{\rm unif}(\R^N)$ for $p>N$ and $\alpha\in(\frac{1}{2},1)$, the existence of local classical solution for initial data in $X^{\alpha}$ is guaranteed by Theorem \ref{Local existence1}. However $u(\cdot,\cdot;u_0) \in C([0, T_{\max}^\alpha(u_0) ), X^\alpha)\cap C^{1}((0,T_{\max}^\alpha(u_0)),L^p(\R^N))\cap C((0,T_{\max}^\alpha(u_0)),X^\beta) $  in Theorem \ref{Local existence2}, which is very important for later use, is  not included in Theorem \ref{Local existence1}. The proof of Theorem \ref{Local existence1}  is based on the contraction mapping theorem and a technical result proved in Lemma \ref{L_Infty bound 2}, while the proof of Theorem \ref{Local existence2} is based on
 the semigroup method. Theorem \ref{Local existence2} is of particular interest because it helps to take advantage of the integration by parts theorems, thus, helps to get a weaker condition on the parameters $\chi, b$  and $N$ to ensure the global existence of classical solutions (see Theorem \ref{Main Theorem 1}). Furthermore, Theorem \ref{Local existence2} will be used to get some extension results for $L^{p}-$ integrable initial data, which are not necessarily continuous, as stated in the next theorem and  Theorem \ref{Main Theorem 2}.

 Since functions of $L^{p}(\R^N)$ are not always continuous, the definition of solution to \eqref{IntroEq1} should be modified. For a nonnegative initial data $u_{0}\in L^{p}(\R^N)$, by a  solution of \eqref{IntroEq1} on $[0,T)$ with initial data $u_0$ we mean nonnegative functions $u(x,t)$, $v(x,t)$
 satisfying that $u(\cdot,\cdot), v(\cdot,\cdot)\in  C^{2,1}(\R^N\times(0, T))$, \eqref{IntroEq1} holds for $(x,t)\in\R^N\times (0,T)$ in classical sense, and $\lim_{t\to 0^+}u(\cdot,t)=u_0(\cdot)$ in the $L^p(\R^N)$-norm.

\begin{tm}\label{Local Existence3}
For every $p>N\ with \ p\geq 2$ and $u_{0}\in L^{p}(\R^N)$ with  $u_{0}\geq 0$, there is a positive number $T_{\max}^p(u_0)\in (0,\infty]$ such that \eqref{IntroEq1} has a unique non-negative solution $(u(x,t;u_{0}),v(x,t;u_{0}))$ on $[0,T_{\max}^p(u_0))$ satisfying that $\lim_{t\rightarrow 0^{+}}u(\cdot,t;u_{0})=u_{0}(\cdot)$ in the $L^{p}(\R^N)$-norm,
\begin{equation}
\label{local-3-eq1}
  u(\cdot,\cdot;u_0)\in C([0, T_{\max}^p(u_0) ), L^p(\R^N)) \cap C^1((0,T_{\max}^p(u_0)),L^p(\R^N)),
  \end{equation}
  \begin{equation}
  \label{local-3-eq2}
  u(\cdot,\cdot;u_0)\in C((0,T_{\max}^p(u_0)),X^\beta)\cap C^1((0,T_{\max}^p(u_0)),C_{\rm unif}^b(\R^N)),
  \end{equation}
  and
  \begin{equation}
  \label{local-3-eq3}
  u(\cdot,\cdot;u_0),\,\, \partial_{x_i} u(\cdot,\cdot),\,\,\partial^2_{x_ix_j} u(\cdot,\cdot;u_0),\,\,  \partial_t u(\cdot,\cdot;u_0)\in C^\theta ((0,T_{\max}^p(u_0)),C^\nu_{\rm unif}(\R^N))
  \end{equation}
 for all $0\le \beta<1$, $i,j=1,2,\cdots,N$, $0<\theta\ll 1$, and $0<\nu\ll 1$.
  Moreover,
  if $T_{\max}^p(u_0) <+\infty$, then
   $\lim_{t\to T_{\max}^p(u_0))}\|u(\cdot,t;u_0)\|_{L^{p}(\R^N)} =\infty.$
\end{tm}

The following theorem establishes the equality of the maximal existence intervals of the classical solutions of \eqref{IntroEq1}
in different phase spaces.

\begin{tm}
\label{Lp-Bound1}
Assume that $p>N$, $\alpha\in(\frac{1}{2},1)$, and  $X^\alpha$ is the fractional power space of $I-\Delta$ on $X=L^p(\R^N)$.
Then the following hold,
\begin{itemize}
\item[(1)] if $u_0\in X^\alpha$, then $T_{\max}^\infty(u_0)=T_{\max}^\alpha(u_0)$;

\item[(2)] if $u_0\in X^\alpha$ and $p\ge 2$, then $T_{\max}^\alpha(u_0)=T_{\max}^p(u_0)$;

\item[(3)] if $p\ge 2$ and $u_0\in C_{\rm unif}^b(\R^n)\cap L^p(\R^n)$, then $T_{\max}^\infty(u_0)=T_{\max}^p(u_0)$;

\item[(4)] if  $p\ge 2$ and $u_{0}\in L^p(\R^N)\cap L^{1}(\R^N)$, then  $u(\cdot,\cdot,u_{0})\in C([0,T_{\max}^p(u_0)),L^{1}(\R^N))$,
\end{itemize}
where $T_{\max}^\infty(u_0)$, $T_{\max}^\alpha(u_0)$, and $T_{\max}^p(u_0)$ are as in Theorems \ref{Local existence1},
\ref{Local existence2}, and \ref{Local Existence3}, respectively.
\end{tm}

Next, we study the global existence of classical solutions of \eqref{IntroEq1}.
The following are the main  results on the global existence.

\begin{tm} \label{Main Theorem 0}
Suppose that $\chi\leq b.$ Then for every nonnegative $u_{0} \in C_{\rm unif}^b(\R^{N})$,  \eqref{IntroEq1} has a unique global  classical solution $(u(x,t;u_0),v(x,t;u_0))$ satisfying that $u(\cdot,\cdot;u_0)\in C([0,\infty),C_{\rm unif}^b(\R^N))$. Furthermore if $\chi<b,$ the solution is globally bounded.
\end{tm}

\begin{tm}\label{Main Theorem 1}
Suppose that $u_{0}\in L^{1}(\R^N)$ satisfies the hypothesis of Theorem \ref{Local existence2} and that
\begin{equation}\label{Eq1_Main Theorem}
\frac{N}{2}<\frac{\chi}{(\chi-b)_{+}}.
\end{equation}
Then \eqref{IntroEq1} has a unique global  classical solution  $(u(x,t;u_0),v(x,t;u_0))$ satisfying that
$u(\cdot,\cdot;u_0)\in C([0,\infty), X^\alpha)$. Furthermore, it holds that
 \begin{equation}\label{Eq2_Main Theorem}
\|u(t)\|_{L^{\infty}}\leq C_{1}t^{-\frac{N}{2p}}e^{-t}\|u_{0}\|_{L^{p}} + C_{2} \left[ \|u_{0}\|_{L^{p}} + {K}_{p}^{\frac{1}{p}}\|u_{0}\|_{1}^{\frac{\tilde{\tilde{\lambda}}_{p}}{p}}\|u_{0}\|_{p}^{\frac{\tilde{\lambda}_{p}}{p}}t^{\frac{1}{p}}e^{({\lambda}_{p}-1)at} \right]e^{at}
\end{equation}
for any $t>0$,
where ${K}_{p}, \lambda_{p} ,\tilde{\lambda}_{p}, \tilde{\tilde{\lambda}},\ C_{1}$ and $C_{2}$ are positive constants depending on $N ,\ p,\ a,\ b,$ and $\chi.$
\end{tm}

\begin{tm}\label{Main Theorem 2} Let $N$ be a positive integer and $p$ be a positive real number with $p>N$ and $p\geq 2$. Suppose that \eqref{Eq1_Main Theorem} holds. Then for every nonnegative initial data $u_{0}\in L^{1}(\R^{N})\cap L^{p}(\R^{N})$, \eqref{IntroEq1} has a unique  global classical  solution $(u(x,t;u_0),v(x,t;u_0))$  satisfying that $u(\cdot,\cdot;u_0)\in C([0,\infty),L^p(\R^N))\cap C([0,\infty),L^1(\R^N))$ and  that \eqref{Eq2_Main Theorem} holds.
\end{tm}

We point out that if the domain is bounded, it allows us to take advantage of the fact that the domain has finite measure to obtain that the solution is globally bounded. However, in the present case where domain has infinite size, no such trick can be used. This makes the study of this problem on unbounded domain more complicated. We also point out that the global solution in Theorem \ref{Main Theorem 1} or Theorem \ref{Main Theorem 2} may not be bounded in $L^p(\R^N)$-norm (see the remarks after Theorem \ref{Main Theorem 4}).
We remark that Theorem  \ref{Main Theorem 0} requires less assumption on the initial data and more assumption on the parameters ($\chi,  b, N$) while Theorem \ref{Main Theorem 1} requires more assumption on the initial data and less assumption on the parameters. Theorem \ref{Main Theorem 2} generalizes the known results when \eqref{IntroEq1} is studied on bounded domain with Neumann boundary condition.   In the case of bounded domains, under \eqref{Eq1_Main Theorem}, it follows from our results that \eqref{IntroEq1}  has a unique global classical solution with given initial function $u_{0}\in L^{p}(\Omega)$. It is obvious that $C^{0}(\overline{\Omega})\subset L^{p}({\Omega})$ for every $p\geq 1$ when $\Omega$ is bounded;  then our results cover initial data in $C^{0}(\overline{\Omega})$.

When \eqref{Eq1_Main Theorem} does not hold, we do not know yet if \eqref{IntroEq1}  has a global classical solution for given initial function $u_0$ as in Theorem \ref{Main Theorem 1} or \ref{Main Theorem 2}. As it is mentioned in the above, this problem is still open for bounded domain case as well. For the whole space case, when $\frac{N}{2}<\frac{\chi}{(\chi-b)_{+}}<\infty$ (hence $b<\chi$),  it also remains open whether \eqref{IntroEq1}  has a global classical solution for given initial function $u_0$ as in Theorem \ref{Main Theorem 0}.

Finally, we  explore the asymptotic behavior of global classical solutions of \eqref{IntroEq1} and obtain the following main results.

\begin{tm}\label{Main Theorem 3}
Suppose that $u_{0}\in C_{\rm unif}^{b}(\R^N)$ with $\inf_{x\in\R^N}u_{0}(x)>0.$ If
\begin{equation}
  \quad b>2\chi,
\end{equation}
then the   unique global classical solution $(u(x,t;u_0),v(x,t;u_0))$ of \eqref{IntroEq1} with $u(x,0;u_0)=u_0(x)$  satisfies that
\begin{equation}
\|u(\cdot,t;u_0)-\frac{a}{b}\|_{\infty} + \|v(\cdot,t;u_0)-\frac{a}{b}\|_{\infty}\rightarrow 0 \ \text{as} \ t\rightarrow \infty.
\end{equation}
\end{tm}

\medskip

\begin{tm}\label{Main Theorem 4}
 Assume that
\begin{equation}\label{A_Eqtm2}
\chi< \frac{2b}{3+\sqrt{aN+1}}.
\end{equation}
\begin{itemize}
\item[(1)]
Suppose that $u_{0}\in C_{\rm unif}^b(\R^N)$  is  nonnegative and ${\rm supp}(u_0)$ is non-empty.
 There is $c_{\rm low}^*(u_0)>0$ such that  the   unique global classical solution $(u(x,t;u_0)$, $v(x,t;u_0))$ of \eqref{IntroEq1}  satisfies that
\begin{equation}\label{A_Eq8}
\lim_{t\rightarrow\infty}\left[\sup_{|x|\leq ct}|u(x,t;u_0)-\frac{a}{b}|+\sup_{|x|\leq ct}|v(x,t;u_0)-\frac{a}{b}|\right]=0\quad \forall\,\, 0\leq c < c_{\rm low}^{\ast}.
\end{equation}

\item[(2)]
Suppose that $u_{0}\in C_{\rm unif}^b(\R^N)$  is  nonnegative and ${\rm supp}(u_0)$ is non-empty and compact. There is $c_{\rm up}^*(u_0)\ge c_{\rm low}^*(u_0)$ such that  the   unique global classical solution $(u(x,t;u_0)$, $v(x,t;u_0))$ of \eqref{IntroEq1}  satisfies that
\begin{equation}\label{A_Eq9}
\lim_{t\rightarrow\infty}\left[\sup_{|x|\geq ct}u(x,t;u_0)+ \sup_{|x|\geq ct}v(x,t;u_0)\right]=0\quad \forall\,\, c > c_{\rm up}^{\ast}.
\end{equation}
\end{itemize}
\end{tm}

We remark that $(\frac{a}{b},\frac{a}{b})$ is usually called a trivial equilibrium solution of \eqref{IntroEq1}. By Theorem \ref{Main Theorem 3}, when the logistic damping constant $b$ is large, the trivial equilibrium $(\frac{a}{b},\frac{a}{b})$ is globally stable with strictly positive initial data.
As mentioned in the above, for \eqref{IntroEq0} in the bounded domain with
$m(u)=1$, $\chi(u,v)=\chi u$, $f(u,v)=au-bu^2$, and $g(u,v)=u-v$, the global stability of the trivial solution $(\frac{a}{b},\frac{a}{b})$ has been obtained in
\cite{win_JDE2014} when $\tau>0$,   and  in \cite{TeWi} when  $\tau=0$.
 It is worthwhile mentioning that, when $b$ is not large, there may be lots of nontrivial equilibria - at least in bounded domains, quite a few have been detected (see \cite{kuto_PHYSD}, \cite{TeWi}).

\smallskip

We also remark that it is not required that ${\rm supp}(u_0)$ is compact in Theorem \ref{Main Theorem 4}(1). Hence it applies to nonnegative $u_0$
in Theorems \ref{Main Theorem 1} and \ref{Main Theorem 2}. Then by \eqref{A_Eq8},  if $p>N$,\ $p\geq 2$,\ $\chi<\frac{2b}{3+\sqrt{1+Na}}$ and  $u_{0}\in L^{1}(\R^N)\cap L^{p}(\R^N)\setminus\{0\}$, then
\begin{equation}
\label{remark-on-boundedness}
\lim_{t\to \infty}\|u(\cdot,t;u_{0})\|_{L^{p}(\R^N)}=\infty.
\end{equation}

The limit properties in \eqref{A_Eq8} and \eqref{A_Eq9} reflect the spreading feature of the mobile species. In the absence of the chemotaxis
(i.e. $\chi=0$), the first equation in  \eqref{IntroEq1} becomes the following scalar reaction diffusion equation,
\begin{equation}
\label{fisher-eq}
u_{t}=\Delta u + u(a-bu),\quad  x\in\R^N,\,\, t>0,
\end{equation}
which is referred to as Fisher or KPP equations due to  the pioneering works by Fisher (\cite{Fis}) and Kolmogorov, Petrowsky, Piscunov
(\cite{KPP}) on the spreading properties of \eqref{fisher-eq}.
It follows from the works \cite{Fis}, \cite{KPP}, and \cite{Wei1}  that $c^*_{\rm low}(u_0)$ and $c^*_{\rm up}(u_0)$  in Theorem  \ref{Main Theorem 4} can be chosen so that $c^{\ast}_{\rm low}(u_0)=c^{\ast}_{\rm up}(u_0)=2\sqrt a$ for any nonnegative $u_0\in C_{\rm unif}^b(\R^N)$ with ${\rm supp}(u_0)$ being not empty and compact
 ($c^*:=2\sqrt a$ is called the {\it spatial spreading speed} of \eqref{fisher-eq} in literature), and that \eqref{fisher-eq} has traveling wave solutions $u(t,x)=\phi(x-ct)$  connecting $\frac{a}{b}$ and $0$ (i.e.
$(\phi(-\infty)=\frac{a}{b},\phi(\infty)=0)$) for all speeds $c\geq c^*$ and has no such traveling wave
solutions of slower speed.
 Since the pioneering works by  Fisher \cite{Fis} and Kolmogorov, Petrowsky,
Piscunov \cite{KPP},  a huge amount research has been carried out toward the spreading properties of reaction diffusion equations of the form,
\begin{equation}
\label{general-fisher-eq}
u_t=\Delta u+u f(t,x,u),\quad x\in\R^N,
\end{equation}
where $f(t,x,u)<0$ for $u\gg 1$,  $\p_u f(t,x,u)<0$ for $u\ge 0$ (see \cite{Berestycki1, BeHaNa1, BeHaNa2, Henri1, Fre, FrGa, LiZh, LiZh1, Nad, NoRuXi, NoXi1, She1, She2, Wei1, Wei2, Zla}, etc.).

  When $\chi> 0$, up to our best knowledge, the spreading properties of \eqref{IntroEq1} is studied for the first time in this paper. It remains open whether $c^*_{\rm low}(u_0)$ and $c^*_{\rm up}(u_0)$ in \eqref{A_Eq8} and \eqref{A_Eq9} can be chosen so that $c^{\ast}_{\rm low}(u_0)$ and $c^{\ast}_{\rm up}(u_0)$ are independent of
 $u_0$; whether  $c_{\rm low}^*(u_0)=c^*_{\rm up}(u_0)$; and what is the relation between $c^*_{\rm low}(u_0)$, $c^*_{\rm up}(u_0)$ and $2\sqrt a$. These questions are very important in the understanding of the spreading feature of \eqref{IntroEq1} because they are related to the issue whether the chemotaxis speeds up or slows down the spreading of the species. We plan to study these problems in our future works. Another interesting question about \eqref{IntroEq1} is the existence of traveling wave solutions connecting $(\frac{a}{b},\frac{a}{b})$ and $(0,0)$. We also plan to study the existence of such solutions in our future works.

\medskip

The rest of the paper is organized as follows. In section 2, we collect some important results from literature that will be needed in the proofs
of our main results. Section 3 is devoted to the proofs of the local existence theorems (i.e. Theorems \ref{Local existence1} to \ref{Lp-Bound1}). In section 4,
we  prove the global existence  theorems (i.e. Theorems \ref{Main Theorem 0} to \ref{Main Theorem 2}). Finally in section 6, we present the asymptotic behavior of classical solutions and prove Theorems \ref{Main Theorem 3} and \ref{Main Theorem 4}.

\medskip

\noindent {\bf Acknowledgment.} The authors would like to thank Professors J. Ignacio Tello and Michael Winkler
 for valuable discussions, suggestions,  and references. %The authors also would like to thank the referee for valuable comments and suggestions which improved the presentation of this paper considerably.

\section{Preliminaries}

In this section, we shall prepare several lemmas which will be used often in the next sections. We start by stating some standard definitions from semigroup theory. The reader is referred to \cite{Dan Henry}, \cite{A. Pazy} for the details.

Let X be a Banach space and $\{T(t)\}_{t\geq 0}$ be a $C_{0}-$semigroup on $X$ generated by $A.$ It is well known that $A$ is closed and densely  defined linear operator on $X.$ Furthermore, there are constants
 $M>0$ and $\omega\in \mathbb{R}$ such that $\|T(t)\|\leq Me^{t\omega}$ for every $t\geq 0$ and $(\omega, \infty)\subset \rho(A)$ with
$$
\|(\lambda I-A)^{-1}\|\leq \frac{M}{\lambda-\omega}\ \ \forall\ \lambda>\omega,
$$
where $\rho(A)$ denotes the resolvent set of $A.$
Moreover for every $t>0$ and every continuous function $w \in C( [0 , t] : X) $, the map
\begin{equation} \label{cont-func}
[0, t]\ni s \mapsto T(t-s)w(s)\in X
\end{equation}
is continuous.

For our purpose, we will be concerned with the spaces $C_{\rm unif}^b(\R^{N})$ and $L^{p}(\mathbb{R}^{N})$ for $1\leq p<\infty$,  and the analytic semigroup $T(t)$ generated by  $A=\Delta-I$ on $X=C_{\rm unif}^b(\R^{N})$ or $X=L^{p}(\mathbb{R}^{N})$.  Observe  that
\begin{equation}
\label{semigroup-eq}
(T(t)u)(x)=e^{-t}(G(\cdot,t)\ast u)(x)= \int_{\R^{N}}e^{-t}G(x-y,t)u(y)dy
\end{equation}
for every $u\in X$, $t\ge 0$,  and $x\in\R^N$, where $X=C_{\rm unif}^b(\R^{N})$ or $X=L^{p}(\mathbb{R}^{N})$,  and $G(x,t)$ is the heat kernel defined by
\begin{equation}
\label{heat-kernel}
G(x,t)=\frac{1}{(4\pi t)^{\frac{N}{2}}}e^{-\frac{|x|^{2}}{4t}}.
\end{equation}

Let $X=C_{\rm unif}^b(\R^{N})$ or $X=L^{p}(\mathbb{R}^{N})$ and
 $X^{\alpha}={\rm Dom}((I-\Delta)^{\alpha})$ be the fractional power spaces of $I-\Delta$ on $X$ ($\alpha\in [0,\infty)$).
 Note that $X^0=X$ and $X^1={\rm Dom}(I-\Delta)$.
It is well known that $\Delta$  generates a contraction $C_{0}-$semigroup   defined by the heat kernel, $\{G(t)\}_{t\geq 0},$ on $X$ with spectrum $\sigma(\Delta)=(-\infty, 0]$ (see \cite{Dan Henry}). Thus, the Hille-Yosida theorem implies that the resolvent operator $R(\lambda)$ associated with $\Delta$ is the Laplace transform of $\{G(\cdot,t)\}_{t}.$ Thus the operator $\Delta -I$ is invertible  with
\begin{equation}\label{Eq_Inv}
(I-\Delta)^{-1}u=\int_{0}^{\infty}e^{-t}G(\cdot,t)\ast udt=\int_{0}^{\infty}T(t)udt
\end{equation}
for all $u\in X$. Furthermore the restriction operator $(\Delta-I)^{-1}|_{X^{\alpha}}: X^{\alpha} \rightarrow X^{\alpha}$ is a bounded linear map.

Our approach to prove Theorems \ref{Local existence1}, \ref{Local existence2} and \ref{Local Existence3} is first to prove the existence
 of a mild solution with giving initial function $u_0$ and then to prove the mild solution is a classical solution, which will be partially achieved
  by the tools from semigroup theory. Hence it is necessary to collect some results that will be used from the semigroup theory. In this regards, we recall the following theorems on the existence of  mild and classical solutions of
  \begin{equation}
  \label{aux-mild-solu-eq}
  \begin{cases}
  u_t=(\Delta-1)u+\tilde F(t,u),\quad t>t_0\cr
  u(t_0)=u_0.
  \end{cases}
  \end{equation}
  For given $u_0\in X^\alpha$, $u(t)$ is called a {\it mild solution} of \eqref{aux-mild-solu-eq} on $[t_0,T)$ if
  $u\in C([t_0,T),X^\alpha)$ and
  $$
  u(t)=T(t)u_0+\int_{t_0}^t T(t-s)\tilde F(s,u(s))ds\quad {\rm for}\quad s\in (t_0,T).
  $$

\begin{tm}\label{Existence of Mild Soltion} (\cite[Theorem 3.3.3, Theorem 3.3.4, Lemma 3.5.1]{Dan Henry}) Assume that  $0 \leq \alpha <1$, $U$ an open subset of $\R\times X^{\alpha}$, and $\tilde F: U\rightarrow X$  is locally H\"older continuous and locally  Lipschitz continuous in $u$. Then for any $(t_0,u_{0})\in U$
there exists $T = T(u_{0}) > t_0$ such that \eqref{IVP'} has a unique mild solution $u(t;t_0,u_0)$
on $[t_0,T)$ with initial value $u(t_0;t_0,u_0) =u_{0}$. Moreover, $u(\cdot;t_0,u_0)\in C^1((t_0,T),X)$; $u(t;t_0,u_0)\in {\rm Dom}(\Delta-I)$ for $t\in (t_0,T)$;  \eqref{aux-mild-solu-eq}
holds in $X$ for $t\in [t_0,T)$; and the mappings
$$(t_0,T)\ni t\mapsto u(\cdot;t_0,u_0)\in X^\alpha,\quad (t_0,T)\ni t\mapsto \partial u_t(\cdot;t_0,u_0)\in X^\gamma$$ are locally H\"older continuous for $0<\gamma\ll 1$.
\end{tm}

\begin{tm}\label{Existence of Mild Soltion and maximal time}(\cite[Theorem 3.3.4]{Dan Henry})
Assume that $\tilde F$ is in as the previous theorem, and also assume that for every closed bounded set $B \subset U,$ the image $\tilde F(B)$ is bounded in $X$. If $u(t;t_0,u_0)$ is a solution of \eqref{aux-mild-solu-eq} on $[t_0,t_{1})$ and $t_{1}$ is maximal in the sense that there is no solution of $\eqref{IVP'}$ on $[t_0,t_2)$ if $t_2 > t_1$ (when $t_1<\infty$), then either $t_{1}= +\infty$ or else there exists a sequence $t_{n}\rightarrow t_{1}$ as $n\rightarrow +\infty$ such that $u(t_n;t_0,u_0)\rightarrow \partial U.$ (If $U$ is unbounded, the point at infinity is included in $\partial U.)$
\end{tm}

\begin{tm}
\label{more-on-existence-of-mild-solution}(\cite[Theorem 7.1.3]{Dan Henry})
Assume that $0\le \alpha <1$, $\tilde F(t,u)=B(t)u$, and $[t_0,T]\ni t\mapsto B(t)\in \mathcal{L}(X^\alpha,X)$ is H\"older continuous. Then for any $u_0\in X$, there is a unique $u(\cdot;t_0,u_0)\in C([t_0,T],X)$ such that $u(t_0;t_0,u_0)=u_0$; $u(t;t_0,u_0)\in {\rm Dom}(I-\Delta)$; $u(\cdot;t_0,u_0)\in C^1((t_0,T],X)$; and $u(t;t_0,u_0)$ satisfies \eqref{aux-mild-solu-eq}
in $X$ for $t\in [t_0,T]$. Moreover, the mapping $(t_0,T]\ni t\to u(t;t_0,u_0)\in X^\beta$ is locally H\"older continuous for any $0\le \beta<1$.
\end{tm}

 For given $u_0\in X^\alpha$ ($0\le \alpha<1$) and $T>0$, $[0,T)\ni t\mapsto  (u(\cdot,t;u_0),v(\cdot,t;u_0))\in X^\alpha\times X^\alpha$ is
 called a {\it mild solution} of \eqref{IntroEq1} on $[0,T)$ with initial function $u_0$  if $u(\cdot,\cdot;u_0),v(\cdot,\cdot;u_0)\in C([0,T),X^\alpha)$, $v(\cdot,t;u_0)=-(\Delta-I)^{-1}(u(\cdot,t;u_0))$, and $u(t):=u(\cdot,t;u_0)$ satisfies
$$
u(t)=T(t)u_{0} +\int_{0}^{t}T(t-s)\Big(\chi\nabla\cdot ( u(s)\nabla(\Delta-I)^{-1}(u(s)))+(a+1) u(s) -b u^{2}(s)\Big)ds
$$
for $t\in [0,T)$.
Note that, if $(u\cdot,\cdot;u_0),v(\cdot,\cdot;u_0))$ is a mild solution of \eqref{IntroEq1} with initial function $u_0\in X^\alpha$, then $u(t):=u(\cdot,t;u_0)$ is a mild solution of the Cauchy problem (initial value problem)
\begin{equation}\label{IVP'}
\begin{cases}
u_{t}=(\Delta-I) u + F(u) \qquad \text{in} \,\, (0, T) \\
u(0)=u_{0}\quad {\rm in}\quad X^\alpha,
\end{cases}
\end{equation}
  where $F(u)=\chi \nabla\cdot (u\nabla(\Delta-I)^{-1}u )+(a +1)u -b u^{2}$. Conversely, for given $u_0\in X^\alpha$, if $u(t)$ is a mild solution of \eqref{IVP'}, then $(u(\cdot,t;u_0),v(\cdot,t;u_0))$ is a mild solution of \eqref{IntroEq1} with initial function $u_0$, where $u(\cdot,t;u_0)=u(t)$ and
$v(\cdot,t;u_0)=-(\Delta-I)^{-1}(u(t))$.

We next present some important embedding results on the fractional spaces in the case that $X=L^{p}(\R^N)$ (see \cite{Dan Henry}). Let $A=\Delta-I$. We have that ${\rm Dom}(A)=W^{2,p}(\R^N)$ ($1\leq p<\infty$)\  and the following continuous imbeddings
\begin{eqnarray}\label{Fractional power Imbedding}
X^{\alpha} \subset C^{\nu} \quad \text{if} \quad 0\leq \nu < 2\alpha -\frac{N}{p};\\
X^{\alpha}\subset W^{1,q}(\R^N) \quad \text{if}\ \alpha>\frac{1}{2} \ \text{and}\ \frac{1}{q}>\frac{1}{p}-\frac{(2\alpha-1)}{N};\\
X^{\alpha}\subset L^{q}\quad \text{if} \ \frac{1}{q}>\frac{1}{p}-\frac{2\alpha}{N}, \ \ q\geq p.
\end{eqnarray}
We have that $X^{\frac{1}{2}}=W^{1,p}(\R^{N}).$   Furthermore, there is a constant $C_{\alpha}$ such that
\begin{equation}\label{Eq_00}
\|(T(t)-I)u\|_{L^{p}}\leq C_{\alpha}t^{\alpha}\|u\|_{X^{\alpha}} \quad \text{for all}\ u\in X^{\alpha}.
\end{equation}
Using the $L^{p}-L^{q}$ estimates  for the convolution product, concretely,
\begin{equation}\label{Convolution Product Lp estimates}
\|f\ast g\|_{L^{q}(\R^{N}}\leq \|f\|_{L^{r}(\R^{N})}\|g\|_{L^{p}(\R^{N})} \ \ \ \forall \ 1\leq p, q, r \leq +\infty \ \ \text{and}\ \ \frac{1}{q}+1=\frac{1}{p}+\frac{1}{r},
\end{equation}
 we can easily show that the analytic semigroup $\{T(t)\}_{t\geq 0} $ generated by $\Delta-I$ on $L^p(\R^N)$ enjoys the following $L^{p}-L^q$ a priori estimates:
 \begin{equation}\label{Lp Estimates}
 \begin{cases}
 \|T(t)u \|_{L^{q}(\mathbb{R}^{N})}\leq C t^{-(\frac{1}{p}-\frac{1}{q})\frac{N}{2}}e^{-t}\|u\|_{L^{p}(\mathbb{R}^{N})} \ \ \text{for}\ \ 1\leq p\leq q\leq +\infty\cr\cr
 \|(I-\Delta)^{\alpha} T(t)u \|_{L^q}\leq C_{\alpha} t^{-\alpha-(\frac{1}{p}-\frac{1}{q})\frac{N}{2}}e^{-t}\|u\|_{L^{p}(\mathbb{R}^{N})} \ \ \text{for}\ \ 1\leq p\leq q\leq +\infty
 \end{cases}
 \end{equation}
 for every $t>0$ and $\alpha\geq 0$, where $C$ and $C_{\alpha}$ are constant depending only on $p,q$ and $N.$ In fact the first inequality in $(\ref{Lp Estimates})$ is a direct consequence of $(\ref{Convolution Product Lp estimates})$, while the second is a result of the combination of Theorem 1.4.3 in \cite{Dan Henry} and the first inequality.

We now state a result that will be needed in the proof of time global existence theorem. The result is a variant of Gagliardo-Nirenberg inequality.

\medskip

 \begin{lem}\label{Pre1} (\cite[Lemma 2.4]{Sug}) Let $N\geq1$, $a_0>2$,  $u\in L^{q_{1}}(\mathbb{R}^{N})$ with $q_{1}\geq 1$ and $u^{\frac{r}{2}}\in H^{1}(\mathbb{R}^{N})$
with $r>1.$ If $q_{1}\in [1, r], q_{2}\in[\frac{r}{2}, \ a_0\frac{r}{2}]$ and
\begin{equation}
\begin{cases}
1\leq q_{1}\leq q_{2}\leq \infty \ \ \text{when N=1,}\\
1\leq q_{1}\leq q_{2} < \infty \ \ \text{when N=2,}\\
1\leq q_{1}\leq q_{2}\leq \frac{Nr}{N-2} \ \ \text{when}\  N\geq 3,
\end{cases}
\end{equation}
then, it holds that
\begin{equation}
\|u\|_{L^{q_{2}}(\mathbb{R}^{N})}\leq C^{\frac{2}{r}}\|u\|^{1-\theta}_{L^{q_{1}}(\mathbb{R}^{N})}\|\nabla u^{\frac{r}{2}}\|^{\frac{2\theta}{r}}_{L^{2}(\mathbb{R}^{N})}
\end{equation}
with
$$
\theta=\frac{r}{2}\left( \frac{1}{q_{1}}-\frac{1}{q_{2}} \right)\left(\frac{1}{N}-\frac{1}{2}+\frac{r}{2q_{1}} \right)^{-1},
$$
where
\begin{equation*}
\begin{cases}
C \ \text{depends only on}\,\,  N\,\, {\rm  and}\,\, a_0 \ \  \text{when } q_{1}\geq \frac{r}{2},\\
C=C_{0}^{\frac{1}{\beta}}\,\, \ \text{with}\,\, C_{0}\,\, \text{depending only on}\,\,  N\,\,{\rm and}\,\, a_0 \,\,\text{when }1\leq q_{1}< \frac{r}{2},
\end{cases}
\end{equation*}
and
\begin{equation*}
\beta=\frac{q_{2}-\frac{r}{2}}{q_{2}-q_{1}}\left[ \frac{2q_{1}}{r} +\left( 1-\frac{2q_{1}}{r}\right)\frac{2N}{N+2}\right].
\end{equation*}
\end{lem}

We end this section by stating an important result that will be needed to complete the proof of the main theorem.

\begin{lem}(\cite[Exercise $4^{*}$, page 190]{Dan Henry})
\label{Lem2}
Assume that $a_{1},a_{2},\alpha,\beta$ are nonnegative constants , with $0\leq \alpha, \ \beta<1,$ and $0<T<\infty$.  There exists a constant $M(a_{2},\beta,T)<\infty$ so that for any integrable function $u : [0,\ T]\rightarrow \R$ satisfying that
$$
0\leq u(t)\leq a_{1}t^{-\alpha} +a_{2}\int_{0}^{t}(t-s)^{-\beta}u(s)ds
$$
for a.e t in $[0, \ T],$ we have
$$
0\leq u(t)\ \leq\ \frac{a_{1} M}{1-\alpha}t^{-\alpha}, \ \ \text{a.e. on }\ 0< t< T.
$$
\end{lem}

\section{Local existence of classical solutions}

In this section, we investigate the local existence and uniqueness of classical solutions of \eqref{IntroEq1} with various  given initial functions
 and prove Theorems \ref{Local existence1}, \ref{Local existence2}, \ref{Local Existence3}, and \ref{Lp-Bound1}. We first establish some  important lemmas.

\begin{lem}\label{L_Infty bound} Let $p\in[1,\infty)$ and  $\{T(t)\}_{t>0}$ be the semigroup in \eqref{semigroup-eq} generated by $\Delta-I$ on $L^p(\R^N)$. For every $t>0$, the operator $T(t) \nabla \cdot $ has a unique bounded extension on $\big(L^{p}(\R^{N})\big)^N$ satisfying
\begin{equation}\label{EqL_Infty_001}
\|T(t)\nabla \cdot  u\|_{L^{p}(\R^{N})}\leq C_{1}t^{-\frac{1}{2}}e^{-t}\|u\|_{L^{p}(\R^{N})} \ \ \ \forall \ u\in\big( L^{p}(\R^{N})\big)^N, \ \forall \ t>0,
\end{equation}
 where $C_1$ depends only on $p$ and $N$.
 Furthermore, for every $q\in [p , \infty]$, we have that $T(t)\nabla\cdot  u\in L^{q}(\R^N)$ with
\begin{equation} \label{EqL_Infty_002}
\|T(t)\nabla \cdot u\|_{L^{q}}\leq C_{2}t^{-\frac{1}{2}-\frac{N}{2}(\frac{1}{p}-\frac{1}{q})}e^{-t}\|u\|_{L^{p}(\R^{N})} \ \ \ \forall \ u\in \big(L^{p}(\R^{N})\big)^N, \ \forall \ t>0,
\end{equation}
where $C_{2}$ is constant depending only on $N$, $q$ and $p$.
\end{lem}

\begin{proof}
Since $C^{\infty}_{c}(\R^{N})$ is dense in $L^{p}(\R^N),$ it is enough to prove that inequalities \eqref{EqL_Infty_001} and \eqref{EqL_Infty_002} hold on $\big(C^{\infty}_{c}(\R^{N})\big)^N.$
For every $u=(u_1,u_2,\cdots,u_N)\in \big(C^{\infty}_{c}(\R^{N})\big)^N$ and $t>0$, using integration by parts, we obtain that
\begin{equation*}\label{Eq_Linfty00}
T(t)\partial_{x_{i}} u_i  =  \frac{e^{-t}}{\left(4\pi t\right)^{\frac{N}{2}}}\int_{\R^{N}}e^{-\frac{|z|^{2}}{4t}}\partial_{x_{i}}u_i(x-z)dz
 =\frac{e^{-t}}{2t\left(4\pi t\right)^{\frac{N}{2}}}\int_{\R^{N}}z_{i}e^{-\frac{|z|^{2}}{4t}}u_i(x-z)dz.
\end{equation*}
Using the $L^{p}-L^q$ estimates \eqref{Convolution Product Lp estimates}, we have that
\begin{equation}\label{Eq_Linfty01}
\|T(t)\partial_{x_{i}} u_i \|_{L^{p}(\R^{N})}\leq\frac{e^{-t}}{2t\left(4\pi t\right)^{\frac{N}{2}}}\|H_{i}(\cdot,t)\|_{L^{1}(\R^{N})} \|u_i\|_{L^{p}(\R^{N})}
\end{equation}
and
\begin{equation}\label{Eq_Linfty02}
\|T(t)\partial_{x_{i}} u_i \|_{L^{q}(\R^{N})}\leq\frac{e^{-t}}{2t\left(4\pi t\right)^{\frac{N}{2}}}\|H_{i}(\cdot,t)\|_{L^{r}(\R^{N})} \|u_i\|_{L^{p}(\R^{N})} ,
\end{equation}
where $\frac{1}{q}+1=\frac{1}{r}+\frac{1}{p}$ and
\begin{equation}\label{Eq_Linfty03}
H_{i}(z,t)=z_{i}e^{-\frac{|z|^{2}}{4t}}.
\end{equation}
Making change of variable $z=\sqrt{4t}y,$ we obtain that
\begin{equation}\label{Eq_Linfty04}
\|H_{i}(\cdot,t)\|_{L^{r}(\R^{N})}=\left(4t\right)^{\frac{1}{2}+\frac{Nr}{2}}\|H_{i}(.,1)\|_{L^{r}(\R^{N})}=\left(4t\right)^{\frac{1}{2}+\frac{N}{2}(1+\frac{1}{q}-\frac{1}{p})}\|H_{i}(.,1)\|_{L^{r}(\R^{N})}.
\end{equation}
Combining the fact that
\begin{equation}\label{Eq_Linfty05}
\frac{t^{\frac{1}{2}+\frac{Nr}{2}}}{t(t)^{\frac{N}{2}}}=\begin{cases}t^{-\frac{1}{2} -\frac{N}{2}(\frac{1}{p}-\frac{1}{q}) }\ \text{if} \ \ q>p \\
t^{-\frac{1}{2}} \qquad \qquad \text{if} \ \ q=p
\end{cases}
\end{equation}
 with inequalities $(\ref{Eq_Linfty01}), \ (\ref{Eq_Linfty02})$  and $(\ref{Eq_Linfty04})$  we obtain inequalities \eqref{EqL_Infty_001} and \eqref{EqL_Infty_002}.
\end{proof}

Considering the analytic semigroup $\{e^{-t\Delta}\}_{t \geq 0}$ generated by $\Delta$ on bounded domains coupled with Neumann boundary condition, a result in the style of inequality \eqref{EqL_Infty_001} was first established in \cite{Dirk and Winkler} (Lemma 2.1) and inequality \eqref{EqL_Infty_002}  was later obtained in \cite{fiwy} (Lemma 3.3). The authors in \cite{Dirk and Winkler} and \cite{fiwy} used different methods to establish these results. It should be noted that the proof presented in \cite{Dirk and Winkler} is difficult to be adapted for the whole space because it is based on the measure of the domain. Taking advantage of the explicit formula  of the analytic semigroup $\{e^{-t\Delta}\}_{t \geq 0}$ generated by $\Delta$ on the whole space $\R^N$, our proof is simpler and yields the same results.

Note that $C^{\infty}_{c}(\R^N)$ is not a dense subset of $C_{unif}^{b}(\R^N)$. Hence the arguments used in the previous proof can not be applied directly on $C_{\rm unif}^{b}(\R^N).$ This problem can be overcome by choosing an adequate dense subset. This leads to a version of this result on $C_{\rm unif}^{b}(\R^N)$ that we formulate in the next lemma.

\medskip

\begin{lem}\label{L_Infty bound 2} Let $T(t)$ be the semigroup in \eqref{semigroup-eq} generated by $\Delta-I$ on $C_{\rm unif}^b(\R^N)$. For every  $t>0$, the operator $T(t) \nabla\cdot $ has a unique bounded extension on $\big(C_{\rm unif}^b(\R^N)\big)^N$  satisfying
\begin{equation}\label{EqL_Infty03}
\|T(t)\nabla \cdot u\|_{\infty}\leq\frac{N}{\sqrt{\pi}} t^{-\frac{1}{2}}e^{-t}\|u\|_{\infty} \ \ \ \forall \ u\in \big(C_{\rm unif}^{b}(\R^N)\big)^N, \ \forall \ t>0.
\end{equation}
\end{lem}

\begin{proof}
Let
$$
C_{\rm unif}^{1,b}(\R^N)=\{u\in C^1(\R^N)\,|\, u(\cdot),\p_{x_i}u(\cdot)\in C_{\rm unif}^b(\R^N),\,\, i=1,2, \cdots,N\}.
$$
Since $C_{\rm unif}^{1,b}(\R^N)$ is dense in $C_{\rm unif}^{b}(\R^{N}),$ it is enough to prove that inequalities \eqref{EqL_Infty03} hold on $\big(C_{\rm unif}^{1,b}(\R^N)\big)^N.$
For every $u=(u_1,u_2,\cdots,u_N)\in \big(C_{\rm unif}^{1,b}(\R^N)\big)^N$ and $t>0$, we have
\begin{equation}\label{E00001}
T(t)\partial_{x_{i}} u_i  =  \frac{e^{-t}}{\left(4\pi t\right)^{\frac{N}{2}}}\int_{\R^{N}}e^{-\frac{|z|^{2}}{4t}}\partial_{x_{i}}u_i(x-z)dz=\lim_{R\rightarrow\infty}\left[\frac{e^{-t}}{\left(4\pi t\right)^{\frac{N}{2}}}\int_{B(0, R)}e^{-\frac{|z|^{2}}{4t}}\partial_{x_{i}}u_i(x-z)dz \right].
\end{equation}
Next, for every $R>0$ using integration by parts, we have
\begin{align}\label{E00002}
\int_{B(0, R)}e^{-\frac{|z|^{2}}{4t}}\partial_{x_{i}}u_i(x-z)dz
& =  \frac{1}{2t}\int_{B(0, R)}z_{i}e^{-\frac{|z|^{2}}{4t}}u_i(x-z)dz-\int_{\partial B(0, R)}e^{-\frac{|z|^{2}}{4t}}u_i(x-z)\nu_{i}(z)ds(z)\nonumber\\
& =  \frac{1}{2t}\int_{B(0, R)}z_{i}e^{-\frac{|z|^{2}}{4t}}u_i(x-z)dz-e^{-\frac{R^{2}}{4t}}\int_{\partial B(0, R)}u_i(x-z)\frac{z_{i}}{R}ds(z).
\end{align}
Since $u$ is uniformly bounded and the function $z\in \R^N\mapsto z_{i}e^{-\frac{|z|^{2}}{4t}}$ belongs to $L^{1}(\R^N),$ then
\begin{equation}\label{E00003}
\lim_{R\rightarrow\infty}\frac{1}{2t}\int_{B(0, R)}z_{i}e^{-\frac{|z|^{2}}{4t}}u_i(x-z)dz=\frac{1}{2t}\int_{\R^{N}}z_{i}e^{-\frac{|z|^{2}}{4t}}u_i(x-z)dz.
\end{equation}
On the other hand we have
\begin{equation}\label{E00004}
\left|e^{-\frac{R^{2}}{4t}}\int_{\partial B(0, R)}u_i(x-z)\frac{z_{i}}{R}ds(z) \right|\leq \|u_i\|_{\infty}e^{-\frac{R^{2}}{4t}}\int_{\partial B(0, R)}ds(z)\rightarrow 0\ \text{as}\ R \rightarrow \infty.
\end{equation}
Combining \eqref{E00001}, \eqref{E00002},\eqref{E00003} and \eqref{E00004}, we obtain that
\begin{equation*}\label{E00005}
T(t)\partial_{x_{i}} u _i =  \frac{e^{-t}}{2t\left(4\pi t\right)^{\frac{N}{2}}}\int_{\R^{N}}z_{i}e^{-\frac{|z|^{2}}{4t}}u_i(x-z)dz= \frac{e^{-t}}{2t\left(4\pi t\right)^{\frac{N}{2}}}\int_{\R^{N}}H_{i}(z,t)u_i(x-z)dz,
\end{equation*}
where the function $H_{i}$ is defined by \eqref{Eq_Linfty03}. Thus, using \eqref{Eq_Linfty04} and \eqref{Eq_Linfty05}, we obtain that
\begin{equation*}
\|T(t)\partial_{x_{i}} u_i\|_{\infty}\leq \frac{t^{-\frac{1}{2}}e^{-t}}{\pi^{\frac{N}{2}}}\|H_{i}(\cdot,1)\|_{L^{1}(\R^N)}\|u_i\|_{\infty}.
\end{equation*}
Direct computations yield that $\|\|H_{i}(\cdot,1)\|_{L^{1}(\R^N)}\|=\pi^{\frac{N-1}{2}}.$ Hence
\begin{equation}\label{E00006}
\|T(t)\partial_{x_{i}} u_i\|_{\infty}\leq \frac{t^{-\frac{1}{2}}e^{-t}}{\sqrt{\pi}}\|u_i\|_{\infty}.
\end{equation}
Inequality \eqref{EqL_Infty03} easily follows from \eqref{E00006}.
\end{proof}

 In the next Lemma, we shall provide an explicit a priori estimate of the gradient of the solution $v(\cdot,\cdot)$ in the second
 equation of \eqref{IntroEq1}. This a priori estimate will be useful in the proof of existence theorem and the discussion on the asymptotic behavior of the solution.

 \begin{lem}\label{lemma2}
 For every $u\in C^{b}_{\rm unif}(\R^{N})$, we have that
 \begin{equation}
 \|\partial_{x_{i}}(\Delta-I)^{-1}u\|_{\infty}\leq \|u\|_{\infty}
 \end{equation}
 for each $i=1,\cdots,N$. Therefore we have
 \begin{equation}\label{EqT}
 \|\nabla(\Delta-I)^{-1}u\|_{\infty}\leq\sqrt{N} \|u\|_{\infty}
 \end{equation}
 for every $u\in C_{\rm unif}^b(\R^{N})$.
 \end{lem}
 \begin{proof} Let $u\in C_{\rm unif}^b(\R^N)$ and set  $v=(I-\Delta)^{-1}u.$ According to \eqref{Eq_Inv} it follows that
 \begin{equation*}
v(x)=\int_{0}^{\infty}\int_{\R^N}\frac{e^{-s}}{(4\pi s)^{\frac{N}{2}}}e^{-\frac{|x-z|^{2}}{4s}}u(z)dzds
\end{equation*}
for every $x\in R^N$. Hence
\begin{eqnarray*}
\partial_{x_{i}}v(x)& = &\int_{0}^{\infty}\int_{\R^N}\frac{(z_{i}-x_{i})e^{-s}}{2s(4\pi s)^{\frac{N}{2}}}e^{-\frac{|x-z|^{2}}{4s}}u(z)dzds\nonumber \\
& = & \int_{0}^{\infty}\int_{\R^N}\frac{e^{-s}}{ {\pi}^{\frac{N}{2}}\sqrt{s}}H_{i}(z,1)u(x-z)dzds.
\end{eqnarray*}
Thus, using the fact that $\Gamma(\frac{1}{2})=\sqrt{\pi}$ and $\|H(\cdot,1)\|_{L^{1}(\R^N)}=\pi^{\frac{N-1}{2}}$, we obtain that
\begin{eqnarray*}
|\partial_{x_{i}}v(x)|& \leq & \frac{1}{\pi^{\frac{N}{2}}}\underbrace{\left[\int_{0}^{\infty}s^{-\frac{1}{2}}e^{-s}ds\right]}_{\Gamma(\frac{1}{2})}\|H_{i}(\cdot,1)\|_{L^{1}(\R^N)}\|u\|_{\infty}=\|u\|_{\infty}.
\end{eqnarray*}
The lemma thus follows.
 \end{proof}

Next, we  prove Theorems \ref{Local existence1}, \ref{Local existence2},  \ref{Local Existence3}, and \ref{Lp-Bound1} in subsections 3.1, 3.2,  3.3, and 3.4,  respectively.
Throughout subsections 3.1, 3.2,  3.3, and 3.4,  $C$  denotes a constant independent of the initial functions and the solutions under consideration,  unless specified otherwise.

 \subsection{Proof of Theorem  \ref{Local existence1}}

 In this subsection, we prove Theorem \ref{Local existence1}.  The main tools for the proof of this theorem are based on the contraction mapping theorem and the existence of classical solutions for linear parabolic equations with H\"older continuous  coefficients. Throughout this subsection, $X=C_{\rm unif}^b(\R^N)$ and $X^\alpha$ is the fractional power space of $I-\Delta$ on $X$ ($\alpha\in (0,1)$).

\begin{proof}[Proof of Theorem \ref{Local existence1}]
\textbf{(i) Existence of mild solution}. We first prove the existence of a mild solution of \eqref{IVP'} with given initial function $u_0\in C_{\rm unif}^b(\R^N)$, which will be done by proving five claims.

 Fix $u_0 \in C_{\rm unif}^b(\R^N) .$ For every $T>0$ and $R>0,$ let
$$\mathcal{S}_{R,T}:=\left\{u \in C([0, T ], C_{\rm unif}^b(\R^N) )\ | \ \|u\|_{X} \leq R \right \}.$$
Note that $\mathcal{S}_{R,T}$ is a closed subset of the Banach space $C([0, T ], C_{\rm unif}^b(\R^N) )$ with the sup-norm.

\medskip

\noindent {\bf Claim 1.} For any $u\in \mathcal{S}_{R,T}$ and $t\in [0,T]$,
 $(Gu)(t)$ is well defined, where
\begin{align*}
(Gu)(t)=& T(t)u_{0} +\chi\int_{0}^{t} T(t-s)\nabla\cdot ( u(s)  \nabla (\Delta-I)^{-1}u(s))ds\\
& +(1+a)\int_{0}^{t} T(t-s)u(s)ds-b\int_{0}^{t} T(t-s)u^{2}(s)ds,
\end{align*}
and the integrals are taken in $C_{\rm unif}^b(\R^N)$.
Indeed, let $u\in\mathcal{S}_{R,T}$ and $0<t\leq T$ be fixed. Since the function
$$[0 , t]\ni s\mapsto (a+1)u(s)-bu^2(s)\in C^{b}_{\rm unif}(\R^N)$$
is continuous, then the function $F_{1} : [0, t]\to C^{b}_{\rm unif}(\R^N)$ defined by
$$ F_{1}(s):= (1+a) T(t-s)u(s)-b T(t-s)u^{2}(s) $$
 is continuous. Hence the Riemann integral $\int_{0}^{t}F_{1}(s)ds$ in $C^b_{\rm unif}(\R^N)$ exists. Observe that for every $0<\varepsilon<t$ and $s\in [0, t-\varepsilon]$, we have
 $$
F_{2,\varepsilon}(s):=T(t-s)\nabla\cdot(u(s)\nabla(\Delta-I)^{-1}u(s))=T(t-\varepsilon-s)T(\varepsilon)\nabla\cdot(u(s)\nabla(\Delta-I)^{-1}u(s)),
$$
and the function $[0, t-\varepsilon]\ni s\mapsto T(\varepsilon)\nabla\cdot(u(s)\nabla(\Delta-I)^{-1}u(s))\in C^{b}_{\rm unif}(\R^N)$ is continuous. Thus the function $F_{2,\varepsilon} : [0 , t-\varepsilon] \to C^b_{\rm uinf}(\R^N)$ is continuous for every $0<\varepsilon<t$. Thus, the function $F_{2} : [0, t)  \to C^b_{\rm uinf}(\R^N)$ defined by
$$
F_{2}(s):= T(t-s)\nabla\cdot(u(s)\nabla(\Delta-I)^{-1}u(s))
$$
is continuous. Moreover,  using Lemma \ref{L_Infty bound 2} and inequality \eqref{EqT}, we have that
\begin{align*}
\int_0^t\|F_{2}(s)\|_{\infty}ds & \leq \chi\int_0^t\|T(t-s)\nabla\cdot\|\|u(s)\|_{\infty}\|\nabla(\Delta-I)^{-1}u(s)\|_{\infty}ds \nonumber\\
&\leq \frac{N\sqrt{N}}{\sqrt{\pi}}\chi\int_0^t(t-s)^{-\frac{1}{2}}e^{-(t-s)}\|u(s)\|^2_{\infty}ds\nonumber\\
&\leq \frac{NR^2\sqrt{N}}{\sqrt{\pi}}\chi\Gamma(\frac{1}{2}).
\end{align*}
Hence, the Riemann integral $\int_{0}^{t}F_2(s)ds$ in $C^b_{\rm unif}(\R^N)$ exists. Note that $(Gu)(t)=T(t)u_0+\int_0^tF_2(s)ds +\int_0^tF_1(s)ds$. Whence, Claim 1 follows.

\medskip

\noindent {\bf Claim 2.} For every $u\in \mathcal{S}_{R,T}$ and $0<\beta<\frac{1}{2}$, the function $(0, T]\ni t\to (Gu)(t)\in X^{\beta}$ is locally H\"older continuous, and $G$ maps $\mathcal{S}_{R,T}$ into $C([0,T],C_{\rm unif}^b(\R^N))$.

First, observe that
\begin{align}\label{aux-correct01}
(Gu)(t)=& \, \,\underbrace{T(t)u_{0}}_{I_{0}(t)}+\chi\underbrace{\int_{0}^{t}T(t-s)\nabla\cdot (u(s)\nabla (\Delta-I)^{-1}u(s))ds}_{I_{1}(t)} \nonumber\\ &\,\,\,\, +\underbrace{\int_{0}^{t}T(t-s)((a+1)u(s)-bu^{2}(s))ds}_{I_{2}(t)}.
\end{align}
For every $t>0$, it is clear that $T(t)u_{0}\in X^{\beta}$ because the semigroup $\{T(t)\}_{t}$ is  analytic. Furthermore, the divergence operator  $T(t)\nabla\cdot   $ satisfy
$$
T(t)\nabla\cdot  w=T(\frac{t}{2})(T(\frac{t}{2})\nabla\cdot w)\in \text{Dom}(\Delta)\subset X^{\beta}
$$
for all $t>0,\ w\in \big(C_{\rm unif}^b(\R^N)\big)^N$. Using Lemma \ref{L_Infty bound 2} and inequality \eqref{EqT}, we obtain that
\begin{align}
\int_{0}^{t}\|T(t-s)\nabla\cdot  (u(s)\nabla (\Delta-I)^{-1}u(s)) \|_{X^{\beta}}ds &= \int_{0}^{t}\|(\Delta-I)^{\beta}T(t-s)\nabla\cdot  (u(s)\nabla (\Delta-I)^{-1}u(s)) \|_{\infty}ds\nonumber \\
&\leq  C\int_{0}^{t}(t-s)^{-\beta-\frac{1}{2}}e^{-(t-s)}\|u(s)\nabla(\Delta-I)^{-1}u(s)\|_{\infty}ds\nonumber\\
&\leq  CR^{2}\int_{0}^{t}(t-s)^{-\beta-\frac{1}{2}}e^{-(t-s)}ds\nonumber\\
&\leq  CR^{2}\Gamma(\frac{1}{2}-\beta).
\end{align}
Since the operator $(\Delta-I)^{\beta}$ is closed, we have that
\[
I_{1}(t)=\int_{0}^{t}T(t-s)\nabla\cdot  (u(s)\nabla (\Delta-I)^{-1}u(s))ds \in X^{\beta}
\]
for every $t>0.$ Similar arguments show that $I_{2}(t)\in X^{\beta}$ for every $0<t\leq T$. Hence $u(t)\in X^{\beta}$ for every $t>0$.

Next, let $t\in (0, T)$ and $h>0$ such that $t+h\leq T$. We have
\begin{eqnarray}\label{correct02}
\|I_{0}(t+h)-I_{0}(t)\|_{\beta}
&\leq & \|(T(h)-I)T(t)u_{0}\|_{X^{\beta}}\leq  Ch^{\beta}e^{-h}\|T(t)u_{0}\|_{X^{\beta}}\nonumber\\
 &\leq & Ch^{\beta}t^{-\beta}e^{-(h+t)}\|u_{0}\|_{\infty}\leq  Ch^{\beta}t^{-\beta}\|u_{0}\|_{\infty},
\end{eqnarray}
\begin{align}\label{correct03}
\|I_{1}(t+h)-I_{1}(t)\|_{\beta}&\leq \int_{0}^{t}\|(T(h)-I)T(t-s)\nabla\cdot  (u(s)\nabla(\Delta-I)^{-1}u(s))\|_{X^{\beta}}ds \nonumber\\
& \,\,\,\,  + \int_{t}^{t+h}\|T(t+h-s)\nabla\cdot  (u(s)\nabla(\Delta-I)^{-1}u(s))\|_{X^\beta}\nonumber\\
&\leq  Ch^{\beta}\int_{0}^{t}(t-s)^{-\beta-\frac{1}{2}}e^{-(t+h-s)}\|(u(s)\nabla(\Delta-I)^{-1}u(s))\|_{\infty}ds \nonumber\\
& \,\,\,\,  + C\int_{t}^{t+h}(t+h-s)^{-\beta-\frac{1}{2}}e^{-(t+h-s)}\|u(s)\nabla(\Delta-I)^{-1}u(s))\|_{\infty}ds\nonumber\\
&\leq  CR^{2}h^{\beta}\int_{0}^{t}(t-s)^{-\beta-\frac{1}{2}}e^{-(t+h-s)}ds + CR^{2}\int_{t}^{t+h}(t+h-s)^{-\beta-\frac{1}{2}}e^{-(t+h-s)}ds\nonumber\\
&\leq  CR^{2}(h^{\beta} +  h^{\frac{1}{2}-\beta} ),
\end{align} and
 \begin{eqnarray}\label{correct04}
\|I_{2}(t+h)-I_{2}(t)\|_{\beta}&\leq &\int_{0}^{t}\|(T(h)-I)T(t-s)((a+1)u(s)-bu^{2}(s))\|_{X^{\beta}}ds \nonumber\\
& & + \int_{t}^{t+h}\|T(t+h-s)((a+1)u(s)-bu^{2}(s))\|_{X^{\beta}}ds \nonumber\\
&\leq & Ch^{\beta}\int_{0}^{t}(t-s)^{-\beta}e^{-(t+h-s)}\|(a+1)u(s)-bu^{2}(s)\|_{\infty}ds \nonumber\\
& & + C\int_{t}^{t+h}(t+h-s)^{-\beta}e^{-(t+h-s)}\|(a+1)u(s)-bu^{2}(s)\|_{\infty}ds\nonumber\\
&\leq  & CR^{2}(h^{\beta} +  h^{1-\beta} ).
\end{eqnarray}
Combining \eqref{aux-correct01},\eqref{correct02},\eqref{correct03} and \eqref{correct04}, we deduce that the function $(0, T]\ni t\to (Gu(t))\in X^{\beta}$ is locally Holder continuous.

Now it is clear that $t\to (Gu)(t)\in C_{\rm unif}^b(\R^N)$ is continuous in $t$ at $t=0$. The claim then follows.

\medskip

\noindent {\bf Claim 3.} For every $R>\|u_{0}\|_{\infty},$ there exists $T:=T(R)$ such that  $G$ maps $\mathcal{S}_{R,T}$ into itself.

First, observe that for any $u\in \mathcal{S}_{R,T}$,  we have \begin{eqnarray} \|G(u)(t)\|_{\infty }  & \leq & \|T(t)u_{0}\|_{\infty}+\chi\int_0^t\| T(t-s)\nabla\cdot ( u(s)  \nabla (\Delta-I)^{-1}u(s))\|_{\infty}ds \nonumber\\ & & +(1+a)\int_{0}^{t} \|T(t-s)u(s)\|_{\infty}ds+b\int_{0}^{t} \|T(t-s)u^{2}(s)\|_{\infty}ds \nonumber\\ & \leq & e^{-t}\|u_{0}\|_{\infty}+\chi\int_0^t\| T(t-s)\nabla\cdot ( u(s)  \nabla (\Delta-I)^{-1}u(s))\|_{ \infty}ds \nonumber\\ & & +(1+a)R\int_{0}^{t} e^{-(t-s)}ds+bR^{2}\int_{0}^{t} e^{-(t-s)}ds \nonumber\\ & = & e^{-t}\|u_{0}\|_{\infty}+ R\left((1+a)+bR\right)(1-e^{-t})\nonumber\\ & & +\chi\int_0^t\| T(t-s)\nabla\cdot ( u(s)  \nabla (\Delta-I)^{-1}u(s))\|_{ \infty}ds. \end{eqnarray}
Using Lemma \ref{L_Infty bound 2} and inequality \eqref{EqT}, the last inequality can be improved to \begin{eqnarray} \label{aux-claim3-eq1} \|G(u)(t)\|_{\infty } & \leq & e^{-t}\|u_{0}\|_{\infty}+ R\left((1+a)+bR\right)(1-e^{-t})\nonumber\\ & &+C\chi\int_{0}^{t}(t-s)^{-\frac{1}{2}}e^{-(t-s)}\|( u(s)  \nabla (\Delta-I)^{-1}u(s))\|_{ \infty}ds \nonumber\\ & \leq & e^{-t}\|u_{0}\|_{\infty}+ R\left((1+a)+bR\right)(1-e^{-t})+C\chi R^{2}\int_{0}^{t}(t-s)^{-\frac{1}{2}}e^{-(t-s)}ds \nonumber\\ & \leq & e^{-t}\|u_{0}\|_{\infty}+ R\left((1+a)+bR\right)(1-e^{-t})+2C\chi R^{2}t^{\frac{1}{2}}. \end{eqnarray}

Now, by \eqref{aux-claim3-eq1}, we can now chose $T>0$ such that
$$
\|G(u)(t)\|_{\infty }\le e^{-t}\|u_{0}\|_{\infty}+ R\left((1+a)+bR\right)(1-e^{-t})+2C\chi R^{2}t^{\frac{1}{2}}<R\,\,\forall \, t\in [0,T].
$$
This together with Claim 2  implies  Claim 3.

\medskip

\noindent \textbf{Claim 4.} $G$ is a contraction map for T small and hence  has a fixed point $u(\cdot)\in \mathcal{S}_{R,T}$. Moreover, for every
$0<\beta<\frac{1}{2}$,
$(0,T]\ni t\to u(t)\in X^{\beta}$ is locally Holder continuous.

For every $u, w\in\mathcal{S}_{R,T},$ using again Lemma \ref{L_Infty bound 2}, we have
\begin{align*}
&  \|(G(u)-G(w))(t)\|_{\infty}\nonumber\\
&\leq  \chi\int_0^t\| T(t-s)\nabla\cdot ( u(s) \nabla (\Delta-I)^{-1}u(s)- w(s) \nabla (\Delta-I)^{-1}w(s))\|_{\infty}ds \nonumber\\
& \quad  +(1+a)\int_{0}^{t} \|T(t-s)(u(s)-w(s))\|_{\infty}ds+b\int_{0}^{t} \|T(t-s)(u^{2}(s)-w^{2}(s)\|_{\infty}ds\nonumber\\
&\leq  C\chi\int_{0}^{t}(t-s)^{-\frac{1}{2}}e^{-(t-s)}\|( u(s) \nabla (\Delta-I)^{-1}u(s)- w(s) \nabla (\Delta-I)^{-1}w(s))\|_{\infty}ds \nonumber\\
& \quad +((1+a)+2Rb)\int_{0}^{t} e^{-(t-s)}\|(u(s)-w(s))\|_{\infty}ds\nonumber\\
&\leq  C\chi\int_{0}^{t}(t-s)^{-\frac{1}{2}}e^{-(t-s)}\|( u(s) -w(s))\|_{\infty}\|\nabla (\Delta-I)^{-1}u(s)\|_{\infty}ds \nonumber\\
& \quad  + C\chi\int_{0}^{t}(t-s)^{-\frac{1}{2}}e^{-(t-s)}\|w(s)\|_{\infty}\| \nabla (\Delta-I)^{-1}(w(s)-u(s))\|_{\infty}ds \nonumber\\
& \quad  +(1+a+2Rb)\|u-w\|_{\mathcal{S}_{R,T}}\int_{0}^{t} e^{-(t-s)}ds\nonumber\\
&\leq  2CR\chi\|( u(s) -w(s))\|_{\mathcal{S}_{R,T}}\int_{0}^{t}(t-s)^{-\frac{1}{2}}e^{-(t-s)}ds   +(1+a+2Rb)\|u-w\|_{\mathcal{S}_{R,T}}\int_{0}^{t} e^{-(t-s)}ds\nonumber\\
&\leq \left[4CR\chi t^{\frac{1}{2}}+ (1+a+2Rb)t \right]\|( u(s) -w(s))\|_{\mathcal{S}_{R,T}}.
\end{align*}
Hence, choose T small satisfying
$$
4CR\chi t^{\frac{1}{2}}+ (1+a+2Rb)t < 1\quad \forall\,\, t\in [0,T],
$$
we have that $G$ is a contraction map. Thus there is $T>0$ and a unique function $u\in \mathcal{S}_{R,T}$ such that
\begin{equation*}\label{correct01}
u(t)={T(t)u_{0}}+\chi{\int_{0}^{t}T(t-s)\nabla\cdot  (u(s)\nabla (\Delta-I)^{-1}u(s))ds} +{\int_{0}^{t}T(t-s)((a+1)u(s)-bu^{2}(s))ds}.
\end{equation*}
Moreover, by Claim 2,  for every $0<\beta<\frac{1}{2}$, the function $t\in(0,T]\to u(t)\in X^{\beta}$ is locally Holder continuous.
Clearly, $u(t)$ is a mild solution of \eqref{IVP'} on $[0,T)$ with $\alpha=0$ and $X^0=C_{\rm unif}^b(\R^N)$.

\medskip

\noindent {\bf Claim 5.} There is $T_{\max}\in (0,\infty]$ such that \eqref{IVP'} has a mild solution $u(\cdot)$ on $[0,T_{\max})$ with
$\alpha=0$ and $X^0=C_{\rm unif}^b(\R^N)$.
Moreover, for every $0<\beta<\frac{1}{2}$, $(0,T_{\max})\ni t\mapsto u(\cdot)\in X^\beta$ is locally H\"older continuous.  If $T_{\max}<\infty$, then $\limsup_{t\to T_{\max}}\|u(t)\|_\infty=\infty$.

\smallskip
This claim follows the regular extension arguments.

\medskip

\noindent \textbf{(ii) Regularity and non-negativity.} We next show that the mild solution $u(\cdot)$ of \eqref{IVP'} on $[0,T_{\max})$ obtained in (i) is a nonnegative classical solution of \eqref{IVP'} on $[0,T_{\max})$ and satisfies \eqref{local-1-eq1}, \eqref{local-1-eq2}.

Let $0<t_{1}<T_{\max}$ be fixed. It follows from claim 2 that for $0<\nu\ll 1$, $u_{1}:=u(t_{1})\in C^{\nu}_{\rm unif}(\R^{N})$, and the mappings
$$ t\to u(\cdot,t+t_1):=u(t+t_1)(\cdot)\in C^\nu_{\rm unif}(\R^N),\,\,\, t\mapsto v(\cdot,t+t_1)\in C^\nu_{\rm unif}(\R^N)
$$
$$ t\mapsto \frac{\partial v(\cdot,t+t_1)}{\partial x_i}\in C^\nu_{\rm unif}(\R^N),\,\,\,\,  t\mapsto \frac{\partial^2 v(\cdot,t+t_1)}{\partial x_i\partial x_j}\in C^\nu_{\rm unif}(\R^N)
$$
are locally H\"older continuous in $t\in (-t_1,T_{\max}-t_1)$, where
$v(\cdot,t+t_1):=(I-\Delta)^{-1}u(\cdot,t+t_1)$ and $i,j=1,2,\cdots,N$.
 Consider  the initial value problem
\begin{equation}
\label{aux-eq-for-thm1}
\begin{cases}
\frac{\partial}{\partial t}\tilde{u}=(\Delta -1)\tilde{u}+\tilde F(t,\tilde u), \quad x\in\R^N,\,\, t>0\\
\tilde{u}(x,0)=u_{1}(x), \quad  x\in \R^N,
\end{cases}
\end{equation}
where $\tilde F(t,\tilde u)=-\chi\nabla v(\cdot,t+t_1)\nabla \tilde{u} +(a+1-\chi v(\cdot,t+t_1)-(b-\chi)u(\cdot,t+t_1))\tilde{u}$.
Then by \cite[Theorem 11 and Theorem 16 in Chapter 1]{Friedman}, \eqref{aux-eq-for-thm1} has a unique classical solution
$\tilde u(x,t)$ on $[0,T_{\max}-t_1)$ with $\lim_{t\to 0+} \|\tilde u(\cdot,t)-u_1\|_\infty=0$.
In fact $\tilde{u}$ is given by the equation
\begin{equation*}
\tilde{u}(x,t)=\int_{\R^N}\Gamma(x,t,y,t_{1})u_{1}(y)dy
\end{equation*}
with the function $\Gamma$ satisfying the inequalities
\begin{equation*}
|\Gamma(x,t,y,\tau)|\leq C \frac{e^{-\frac{\lambda_{0}|x-y|^{2}}{4(t-\tau)}}}{(t-\tau)^{-\frac{N}{2}}}
\quad {\rm and}\quad
|\partial_{x_{i}}\Gamma(x,t,y,\tau)|\leq C \frac{e^{-\frac{\lambda_{0}|x-y|^{2}}{4(t-\tau)}}}{(t-\tau)^{-\frac{(N+1)}{2}}}
\end{equation*}
for every $0<\lambda_{0}<1$.
 By a priori interior estimates
for parabolic equations (see \cite[Theorem 5]{Friedman}), we  have
that $$\tilde u(\cdot,\cdot)\in C^1((0,T_{\max}-t_1),C_{\rm unif}^b(\R^N),$$
 and  the mappings
$$
 t\mapsto \tilde u(\cdot,t)\in C_{\rm unif}^\nu(\R^N),\quad
t\mapsto  \frac{\partial\tilde  u}{\partial x_i}(\cdot,t)\in C_{\rm unif}^\nu(\R^N),
$$
$$
 t\mapsto  \frac{\partial^2 \tilde u}{\partial x_i\partial x_j}(\cdot,t)\in C_{\rm unif}^\nu(\R^N),\quad
 t\mapsto  \frac{\partial\tilde  u}{\partial t}(\cdot,t)\in C_{\rm unif}^\nu(\R^N)
$$
are locally H\"older continuous in $t\in (0,T_{\max}-t_1)$ for  $i,j=1,2,\cdots,N$ and $0<\nu\ll 1$. Hence, by \cite[Lemma 3.3.2]{Dan Henry},
 $\tilde u(t)(\cdot)=\tilde u(\cdot,t)$ is also a mild solution of \eqref{aux-eq-for-thm1} and then satisfies the following integral equation,
\begin{eqnarray*}\label{utilde_Eq}
\tilde{u}(t) = T(t)u_{1}+\int_{0}^{t}T(t-s)(-\chi\nabla v(s+t_1)\nabla\tilde{u}(s)+ (a+1-\chi v(s+t_1) -(b-\chi)u(s+t_{1}))\tilde{u}(s) )ds
\end{eqnarray*}
for $t\in [0,T_{\max}-t_1)$.

Now, using the fact that $\nabla\tilde{u}\nabla v(\cdot+t_1)= \nabla\cdot (\tilde{u}\nabla v(\cdot+t_1))-(v(\cdot+t_1)-u(\cdot+t_{1}))\tilde{u}$, we have
\begin{align}\label{utilde_Eq}
\tilde{u}(t)
 &= T(t)u_{1}-\chi\int_{0}^{t}T(t-s)(\nabla\cdot (\tilde{u}(s) \nabla v(s+t_1))ds +\chi\int_{0}^{t}T(t-s)(v(s+t_1)-u(s+t_{1}))\tilde{u}(s)ds\nonumber \\
 & \,\,  +\int_{0}^{t}T(t-s)(a+1-\chi v -(b-\chi)u(s+t_{1}))\tilde{u}(s)ds\nonumber\\
 &= T(t)u_{1}-\chi\int_{0}^{t}T(t-s)\nabla\cdot  (\tilde{u}(s) \nabla v(s+t_1))ds +\int_{0}^{t}T(t-s)(a+1 -bu(s+t_{1}))\tilde{u}(s)ds.
\end{align}
On the other hand from equation \eqref{aux-correct01}, we have that
\begin{eqnarray}\label{uEq}
u(t+t_{1})& = &T(t)u_{1}-\chi\int_{0}^{t}T(t-s)\nabla\cdot  ( u(s+t_{1})\nabla v(s+t_1))ds \nonumber\\
& & +\int_{0}^{t}T(t-s)(a+1 -bu(s+t_{1}))u(s+t_{1})ds.
\end{eqnarray}
Taking the difference side by side of \eqref{utilde_Eq}
and \eqref{uEq} and using Lemma \ref{L_Infty bound 2}, we obtain for any $\epsilon>0$ and $0<t<T_\epsilon<T_{\max}-t_1-\epsilon$ that
\begin{eqnarray*}\label{correct05}
\|\tilde{u}(t)-u(t+t_{1})\|_{\infty} &\leq & \chi\int_{0}^{t}\|T(t-s)\nabla\cdot  ( (u(s+t_{1})-\tilde{u}(s))\nabla v(s+t_1))\|_{\infty}ds \nonumber\\
& & +\int_{0}^{t}\|T(t-s)(a+1 -bu(s+t_{1}))(u(s+t_{1})-\tilde{u}(s))\|_{\infty}ds\nonumber\\
&\leq & C\chi\int_{0}^{t}(t-s)^{-\frac{1}{2}}e^{-(t-s)}\| (u(s+t_{1})-\tilde{u}(s))\nabla v(s+t_1))\|_{\infty}ds \nonumber\\
& & +\int_{0}^{t}e^{-(t-s)}\|(a+1 -bu(s+t_{1}))(u(s+t_{1})-\tilde{u}(s))\|_{\infty}ds\nonumber\\
&\leq & C\chi \sup_{s\in[0,T_\epsilon]}\|\nabla v(s+t_1)\|_\infty\int_{0}^{t}(t-s)^{-\frac{1}{2}}e^{-(t-s)}\| u(s+t_{1})-\tilde{u}(s)\|_{\infty}ds \nonumber\\
& & +C (a+1+b\sup_{s\in[0,T_\epsilon]}\|u(s+t_1)\|_\infty)\int_{0}^{t}e^{-(t-s)}\|u(s+t_{1})-\tilde{u}(s)\|_{\infty}ds.
\end{eqnarray*}
Combining this last inequality with Lemma \ref{Lem2}, we conclude that
$$\tilde{u}(t)=u(t+t_{1}) $$
for every $t\in [0,T_\epsilon]$. We then have that $u$ is a classical solution of \eqref{IVP'}  on $[0,T_{\max})$ and  satisfies \eqref{local-1-eq1} and \eqref{local-1-eq2}.
Since $u_{0}\geq 0$, by comparison principle for parabolic equations, we get $u(x,t)\geq  0$.

Let $u(\cdot,t;u_0)=u(t)(\cdot)$ and $v(\cdot,t;u_0)=(I-\Delta)^{-1} u(\cdot,t;u_0)$. We have that
$(u(\cdot,\cdot;u_0),v(\cdot,\cdot;u_0))$ is a nonnegative classical solution of \eqref{IntroEq1} on $[0,T_{\max})$ with initial
function $u_0$ and $u(\cdot,t;u_0)$ satisfies \eqref{local-1-eq1} and \eqref{local-1-eq2}.

\medskip

\noindent {\bf (iii)  Uniqueness.} We now prove that for given $u_0\in C_{\rm unif}^b(\R^N)$, \eqref{IntroEq1} has a unique classical solution
$(u(\cdot,\cdot;u_0),v(\cdot,\cdot;u_0))$ satisfying \eqref{local-1-eq1} and \eqref{local-1-eq2}.

Any classical solution of \eqref{IntroEq1} satisfying the properties of Theorem \ref{Local existence1} clearly satisfies the integral equation \eqref{uEq}. Suppose that for given $u_0\in C_{\rm unif}^b(\R^1)$ with $u_0\ge 0$, $(u_1(t,x),v_1(t,x))$ and $(u_2(t,x),v_2(t,x))$ are two classical solutions of \eqref{IntroEq1} on $\R^N\times[0, T)$ satisfying the properties of Theorem \ref{Local existence1}. Let $0<t_1<T'<T$ be fixed. Thus $\sup_{0\leq t\leq T'}(\|u_{1}(\cdot,t)\|_{\infty}+\|u_{2}(\cdot,t)\|_{\infty})<\infty$.  Let
$u_i(t)=u_i(\cdot,t)$ and $v_{i}(t)=(I-\Delta)^{-1}u_{i}(t)$  for every $i=1,2$ and $0\leq t<T$.  For every $t\in [t_1,T']$, and $i=1, 2$ we have that
$$
u_{i}(t)=T(t-t_1)u_{i}(t_1)+\chi\int_{t_1}^{t}T(t-s)\nabla\cdot  (u_{i}(s)\nabla v_{i}(s))ds +\int_{t_1}^{t}T(t-s)(a+1 -bu_{i}(s))u_{i}(s)ds.
$$
Hence for $t_1\le t\le T^{'}$,
\begin{align*}
\|u_{1}(t)-u_{2}(t)\|_{\infty}& \leq  \|(u_1(t_1)-u_2(t_1))\|_\infty+C\chi \int_{t_1}^{t}(t-s)^{-\frac{1}{2}}e^{-(t-s)}\|u_{1}(s)\nabla v_{1}(s)-u_{2}(s)\nabla v_{2}(s)\|_{\infty}ds \nonumber\\
& \, + \int_{0}^{t}e^{-(t-s)}\|u_{1}(s)-u_{2}(s)\|_{\infty}(a+1+b(\|u_{1}(s)\|_{\infty}+\|u_{2}(s)\|_{\infty}) )ds \nonumber\\
& \leq \|(u_1(t_1)-u_2(t_1))\|_\infty+ C\chi \int_{t_1}^{t}(t-s)^{-\frac{1}{2}}e^{-(t-s)}\| u_{1}(s)-u_{2}(s)\|_{\infty}\|\nabla v_{1}(s)\|_{\infty}\nonumber\\
& \, +C\chi\int_{t_1}^{t}(t-s)^{-\frac{1}{2}}e^{-(t-s)}\|u_{2}(s)\|\|\nabla (v_{2}(s)-v_{1}(s)\|_{\infty}ds \nonumber\\
& \, +(a+1+b\sup_{0\leq \tau\leq T'}(\|u_{1}(\tau)\|_{\infty}+\|u_{2}(\tau)\|_{\infty}) ) \int_{0}^{t}e^{-(t-s)}\|u_{1}(s)-u_{2}(s)\|_{\infty}ds \nonumber\\
& \leq  \|(u_1(t_1)-u_2(t_1))\|_\infty+ C\sqrt{N}\chi \int_{t_1}^{t}(t-s)^{-\frac{1}{2}}e^{-(t-s)}\| u_{1}(s)-u_{2}(s)\|_{\infty}\| u_{1}(s)\|_{\infty}\nonumber\\
& \,  +C\sqrt{N}\chi\int_{t_1}^{t}(t-s)^{-\frac{1}{2}}e^{-(t-s)}\|u_{2}(s)\|\| u_{2}(s)-u_{1}(s)\|_{\infty}ds \nonumber\\
& \, +(a+1+b\sup_{0\leq \tau\leq T'}(\|u_{1}(\tau)\|_{\infty}+\|u_{2}(\tau)\|_{\infty}) ) \int_{t_1}^{t}e^{-(t-s)}\|u_{1}(s)-u_{2}(s)\|_{\infty}ds \nonumber\\
& \leq\|(u_1(t_1)-u_2(t_1))\|_\infty+  M\int_{t_1}^{t}(t-s)^{-\frac{1}{2}}e^{-(t-s)}\|u_{2}(s)\|\| u_{2}(s)-u_{1}(s)\|_{\infty}ds,
\end{align*}
where  $M= a+1+ (C\sqrt{N}\chi + b\sqrt{T'})\sup_{0\leq t\leq T'}(\|u_{1}(\tau)\|_{\infty}+\|u_{2}(\tau)\|_{\infty})<\infty.$
Let $t_1\to 0$, we have
\vspace{-0.1in}$$
\|u_{1}(t)-u_{2}(t)\|_{\infty}\le   M\int_{0}^{t}(t-s)^{-\frac{1}{2}}e^{-(t-s)}\|u_{2}(s)\|\| u_{2}(s)-u_{1}(s)\|_{\infty}ds.
$$
By Lemma \ref{Lem2} again, we get
$u_1(t)\equiv u_2(t)$ for all $0\leq t\leq T'$. Since $T'<T$ was arbitrary chosen, then $u_1(t)\equiv u_2(t)$ for all $0\leq t< T$. The theorem is thus proved.
\end{proof}

\subsection{Proof of Theorem \ref{Local existence2}}

In this subsection, we prove Theorem \ref{Local existence2}. Throughout this section,  we let
 $\alpha\in (\frac{1}{2}, 1), \delta\in[0,2\alpha-1)$ and $p>N$ such that $\frac{(2\alpha-1-\delta)p}{N}>1.$ Let $X=L^{p}(\R^{N})$ and
 $X^\alpha$ be the fractional power space of $\Delta-I$ on $X$. By the inequalities in $(\ref{Fractional power Imbedding})$,  we have the continuous inclusions
$$
X^{\alpha}\subset C^{1+\delta}(\mathbb{R}^{N})\quad \text{and}\quad X^{\alpha}\subset W^{1,lp}\ \ \ \forall\ l\geq 1.
$$

\begin{proof}[Proof of Theorem \ref{Local existence2}]
We prove this theorem using semigroup method.

First, consider the functions $B: X^{\alpha}\times X^{\alpha}\rightarrow L^{p}(\mathbb{R}^{N})$ and  $F : X^{\alpha}\rightarrow L^{p}(\mathbb{R}^{N})$ defined by
\begin{equation*}
B(u,v)= -\chi \nabla u\nabla(\Delta-I)^{-1}v -\chi u(\Delta-I)^{-1}v -(b-\chi)uv
\vspace{-0.1in}\end{equation*}
and
\vspace{-0.1in}\begin{align*}\label{Func}
F(u)= \chi \nabla u\nabla(\Delta-I)^{-1}u -\chi u(\Delta-I)^{-1}u -(b-\chi)u^{2}+(a+1)u
\end{align*} for every $u,v\in X^{\alpha}.$ Clearly, $B$ is a  bilinear function and $F(u)= B(u,u)+ (1+a)u$ for every $u\in X^{\alpha}.$  Since $X^{\alpha}$ is continuously embedded in $W^{1,p}(\R^{N})\cap C^{1+\delta}(\R^{N})$,  there is a constant ${C}>0$ such that
$$
\|w\|_{W^{1,p}(\R^{N})}+\|w\|_{C^{1+\delta}}\leq {C}\|w\|_{X^{\alpha}} \ \ \forall \ w\in X^{\alpha}.
$$
Combining this with  regularity and a priori estimates for elliptic equations, we obtain that
\begin{eqnarray*}
\| B(u,v) \|_{L^{p}}& \leq & \chi \|u\|_{C^{1}}\|(\Delta-I)^{-1}v\|_{W^{1,p}}+\chi \|u\|_{\infty}\|(\Delta-I)^{-1}v\|_{L^{p}} +(b-\chi)\|u\|_{\infty}\|v\|_{L^{p}}\nonumber\\
& \leq & 2\chi\|u\|_{C^{1}}\|(\Delta-I)^{-1}v\|_{W^{1,p}}+(b-\chi)\|u\|_{\infty}\|v\|_{L^{p}}\nonumber\\
& \leq & {C}\|u\|_{X^{\alpha}}\left(\|(\Delta-I)^{-1}v\|_{W^{1,p}}+\|v\|_{X^{\alpha}}\right)\nonumber\\
& \leq & {C}\|u\|_{X^{\alpha}}\|v\|_{X^{\alpha}}.
\end{eqnarray*}
Thus $B$ is continuous. Hence the function $F$ is locally Lipschitz continuous and maps bounded sets to bounded sets. It follows from Theorem \ref{Existence of Mild Soltion} and  Theorem \ref{Existence of Mild Soltion and maximal time} that there exists a maximal time $T_{\max}>0$ and a unique $u\in  C([0,  T_{\max}), X^{\alpha})$ satisfying the integral equation
$$
u(t)=T(t)u_{0} +\int_{0}^{t}T(t-s)(-\chi\nabla\cdot( u(s)\nabla v(s))+(a+1) u(s) -b u^{2}(s))ds,
$$
where $v(s)=(I-\Delta)^{-1}u(s)$. Moreover, $u\in C^1((0,T_{\max}),L^p(\R^N))$ and if $T_{\max}<\infty$,
then $\lim_{t\to T_{\max}}\|u(t)\|_{X^\alpha}=\infty$.

Next, by \cite[Theorem 3.5.2]{Dan Henry},
\begin{equation}
\label{local-2-eq2-support}
u\in C((0,T_{\max}),X^\beta)
\end{equation}
for any $0\le \beta< 1$.  Note that $X^{\alpha}$ is continuously embedded in $C^{1+\delta}(\R^N)$. Then by Theorem \ref{Local existence1} and \eqref{local-2-eq2-support}, we have that
 \eqref{local-2-eq2} and \eqref{local-2-eq3} hold.

Now, let $u(\cdot,t;u_0)=u(t)(\cdot)$ and $v(\cdot,t;u_0)=(I-\Delta)^{-1} u(\cdot,t;u_0)$. We have that
$(u(\cdot,\cdot;u_0)$, $v(\cdot,\cdot;u_0))$ is a nonnegative classical solution of \eqref{IntroEq1} on $[0,T_{\max})$ with initial
function $u_0$ and $u(\cdot,t;u_0)$ satisfies \eqref{local-2-eq1}, \eqref{local-2-eq2}, and \eqref{local-2-eq3}. The uniqueness of classical solutions
 of \eqref{IntroEq1} follows from the similar arguments as in the proof (iii) of
Theorem \ref{Local existence1}.
\end{proof}

\subsection{Proof of Theorem \ref{Local Existence3}}

In this subsection, we prove Theorem \ref{Local Existence3}.

\begin{proof}[Proof of Theorem \ref{Local Existence3}]
 The proof of this theorem follow the similar arguments used in the  proof of Theorem \ref{Local existence1}. Hence lengthy details will be avoided.
 In the following, we fix $u_0 \in L^{p}(\R^N) .$

\medskip

\noindent {\bf Claim 1.} There is $T_{\max}\in (0,\infty]$ such that \eqref{IVP'} has a mild solution $u(\cdot)$ on $[0,T_{\max})$ with
$\alpha=0$ and $X^0=L^p(\R^N)$;
 for every $0<\beta<\frac{1}{2}$, $(0,T_{\max})\ni t\mapsto u(\cdot)\in X^\beta$ is locally H\"older continuous; and  if $T_{\max}<\infty$, then $\lim_{t\to T_{\max}}\|u(t)\|_{L^p(\R^N)}=\infty$.

\smallskip

For every $T>0$ and $R>0,$ let us set
$$\mathcal{S}_{R,T}:=\left\{u \in C([0, T ], L^{p}(\R^N) )\ | \ \|u\|_{L^{p}(\R^N)} \leq R \right \}.$$
Note that $\mathcal{S}_{R,T}$ is a closed subset of the Banach space $C([0, T ]: L^{p}(\R^N) )$ with the sup-norm.

Define $G: \mathcal{S}_{R,T} \to C([0, T ], L^{p}(\R^N))$ by
\begin{align*}
(Gu)(t)=& T(t)u_{0} +\chi\int_{0}^{t} \underbrace{T(t-s)\nabla\cdot ( u(s)  \nabla (\Delta-I)^{-1}u(s))}_{F_2(s)}ds\\
& +(1+a)\int_{0}^{t} T(t-s)u(s)ds-b\int_{0}^{t} \underbrace{T(t-s)u^{2}(s)}_{F_1(s)}ds.
\end{align*}
 It is clear that the Riemann integral $\int_0^tT(t-s)u(s)ds$ in $L^p(\R^N)$ exists for every $u\in \mathcal{S}_{R,T}$ and $t\in [0 , T]$.
Let $u\in\mathcal{S}_{R,T}$ and $0<t\leq T$ be given. For every $s_1, s_2\in[0, t]$, we have that
\begin{align*}
&\|u(s_1)\nabla(\Delta-I)^{-1}u(s_1)-u(s_2)\nabla(\Delta-I)^{-1}u(s_2)\|_{L^p(\R^N)}\nonumber\\
& \leq \|u(s_1)-u(s_2)\|_{L^p(\R^N)}\|(\Delta-I)^{-1}u(s_1)\|_{C^{1,b}_{\rm unif}(\R^N)}\nonumber\\
&+\|u(s_2)\|_{L^p(\R^N)}\|(\Delta-I)^{-1}(u(s_2)-u(s_1))\|_{C^{1,b}_{\rm unif}(\R^N)} \nonumber
\end{align*}
Since $p>N$, by  regularity and a priori estimates for elliptic equations, there is a constant $C>0$ such that
$$
\|(\Delta-I)^{-1}w\|_{C^{1,b}_{\rm unif}(\R^N)}\leq C\|w\|_{L^{p}(\R^N)}, \quad \forall\ w\in L^{p}(\R^N).
$$
Thus we have that
\begin{align*}
&\|u(s_1)\nabla(\Delta-I)^{-1}u(s_1)-u(s_2)\nabla(\Delta-I)^{-1}u(s_2)\|_{L^p(\R^N)}\nonumber\\
& \leq C(\|u(s_1)\|_{L^p(\R^N)}+\|u(s_2)\|_{L^p(\R^N)})\|u(s_1)-u(s_{2})\|_{L^p(\R^N)}\nonumber\\
& \leq 2CR\|u(s_1)-u(s_{2})\|_{L^p(\R^N)},\ \forall\ 0\leq s_1,s_2\leq t.\nonumber
\end{align*}
Hence the function $
[0,t]\ni s\mapsto u(s)\nabla(\Delta-I)^{-1}u(s)\in L^{p}(\R^N)$ is continuous. Similar arguments as  Theorem \ref{Local existence1} Claim 1 yield that the the function $F_{2} : [0, t)\to L^{p}(\R^N)$ is continuous and the Riemann integral $\int_0^tF_2(s)ds$ in $L^p(\R^N)$ exists.  Next, for every $0<\varepsilon<t$ and $s\in [0, t-\varepsilon]$, we have
\begin{align*}
F_1(s)=T(t-\varepsilon-s)T(\varepsilon)u^2(s)
\end{align*}
and by \eqref{Lp Estimates}, \eqref{EqL_Infty_001} and $p\ge 2$, and H\"older's inequality, we have
\begin{align*}
\|T(\varepsilon)u^2(s_1)-T(\varepsilon)u^2(s_2)\|_{L^p(\R^N)} & \leq C\varepsilon^{-\frac{N}{2p}}e^{-\varepsilon}\|u^{2}(s_1)-u^{2}(s_2)\|_{L^{\frac{p}{2}}(\R^N)}\nonumber\\
& \leq C\varepsilon^{-\frac{N}{2p}}e^{-\varepsilon}\|u(s_1)-u(s_2)\|_{L^{p}(\R^N)}\|u(s_1)+u(s_2)\|_{L^{p}(\R^N)}\nonumber\\
& \leq 2RC\varepsilon^{-\frac{N}{2p}}e^{-\varepsilon}\|u(s_1)-u(s_2)\|_{L^{p}(\R^N)}, \ \forall s_{1},s_{2}\in[0, t-\varepsilon].\nonumber
\end{align*}
Thus the function $F_1 :[0, t)\to L^{p}(\R^N)$ is continuous. Moreover,  by \eqref{Lp Estimates}, \eqref{EqL_Infty_001}, $p>N$ and $p\ge 2$, we have
\begin{align*}
\int_{0}^{t}\|F_{1}(s)\|_{L^p(\R^N)}ds&\leq C\int_0^t (t-s)^{-\frac{N}{2p}}e^{-(t-s)}\|u^2\|_{L^{\frac{p}{2}}(\R^N)}ds\nonumber\\
&\leq CR\int_0^t (t-s)^{-\frac{N}{2p}}e^{-(t-s)}ds\leq CR\Gamma(1-\frac{N}{2p}).\nonumber
\end{align*}
Hence the Riemann integral $\int_0^tF_{1}(s)ds$ in $L^p(\R^N)$ exists. Therefore $(Gu)(t)$ is well defined and the integral is taken in $L^p(\R^N)$.

For every $R>\|u_{0}\|_{p}$,  there exists $T:=T(R)$ such that  $G$ maps $\mathcal{S}_{R,T}$ into itself.
Indeed, for every $u \in \mathcal{S}_{R,T},$ we have
\begin{align*}
\|G(u)(t)\|_{L^{p}(\R^N)} & \leq  e^{-t}\|u_{0}\|_{L^{p}(\R^N)}+\chi\int_0^t\| T(t-s)\nabla\cdot ( u(s)  \nabla (\Delta-I)^{-1}u(s))\|_{ L^{p}(\R^N)}ds \nonumber\\
&\,\,\,   +(1+a)\int_{0}^{t} e^{-(t-s)}\|u(s)\|_{L^{p}(\R^N)}ds+b\int_{0}^{t}\|T(t-s)u^{2}(s)\|_{L^{p}(\R^N)}ds.
\end{align*}
Now, by \eqref{Lp Estimates}, \eqref{EqL_Infty_001} and $p\ge 2$, the last inequality can be improved to
\begin{align}\label{P00}
\|G(u)(t)\|_{L^{p}(\R^N)}
& \leq  e^{-t}\|u_{0}\|_{L^{p}(\R^N)}+ {C}\int_{0}^{t}(t-s)^{-\frac{1}{2}}e^{-(t-s)}\|u(s)\nabla(\Delta-I)^{-1}u(s)\|_{L^{p}(\R^N)}ds \nonumber\\
 & \,\,  +(a+1)\int_{0}^{t}e^{-(t-s)}\|u(s)\|_{L^{p}(\R^N)} + {C}\int_{0}^{t}(t-s)^{-\frac{N}{2p}}e^{-(t-s)}\|u^{2}(s)\|_{L^{\frac{p}{2}}}ds \nonumber\\
 & \leq  e^{-t}\|u_{0}\|_{L^{p}(\R^N)}+ {C}\int_{0}^{t}(t-s)^{-\frac{1}{2}}\|u(s)\|_{L^{p}(\R^N)}\|(\Delta-I)^{-1}u(s)\|_{C_{unif}^{1,b}(\R^N)}ds \nonumber\\
 & \,\, +(a+1)Rt+ {C}R^{2}\int_{0}^{t}(t-s)^{-\frac{N}{2p}}ds.
\end{align}
Since $p>N$, by  regularity and a priori estimates for elliptic equations, there is a constant $C>0$ such that
$$
\|(\Delta-I)^{-1}w\|_{C^{1,b}_{\rm unif}(\R^N)}\leq C\|w\|_{L^{p}(\R^N)}, \quad \forall\ w\in L^{p}(\R^N).
$$
Combining this with \eqref{P00}, we obtain that
\begin{eqnarray*}\label{P01}
\|G(u)(t)\|_{L^{p}(\R^N)} & \leq & e^{-t}\|u_{0}\|_{L^{p}(\R^N)}+ {C}R^{2}t^{\frac{1}{2}} +(a+1)Rt+ {C}R^{2}t^{1-\frac{N}{2p}}.
\end{eqnarray*}
Hence we can now chose $T>0$ such that
$$
e^{-t}\|u_{0}\|_{L^{p}(\R^N)}+ {C}R^{2}t^{\frac{1}{2}} +(a+1)Rt+ {C}R^{2}t^{1-\frac{N}{2p}}<R\quad \forall\,\, t\in [0,T].
$$
This implies that $G$ maps $\mathcal{S}_{R,T}$ into itself.

Using again inequalities \eqref{Lp Estimates} and inequality \eqref{EqL_Infty_001}, by following the same ideas as in the proof of claim 3 in Theorem \ref{Local existence1}, we have that $G$ is a contraction map for $T$ sufficiently small.
Thus G has a unique fixed point, say $u(\cdot)$. Using again Lemma \ref{L_Infty bound}, similar arguments used in the proof of Claim 2 of  Theorem \ref{Local existence1} yields that the function $ (0, T)\ni t\to u(t)\in X^{\beta}$ is locally H\"older continuous for every $0<\beta<\frac{1}{2}$. Clearly, $u(\cdot)$ is a mild solution of
\eqref{IVP'} with $\alpha=0$ and $X^0=L^p(\R^N)$.
 The claim then follows from regular extension arguments.

 \medskip

 \noindent {\bf Claim 2.} $u(t)$ obtained in Claim 1 is the unique classical solution of \eqref{IVP'} on $[0,T_{\max})$ satisfying
 \eqref{local-3-eq1}, \eqref{local-3-eq2}, and \eqref{local-3-eq3}.

 By Claim 1, for any $0<\beta<1/2$ and $t_1\in (0,T_{\max})$, $u(t_1)\in X^\beta$. It then follows that $u_1(\cdot)=u(t_1)(\cdot)\in C^\nu_{\rm unif}(\R^N)\cap L^p(\R^N)$ for
 $0<\nu\ll 1$.
 Consider  \eqref{aux-eq-for-thm1}. By the similar arguments as those in the proof (ii) of Theorem \ref{Local existence1},  $u(\cdot)$ is a classical solution of \eqref{IVP'} on $[0,T_{\max})$. Moreover,
 $u(\cdot)\in C^1((0,T_{\max}),C_{\rm unif}^b(\R^N))$ and satisfies \eqref{local-3-eq3}.
By Theorem \ref{more-on-existence-of-mild-solution}, we have $u(\cdot)\in C((0,T_{\max}),X^\beta)$ for any $0\le \beta<1$. Hence $u$ satisfies \eqref{local-3-eq2}.
By the similar arguments as those in the proof (iii) of Theorem \ref{Local existence1}, we can prove the uniqueness and then the claim follows.
\medskip

\noindent {\bf Claim 3.} The function $u(t)$ obtained in Claim 1 is  nonnegative.

We have that the function $u$ satisfies the integral equation
$$u(t)=T(t)u_{0}+\chi\int_{0}^{t}T(t-s)\nabla\cdot (u\nabla(\Delta-I)^{-1}u(s))ds + \int_{0}^{t}T(t-s)((a+1)u(s)-bu^{2}(s))ds.
$$
Since $u_{0}\geq 0$, there is a sequence of nonnegative functions $\{u_{0n}\}_{n}\in L^{p}(\R^N)\cap C^{b}_{\rm unif}(\R^N)$ such $\|u_{0n}-u_{0}\|_{L^{p}(\R^N)}\rightarrow 0$ as $n\rightarrow \infty$. For $R$ large enough, since $\sup_{n}\|u_{0n}\|_{L^{p}(\R^N)}<\infty$, the time T can be chosen to be independent of n, such for each n, there is a unique $u_{n}(\cdot)\in \mathcal{S}_{R,T}$ satisfying the integral equation
\begin{equation*}
u_{n}(t)=T(t)u_{0n}+\chi\int_{0}^{t}T(t-s)\nabla\cdot (u_{n}\nabla(\Delta-I)^{-1}u_{n}(s))ds + \int_{0}^{t}T(t-s)((a+1)u_{n}(s)-bu^{2}_{n}(s))ds
\end{equation*}
for every $n$. Since the $u_{n}\geq 0$ and belongs to $C^{b}_{\rm unif}(\R^N)$, for every $n$,  Theorem \ref{Local existence1} implies that $u_{n}(t)\geq 0$. Now, similar arguments used to establish \eqref{correct05}, yield  as similar result
\begin{align*}
\|u_{n}(t)-u(t)\|_{L^{p}(\R^N)} & \leq  \|T(t)(u_{0n}-u_{0})\|_{L^{p}(\R^N)}+C\int_{0}^{t}(t-s)^{-\frac{N}{2p}}\|u_{n}(s)-u(s)\|_{L^{p}(\R^N)}ds\nonumber\\
& \leq  \|u_{0n}-u_{0}\|_{L^{p}(\R^N)}+C\int_{0}^{t}(t-s)^{-\frac{N}{2p}}\|u_{n}(s)-u(s)\|_{L^{p}(\R^N)}ds.
\end{align*}
Next, using Lemma \ref{Lem2}, it follows from the last inequality that
\begin{equation*}
\|u_{n}(t)-u(t)\|_{L^{p}(\R^N)}\leq C\|u_{0n}-u_{0}\|_{L^{p}(\R^N)}
\end{equation*}
for every $n\geq 1,\ 0< t< T$, where  $C>0$ is positive constant independent of $n$.
Letting n goes to infinity, we obtain that
\[
\|u_{n}(t)-u(t)\|_{L^{p}(\R^N)}\rightarrow 0 \quad \quad \text{as}\quad n\rightarrow 0.
\]
Thus, for every $t>0$,  $u(t)(x)\geq 0$ for a. e. $x\in\R^N$. Since $u(t)(x)$ is continuous in $x\in\R^N$ for each $t>0$, we conclude that $u(t)(x)\geq 0$ for every $x\in\R^N$, $ t\in (0,  T]$.

Let $u(\cdot,t;u_0)=u(t)(\cdot)$ and $v(\cdot,t;u_0)=(I-\Delta)^{-1} u(\cdot,t;u_0)$. We then have that
$(u(\cdot,\cdot;u_0)$, $v(\cdot,\cdot;u_0))$ is a unique nonnegative classical solution of \eqref{IntroEq1} on $[0,T_{\max})$ with initial
function $u_0$ and $u(\cdot,t;u_0)$ satisfies \eqref{local-3-eq1}, \eqref{local-3-eq2}, and \eqref{local-3-eq3}.
\end{proof}

\subsection{Proof of Theorem \ref{Lp-Bound1}}

In this section, we prove Theorem \ref{Lp-Bound1}.

\begin{proof}[Proof of Theorem \ref{Lp-Bound1}]
(1) Let $u(t)=u(\cdot,t;u_0)$. It is clear that $T_{\max}^\alpha(u_0)\le T_{\max}^\infty(u_0)$. Assume that $T_{max}^\alpha(u_0)<T_{\max}^\infty(u_0)$.
 Then $T_{\max}^\alpha(u_0)<\infty$. Recall that  for every $t\in (0, T_{max}^\alpha(u_0)),$
\begin{align}\label{new-Eq6_0}
u(t)=&T(t)u_{0}-\chi\underbrace{\int_{0}^{t}T(t-s)(\nabla u\nabla v)ds}_{I_{1}}-\chi\underbrace{\int_{0}^{t}T(t-s)(uv)ds}_{I_{2}}\nonumber\\
 &\quad + \underbrace{\int_{0}^{t}T(t-s)((a+1)u+(\chi-b)u^{2})}_{I_{3}}ds.
\end{align}
We investigate the $X^{\alpha}$-norm of each term in the right hand side of $(\ref{new-Eq6_0})$ for $0\le t< T_{\max}^\alpha(u_0)$.  It follows from inequalities $(\ref{Lp Estimates})$ that
\begin{equation}\label{new-Eq6_1}
\|T(t)u_{0}\|_{X^{\alpha}}\leq C_{\alpha} e^{-t}t^{-\alpha}\|u_{0}\|_{L^{p}}.
\end{equation}
  Using inequalities $(\ref{Lp Estimates})$, we have that %and H\"older's inequality, we have
\begin{align*}
\|I_{1}\|_{X^{\alpha}}& \leq  \int_{0}^{t}\|(I-\Delta)^{\alpha}T(t-s)(\nabla u\nabla v)\|_{L^{p}}ds \leq  C_{\alpha}\int_{0}^{t}e^{-(t-s)}(t-s)^{-\alpha}\|\nabla u\nabla v\|_{L^{p}}ds \nonumber \\
& \leq  C_{\alpha}\int_{0}^{t}e^{-(t-s)}(t-s)^{-\alpha}\|u\|_{C^{1}(\R^{N})}\|\nabla v\|_{L^{p}}ds.
\end{align*}
Since $v=(\Delta-I)^{-1}u,$ elliptic regularity implies that
\begin{equation*}\label{new-Eq7_2}
\|\nabla v\|_{L^p(\R^N)}\leq  C \|u\|_{X^\alpha}.
\end{equation*}
We then have
\begin{eqnarray*}\label{new-Eq8}
\|I_{1}\|_{X^{\alpha}}
& \leq & C \sup_{0\le \tau\le T_{\max}^\alpha(u_0)} \|u(\tau)\|_{C^1} \int_{0}^{t}e^{-(t-s)}(t-s)^{-\alpha}\| u(s)\|_{X^{\alpha}}ds.
\end{eqnarray*}
Similar arguments applied to $I_{2}$ and $I_{3}$ yield that
\begin{eqnarray*}\label{new-Eq9}
\|I_{2}\|_{X^{\alpha}}
& \leq &   C \sup_{0\le \tau\le T_{\max}^\alpha(u_0)} \|u(\tau)\|_\infty  \int_{0}^{t}e^{-(t-s)}(t-s)^{-\alpha}\| u(s)\|_{X^{\alpha}}ds
\end{eqnarray*}
and
\begin{eqnarray*}\label{new-Eq10}
\|I_{3}\|_{X^{\alpha}}
& \leq &  C (a+1+ \sup_{0\le \tau\le T_{\max}^\alpha(u_0)} \|u(\tau)\|_\infty)\int_{0}^{t}e^{-(t-s)}(t-s)^{-\alpha}\| u(s)\|_{X^{\alpha}}ds.
\end{eqnarray*}
We then have  that for every $t\in (0, T_{max}^\alpha(u_0))$
\begin{equation*}\label{Eq10_1}
\|u(t)\|_{X^{\alpha}}\leq  M(u_0)\left[t^{-\alpha}+\int_{0}^{t}(t-s)^{-\alpha}\|u(s)\|_{X^{\alpha}}\right],
\end{equation*}
where
\begin{equation*}
M(u_0)=
{C}(\|u_0\|_{L^p(\R^N)}+a+1+\sup_{0\le \tau \le T_{\max}^\alpha(u_0)}\|u(\tau)\|_{C^1}+\sup_{0\le \tau\le T_{\max}^\alpha(u_0)} \|u(\tau)\|_\infty).
\end{equation*}
Thus, it follows from Lemma $\ref{Lem2}$ that
\begin{equation*}\label{Eq10_2}
\|u(t)\|_{X^{\alpha}}\leq {C} M(u_0) t^{-\alpha}  \ \ \forall\ t\in(0, T_{\max}^\alpha(u_0)).
\end{equation*}
This implies that $\limsup_{t\to T_{\max}^\alpha(u_0)}\|u(t)\|_{X^\alpha}<\infty$, a contradiction. Therefore,
$T_{\max}^\alpha(u_0)=T_{\max}^\infty(u_0)$.

\smallskip

\smallskip
(2) It is clear that $T_{\max}^\alpha(u_0)\le T_{\max}^p(u_0)$. Assume that $T_{\max}^\alpha(u_0)<T_{\max}^p(u_0)$. Then $T_{\max}^\alpha(u_0)<\infty$.
By \eqref{local-3-eq2}, $\limsup_{t\to T_{\max}^\alpha(u_0)}\|u(\cdot,t;u_0)\|_{X^{\alpha}}<\infty$, a contradiction. (2) then follows.

\smallskip
\smallskip

(3) By the arguments in (1), we have $T_{\max}^p(u_0)\ge T_{\max}^\infty(u_0)$. By \eqref{local-3-eq2},
$T_{\max}^\infty(u_0)\ge T_{\max}^p(u_0)$. (3) then follows.

\smallskip\smallskip

(4) Let $u(\cdot,t;u_0)$ be as in Theorem \ref{Local Existence3}. For given $T>0$ and $R>\|u_0\|_{L^1(\R^N)}+\|u_0\|_{L^p(\R^N)}$,  consider the set   $$
\mathcal{S}'_{R,T}:=\{ u\in C([0,T],L^{1}(\R^N)\cap L^{p}(\R^N))\,|\,  \|u\|_{L^{1}(\R^N)\cap L^{p}(\R^N)}\leq R\},
$$
where $\|u\|_{L^{1}(\R^N)\cap L^{p}(\R^N)}=\|u\|_{L^1(\R^N)}+\|u\|_{L^p(\R^N)}.$
Using inequality \eqref{Lp Estimates}, for $u\in \mathcal{S}_{R,T}^{'}$, we have that
\begin{align*}
\left\|\int_{0}^{t}T(t-s)u^{2}ds \right\|_{L^{1}(\R^N)} & \leq  \int_{0}^{t}\|T(t-s)u^{2}(s)\|_{L^{1}(\R^N)}ds\leq  C\int_{0}^{t}e^{-(t-s)}\|u^2(s)\|_{L^{1}}ds\\
  &= C\int_{0}^{t}e^{-(t-s)}\|u(s)\|^{2}_{L^{2}}ds.
\end{align*}
Now, since $1<2\leq p$, Holder's inequality implies that $\|u(s)\|_{L^2}\leq \|u(s)\|_{L^1}^{\lambda}\|u(s)\|_{L^p}^{1-\lambda}\leq R$ with $\lambda=(\frac{1}{2}-\frac{1}{p})/(1-\frac{1}{p})$. Thus, the last inequality becomes
\begin{equation*}
\left\|\int_{0}^{t}T(t-s)u^{2}ds \right\|_{L^{1}(\R^N)}  \leq  CR^{2}\int_{0}^{t}e^{-(t-s)}ds \leq  CR^{2}t.
\end{equation*}
This together with the arguments in Claim 1 of Theorem \ref{Local Existence3}  implies that
$$G: \mathcal{S}^{'}_{R,T} \to C([0, T ], L^{p}(\R^N))\cap C([0,T], L^1(\R^N))$$
 is well defined, where
\begin{align*}
(Gu)(t)=& T(t)u_{0} +\chi\int_{0}^{t} T(t-s)\nabla\cdot ( u(s)  \nabla (\Delta-I)^{-1}u(s))ds\\
& +(1+a)\int_{0}^{t} T(t-s)u(s)ds-b\int_{0}^{t} T(t-s)u^{2}(s)ds.
\end{align*}
By the arguments in Claim 1 of Theorem \ref{Local Existence3}, $u(\cdot,t;u_0)\in C([0,T],L^1(\R^N))$ for $0<T\ll 1$.

Let
$$
T_{\max}^{p,1}(u_0)=\sup\{\tau \in [0,T_{\max}^p(u_0))\,|\, \sup_{0\le t<\tau}\|u(\cdot,t;u_0)\|_{L^1(\R^N)}<\infty\}.
$$
Assume that $T_{\max}^{p,1}(u_0)<T_{\max}^p(u_0)$. Then
 $\sup_{0\le t<T_{max}^{p,1}(u_0)}\|u(\cdot,t;u_0)\|_{L^1(\R^N)}=\infty$.
 Fix any $t_1\in (0,T_{\max}^{p,1}(u_0))$. By Theorem \ref{Local Existence3}, $u_1=u(\cdot,t_1;u_0)\in L^p(\R^N)\cap C_{\rm unif}^b(\R^N)$ and then
$$u(\cdot,\cdot;u_0),\partial_{x_i}u(\cdot,\cdot;u_0)\in C([t_1,T_{\max}^{p,1}(u_0)],C_{\rm unif}^b(\R^N),\quad i=1,2,\cdots,N.$$
Using the arguments in (1) with $u_0$ being replaced by $u_1$ and  $p=1$, $\alpha=0$,   we have
$$
\limsup_{t\to T_{\max}^{p,1}-t_1} (u_0)\|u(\cdot,t;u_1)\|_{L^1(\R^N)}<\infty.
$$
Note that $u(\cdot,t+t_1;u_0)=u(\cdot,t;u_1)$. We then have
$$
\limsup_{t\to T_{\max}^{p,1}(u_0)}\|u(\cdot,t;u_0)\|_{L^1(\R^N)}<\infty,
$$
which is a contradiction. Therefore, $T_{\max}^{p,1}(u_0)=T_{\max}^p(u_0)$.
\end{proof}

\section{Global existence of classical solutions}
In this section, we discuss  the existence of  global in time solutions to \eqref{IntroEq1} and prove Theorems \ref{Main Theorem 0}-\ref{Main Theorem 2}.
Throughout this section, $C$  denotes a constant independent of the initial functions and the solutions under consideration,  unless specified otherwise.

We first recall a well known  lemma for a logistic ODE for convenience and then prove Theorems \ref{Main Theorem 0}, \ref{Main Theorem 1}, and \ref{Main Theorem 2} in subsections 4.1, 4.2, and 4.3, respectively.

\begin{lem}
\label{global-existence-lemma}
Consider the ODE
\begin{equation}
\label{global-stability-ode}
\dot u=u(a_0-b_0 u),
\end{equation}
where $a_0,b_0$ are positive constants. Let $u(t;u_0)$ be the solution of \eqref{global-stability-ode} with $u(0;u_0)=u_0\in\R$. Then
for any $u_0>0$,
$$
\lim_{t\to\infty} u(t;u_0)=\frac{a_0}{b_0}.
$$
\end{lem}

\subsection{Proof of Theorem \ref{Main Theorem 0}}

In this subsection, we prove  Theorem \ref{Main Theorem 0}.

\begin{proof}[Proof of Theorem \ref{Main Theorem 0}]
Let $(u,v)$ be the classical local nonnegative solution given by Theorem \ref{Local existence1} defined on the maximal interval $[0, T_{\max}^\infty(u_0)).$ We have that
\begin{eqnarray*}
u_{t}& = &\De u -\chi\nabla u\nabla v-(b-\chi)u^{2}+au -\chi uv\nonumber\\
 & \leq & \De u -\chi\nabla u\nabla v-(b-\chi)u^{2}+au.
\end{eqnarray*}
%We define the operator $Lu=\De u -\chi\nabla u\nabla v-(b-\chi)u^{2}+au.$
Let $u(t, \|u_{0}\|_{\infty})$ be solution of  the initial value problem,
\begin{equation*}\label{ODE1}
\begin{cases}
u'= -(b-\chi)u^{2}+au \\
u(0)=\max{u_{0}}.
\end{cases}
\end{equation*}
Since $b-\chi\geq 0$, then $u(t, \|u_{0}\|_{\infty})$ is globally defined in time. Since  $u_{0}\leq u(0, \|u_{0}\|_{\infty})$ , by the comparison principle for parabolic equations we have that
\begin{equation}\label{Eq_Cor1}
u(x,t)\leq u(t, \|u_{0}\|_{\infty})
\end{equation}
 for all $x\in\R^N$ and $t\geq 0.$ Hence $u(x,t)$ is globally defined in time.  Furthermore, by Lemma \ref{global-existence-lemma}, if $ \chi<b$, then
\begin{equation}\label{Eq_Cor2}
u(t, \|u_{0}\|_{\infty})\rightarrow \frac{a}{b-\chi}\ \ \ \text{as} \ \ \ t\rightarrow \infty.
\end{equation}
This completes the proof of the theorem.
\end{proof}

\subsection{Proof of Theorem \ref{Main Theorem 1}}

In this subsection, we prove Theorem \ref{Main Theorem 1}. In order to do so,  we first prove an important theorem and some technical lemmas.
Throughout this section,  we let
 $\alpha\in (\frac{1}{2}, 1), \delta\in[0,2\alpha-1)$ and $p>N$ such that $\frac{(2\alpha-1-\delta)p}{N}>1.$ Let $X=L^{p}(\R^{N})$ and
 $X^\alpha$ be the fractional power space of $\Delta-I$ on $X$.

\begin{lem}\label{Lp bound}
Suppose that  satisfies the hypothesis of Theorem \ref{Local existence2} and $(u,v)$ is the solution of \eqref{IntroEq1} as in Theorem
\ref{Local existence2}. For every $r\geq 1$ satisfying
\begin{equation*}\label{Eq1_Lp bound}
r\leq \frac{\chi}{(\chi-b)_{+}},
\end{equation*}
we have that
\begin{equation}\label{Eq2_Lp bound}
\|u(\cdot,t)\|_{L^{r}(\R^N)}\leq \|u_{0}\|_{L^{r}(\R^N)}e^{at} \ \ \ \forall \ t\in[0, \ T_{max}^\alpha(u_0)).
\end{equation}
\end{lem}

\begin{proof} Let $1\leq r\leq \chi/(\chi-b)_{+}$. If $\|u_{0}\|_{L^{r}(\R^N)}=\infty$ there is nothing show. Hence we might suppose that $\|u_{0}\|_{ L^{r}(\R^N)}<\infty$. Let us set  $\delta_{r}:=b-\frac{\chi(r-1)}{r}\geq 0.$ We multiply the first equation in \eqref{IntroEq1} by $u^{r-1},$ and   integrating it, we obtain
\begin{align*}\label{Eq1}
\frac{1}{r}\frac{d}{dt}\int_{\R^N}u^{r} &=  - \int_{\R^N}\nabla u\nabla u^{r-1} + \frac{\chi(r-1)}{r}\int_{\R^N}\nabla u^{r}\nabla v  + \int_{\R^N}(au^{r} -bu^{r+1})\nonumber\\
&=-(r-1) \int_{\R^N} u^{r-2}|\nabla u|^{2} - \frac{\chi(r-1)}{r}\int_{\R^N} u^{r}\De v + \int_{\R^N}(a -bu)u^{r}\nonumber\\
&=-(r-1) \int_{\R^N} u^{r-2}|\nabla u|^{2} -\frac{\chi(r-1)}{r}\int_{\R^{N}}u^{r}v -\underbrace{(b- \frac{\chi(r-1)}{r})}_{\de_{r}}\int_{\R^N} u^{r+1} + \int_{\R^N}a u^{r}\nonumber\\
&\leq   a\int_{\R^N} u^{r}.
\end{align*}
\eqref{Eq2_Lp bound} then follows.
\end{proof}

A natural question that one could ask is under which condition on the expression $\frac{\chi}{(\chi-b)_{+}},$ the $L^{r}-$ a priori estimate in Lemma \ref{Lp bound} can be extended to all $r\geq 1$ or for at least for every $r=p$. An obvious condition would be to require that $p \leq \frac{\chi}{(\chi-b)_{+}}$ so that the hypothesis of Lemma \ref{Lp bound} are satisfied . Hence Lemma \ref{Lp bound} and Theorem \ref{Lp-Bound1}  have a direct consequence that we formulate in the next result.

\medskip

\begin{coro}\label{corol}
Suppose that $u_{0}$ satisfies the hypothesis of Theorem \ref{Local existence2} and $p\leq \frac{\chi}{(\chi-b)_{+}}.$ Then the solution $(u,v)$ of \eqref{IntroEq1} with initial data $u_0$  is global in time.
\end{coro}

\begin{proof}We have that $p\leq \frac{\chi}{(\chi-b)_{+}},$  hence according to Lemma \ref{Lp bound}  we have
\begin{equation}\label{Eq_3 Lp bound} \|u(t)\|_{L^{p}(\R^N)} \leq \|u_{0}\|_{L^{p}(\R^N)}e^{at},\quad \ \quad \forall \,\, 0\le t <T_{\max}^\alpha(u_0).
\end{equation}
By  Theorem \ref{Lp-Bound1}, we have  that $T_{max}^\alpha(u_0)=T_{\max}^p(u_0)=+\infty.$
\end{proof}

\medskip

Note that $p$ was chosen to be strictly greater than $N.$ Thus, it would be nice to find a relationship between $N$ and the expression $\frac{\chi}{(\chi-b)_{+}}$ that will guarantee the existence of a global solution.  Theorem \ref{Main Theorem 1} provides such a sufficient condition to obtain a global in time solution. Note that this condition is weaker than the one giving by Corollary $\ref{corol}$.

\begin{proof}[Proof of Theorem \ref{Main Theorem 1}]

 We divide the proof into two steps. In the first part, we prove that  $L^{r}-$ norms of the $u(t)$ can be bounded by continuous function as required in Theorem \ref{Lp-Bound1}.  We then conclude that $T_{max}:=T_{\max}^\alpha(u_0)=+\infty.$ The last part is the proof of inequality $(\ref{Eq2_Main Theorem})$.

Note that $\frac{\chi}{(\chi-b)_{+}}>1,$ so we can choose $q_{1}\in \left(\max\{1, \frac{N}{2}\}, \min\{p,\ \frac{\chi}{(\chi-b)_{+}}\}\right).$ We have $\de_{q_{1}}=b-\frac{\chi(q_1-1)}{q_1}>0$, and Lemma \ref{Lp bound}, implies that
$$
\|u(\cdot,t)\|_{L^{q_{1}}(\R^{N})}\leq \|u_{0}\|_{L^{q_{1}}(\R^{N})}e^{at} \ \ \forall\ t\in [0, \ T_{max}).
$$

\noindent \textbf{Step 1.}\label{Main_Tm Lem1} We claim that for all $r\geq q_{1}$,
\begin{equation}\label{Step 1:}\|u(t)\|_{L^{r}(\R^N)}\leq \left[\|u_{0}\|_{L^{r}(\R^N)}+ K_{r}^{\frac{1}{r}}t^{\frac{1}{r}}\|u_{0}\|_{L^{q_{1}}(\R^N)}^{\frac{\lambda_{r}}{r}}e^{(\lambda_{r} -1)at} \right]e^{at} \ \ \ \forall \ t\in[0, \ T_{max}),\ \end{equation} where $K_{r}$ and $\lambda_{r}$ are nonnegative real numbers depending on $a,\ b,\ \chi,\ r$ and $N$ with $\lambda_{r}>1.$

 We multiply the first equation in \eqref{IntroEq1} by $u^{r-1},$ and after integrating it by part, we obtain that
\begin{eqnarray}\label{Eq2}
\frac{1}{r}\frac{d}{dt}\int_{\R^N}u^{r}
 &\leq &-(r-1) \int_{\R^N}u^{r-2}|\nabla u|^{2}-(b-\chi\frac{(r-1)}{r})\int_{\R^N}u^{r+1} + a\int_{\R^N}u^{r} \nonumber\\
 &=&-\frac{4(r-1)}{r^{2}}\int_{\R^N}|\nabla u^{\frac{r}{2}}|^{2}-(b-\chi\frac{(r-1)}{r})\int_{\R^N}u^{r+1} + a\int_{\R^N}u^{r}.
\end{eqnarray}
From Lemma $\ref{Pre1}$ it follows that
$$
\|u(t)\|_{L^{r+1}(\R^N)}\leq C_{0}^{\frac{2}{r\beta}}\|u(t)\|_{L^{q_{1}}(\R^N)}^{1-\theta}\|\nabla u^{\frac{r}{2}}\|_{L^{2}(\R^N)}^{\frac{2\theta}{r}},
$$
where
$$
\theta =\frac{r}{2}\left(\frac{\frac{1}{q_{1}}-\frac{1}{r+1}}{\frac{r}{2q_{1}}+\frac{1}{N}-\frac{1}{2}}\right),
$$
$$
\beta=\begin{cases}\frac{r+1-\frac{r}{2}}{r+1-q_{1}}\left(  \frac{2q_{1}}{r}+\left( 1-\frac{2q_{1}}{r}\right)\frac{2N}{N+2}\right) \quad \text{if }\ r>2q_{1}\\
1 \hspace{5.25cm} \text{if }\ r\leq 2q_{1},
\end{cases}$$
and $C_{0}$ depends only on $N.$ Notice that we used $a=3$ in Lemma $\ref{Pre1}.$
Hence
\begin{equation}\label{Eq3}
\int_{\R^{N}}u(t)^{r+1}\leq C_{0}^{\frac{2(r+1)}{r\beta}}\|u(t)\|_{L^{q_{1}}(\R^N)}^{(1-\theta)(r+1)}\|\nabla u^{\frac{r}{2}}(t)\|_{L^{2}(\R^N)}^{\frac{2\theta(r+1)}{r}}.
\end{equation}
Observe from the choice of $q_{1}$ that
$$
\frac{\theta(r+1)}{r}=\frac{\frac{r+1}{q_{1}}-1}{\frac{r}{q_{1}}+\frac{2}{N}-1}=\frac{\frac{r}{q_{1}}-1+ \frac{1}{q_{1}}}{\frac{r}{q_{1}}-1 + \frac{2}{N}}<1.
$$
Combining this with $(\ref{Eq3})$ and using Young's inequality, for every $\varepsilon>0,$ we have that
\begin{equation*}
\int_{\R^N}u(t)^{r+1}\leq\varepsilon\|\nabla u^{\frac{r}{2}}\|_{L^{2}(\R^N)}^{2} + \varepsilon^{-\frac{\theta(r+1)}{r-\theta(r+1) }} C_{0}^{\frac{2(r+1)}{(r-\theta(r+1))\beta}}\|u(t)\|_{L^{q_{1}}(\R^N)}^{\frac{(1-\theta)(r+1)r}{r-\theta(r+1)}},
\end{equation*}
which is equivalent to
\begin{equation}\label{Eq4}
-\|\nabla u^{\frac{r}{2}}\|_{L^{2}(\R^N)}^{2}\leq -\frac{1}{\varepsilon}\int_{\R^N}u(t)^{r+1} + \varepsilon^{-\left(1+\frac{\theta(r+1)}{r-\theta(r+1) }\right)} C_{0}^{\frac{2(r+1)}{(r-\theta(r+1))\beta}}\|u(t)\|_{L^{q_{1}}(\R^N)}^{\frac{(1-\theta)(r+1)r}{r-\theta(r+1)}}.
\end{equation}
Combining inequalities $(\ref{Eq2})$ and $(\ref{Eq4}),$ we obtain that
\begin{align*}
\frac{1}{r}\frac{d}{dt}\int_{\R^N}u^{r}\leq & -\underbrace{\left(\frac{4(r-1)}{\varepsilon r^{2}}+ b-\chi\frac{(r-1)}{r}\right)}_{\varepsilon_{r}}\int_{\R^N}u^{r+1}\\
 & + a\int_{\R^N}u^{r} +\frac{4(r-1)C_{0}^{\frac{2(r+1)}{(r-\theta(r+1))\beta}}}{r^{2}\varepsilon^{\left(1+\frac{\theta(r+1)}{r-\theta(r+1) }\right)}}\|u(t)\|_{L^{q_{1}}(\R^N)}^{\frac{(1-\theta)(r+1)r}{r-\theta(r+1)}}.  \nonumber\\
\end{align*}
If we choose $\varepsilon>0$ such that $\varepsilon_{r}\geq 0$ (for example $\varepsilon=\frac{4}{r\chi  }$ yields $\varepsilon_{r}=b$ ), we obtain that
\begin{eqnarray*}
\frac{1}{r}\frac{d}{dt}\int_{\R^N}u^{r}&\leq& a\int_{\R^N}u^{r} +\frac{4(r-1)C_{0}^{\frac{2(r+1)}{(r-\theta(r+1))\beta}}}{r^{2}\varepsilon_{r}^{\left(1+\frac{\theta(r+1)}{r-\theta(r+1) }\right)}}\|u(t)\|_{L^{q_{1}}(\R^N)}^{\frac{(1-\theta)(r+1)r}{r-\theta(r+1)}}.  \nonumber\\
&\leq& a\int_{\R^N}u^{r} +\frac{4(r-1)C_{0}^{\frac{2(r+1)}{(r-\theta(r+1))\beta}}}{r^{2}\varepsilon_{r}^{\left(1+\frac{\theta(r+1)}{r-\theta(r+1) }\right)}}\left[e^{at}\|u_{0}\|_{L^{q_{1}}(\R^N)}\right]^{\frac{(1-\theta)(r+1)r}{r-\theta(r+1)}}.  \nonumber\\
\end{eqnarray*}
It then follows by  Gronwall's inequality and the mean value theorem that
\begin{equation*}
\|u(t)\|_{L^{r}(\R^N)}^{r}\leq \left[\|u_{0}\|_{L^{r}(\R^N)}^{r}+ K_{r}\|u_{0}\|_{L^{q_{1}}(\R^N)}^{\lambda_{r}}te^{(\lambda_{r} -1)art} \right]e^{art} \ \ \ \forall \ t\in[0, \ T_{max}),
\end{equation*}
with
$$
\lambda_{r}= \frac{(1-\theta)(r+1)}{r-\theta(r+1)}
 \quad {\rm and}\quad  K_{r}=\frac{4(r-1)C_{0}^{\frac{2(r+1)}{(r-\theta(r+1))\beta}}}{r\varepsilon_{r}^{\left(1+\frac{\theta(r+1)}{r-\theta(r+1) }\right)}}.$$
 \eqref{Step 1:} then follows.

\medskip

\noindent \textbf{Step 2.} It follows from Theorem \ref{Local existence2}, Theorem \ref{Lp-Bound1} , Lemma \ref{Lp bound} and  Step 1 that \eqref{IntroEq1} has a unique global classical solution $(u,v)$. To complete the proof of this theorem, we need to prove the following estimate.
\begin{equation*}
\|u(t)|_{L^{\infty}(\R^N)}\leq C_{1}t^{-\frac{N}{2p}}e^{-t}\|u_{0}\|_{L^{p}(\R^N)} + C_{2} \left[ \|u_{0}\|_{L^{p}(\R^N)} + {K}_{p}^{\frac{1}{p}}\|u_{0}\|_{L^1(\R^N)}^{\frac{\tilde{\tilde{\lambda}}_{p}}{p}}\|u_{0}\|_{L^p(\R^N)}^{\frac{\tilde{\lambda}_{p}}{p}}t^{\frac{1}{p}}e^{({\lambda}_{p}-1)at} \right]e^{at},
\end{equation*}
where ${\lambda}_{p},\tilde{\lambda}_{p}, \tilde{\tilde{\lambda}}, {K}_{p}, C_{1}$ and $C_{2}$ are positive constants depending on $N ,\ p,\ a,\ b,$ and $\chi$.  Indeed,  let us recall that
$$
u(t)=T(t)u_{0}-\chi\underbrace{\int_{0}^{t}T(t-s)\nabla(u\nabla v)(s)ds}_{J_{1}} + (a+1)\underbrace{\int_{0}^{t}T(t-s)u(s)ds}_{J_{2}} -b\underbrace{\int_{0}^{t}T(t-s)u^{2}(s)ds}_{J_{3}}.
$$
Note that
\begin{equation}
\label{Eq72}
\|v\|_{W^{2,p}(\R^N)}\le C\|u\|_{L^p(\R^N)}.
\end{equation}
Using Lemma $\ref{L_Infty bound}$ and inequalities $(\ref{Step 1:})$  and $(\ref{Eq72})$  we have that
\begin{eqnarray*}
\|J_{1}\|_{L^{\infty}}&\leq& C\int_{0}^{t}(t-s)^{-\frac{1}{2}-\frac{N}{2p}}e^{-(t-s)}\|(u\nabla v)(s)\|_{L^{p}(\R^N)}ds\nonumber\\
&\leq& C\int_{0}^{t}(t-s)^{-\frac{1}{2}-\frac{N}{2p}}e^{-(t-s)}\|u\|_{L^{p}(\R^N)}\| v\|_{C^{1,b}(\R^N)}ds\nonumber\\
&\leq & C\int_{0}^{t}(t-s)^{-\frac{1}{2}-\frac{N}{2p}}e^{-(t-s)}\|u\|_{L^{p}(\R^N)}^{2}ds\nonumber\\
&\leq & C\int_{0}^{t}(t-s)^{-\frac{1}{2}-\frac{N}{2p}}e^{-(t-s)}\left[\|u_{0}\|_{L^{p}(\R^N)}+ K_{p}^{\frac{1}{p}}s^{\frac{1}{p}}\|u_{0}\|_{L^{q_{1}}(\R^N)}^{\frac{\lambda_{p}}{p}}e^{(\lambda_{p} -1)as} \right]^{2}e^{2as}ds\nonumber\\
&\leq & C \left[\|u_{0}\|_{L^{p}(\R^N)}+ K_{p}^{\frac{1}{p}}t^{\frac{1}{p}}\|u_{0}\|_{L^{q_{1}}(\R^N)}^{\frac{\lambda_{p}}{p}}e^{(\lambda_{p} -1)at} \right]^{2}e^{2at}\Ga\left(\frac{1}{2}-\frac{N}{2p}\right).
\end{eqnarray*}
Since $\frac{1}{2}+\frac{N}{2p}\in (\frac{1}{2}, 1),$ we have that $X^{\frac{1}{2}+\frac{N}{2p}}$ is continuously embedded in $L^{\infty}(\mathbb{R}^{N}).$ Thus
\begin{eqnarray*}
\|J_{2}\|_{L^{\infty}}
&\leq & C\int_{0}^{t}\|T(t-s)u(s)\|_{X^{\frac{1}{2}+\frac{N}{2p}}}ds\nonumber\\
&\leq & C\int_{0}^{t}(t-s)^{-\frac{1}{2}-\frac{N}{2p}}e^{-(t-s)}\|u(s)\|_{L^{p}(\R^N)}ds\nonumber\\
&\leq & C \left[\|u_{0}\|_{L^{p}(\R^N)}+ K_{p}^{\frac{1}{p}}\|u_{0}\|_{L^{q_{1}}(\R^N)}^{\frac{\lambda_{p}}{p}}t^{\frac{1}{p}}e^{(\lambda_{p} -1)at} \right]e^{at}\Ga\left(\frac{1}{2}-\frac{N}{p}\right).
\end{eqnarray*}
By \eqref{Lp Estimates}, we have
\begin{equation}\label{alternative  inequality}
\|T(t)u_{0}\|_{\infty}\leq {C}t^{-\frac{N}{2p}}e^{-t}\|u_{0}\|_{L^{p}(\R^N)}.
\end{equation}
Since $u^{2}(s)\geq 0$ for all $s\geq 0$, then $J_{3}\geq 0$. Combining theses with the fact that $u(t)\geq 0$, we obtain that
$$
\|u(t)\|_{\infty}\leq \|T(t)u_{0}\|_{\infty}+ \chi\|J_{1}(t)\|_{\infty}+ (a+1)\|J_{2}(t)\|_{\infty}
$$
Therefore we conclude that
$$
\|u(t)\|_{L^{\infty}}\leq C_{1}t^{-\frac{N}{2p}}e^{-t}\|u_{0}\|_{L^{p}(\R^N)} + C_{2} \left[ \|u_{0}\|_{L^{p}(\R^N)} + {K}_{p}^{\frac{1}{p}}\|u_{0}\|_{L^{q_{1}}(\R^N)}^{\frac{{\lambda}_{p}}{p}}t^{\frac{1}{p}}e^{({\lambda}_{p}-1)at} \right]^{2}e^{2at},
$$
where $C_{1}$ and $C_{2}$ are positive constants depending on $N ,\ p,\ a,\ b,$ and $\chi$. Now, since $1< q_{1}<p$, then $\|u_{0}\|_{q_{1}}\leq \|u_{0}\|_{1}^{\lambda}\|u_{0}\|_{p}^{1-\lambda}$ for $\lambda=\frac{p-q_{1}}{q_{1}(p-1)}$. Thus the Theorem follows.
\end{proof}
\begin{rk}\label{Remark}We first point out that Theorem \ref{Main Theorem 1} does not extend Theorem \ref{Main Theorem 0}  because it requires for $\|u_{0}\|_{L^{1}}+\|u_{0}\|_{L^{p}}$ to be finite. Also, it should be noted that using \eqref{Lp Estimates}, inequality \eqref{alternative  inequality} can be replaced by
\begin{equation}
\|T(t)u_{0}\|_{\infty}\leq C_{1}t^{-\frac{N}{2}}e^{-t}\|u_{0}\|_{L^1(\R^N)}.
\end{equation}
On the other hand,
under the hypothesis of Corollary  \ref{corol}, that is if $ p\leq \frac{\chi}{(\chi-b)_{+}}$, by following the arguments used in the second part of the proof of Theorem \ref{Main Theorem 1} and making use of inequality \eqref{Eq_3 Lp bound} we obtain that
\begin{description}
\item[(i)] There is a constant $C>0$ depending on $a, b, \chi, N$ and $p$ such
\begin{equation*} \|u(t)\|_{L^{\infty}}\leq \text{C}\left[t^{-\frac{N}{2p}}e^{-(1+2a)t} + \|u_{0}\|_{L^{p}(\R^N)}^{2}\right]e^{2at}\quad \forall\ \ t>0.
\end{equation*}
\item[(ii)] For every $\varepsilon>0,$ we have that
\begin{equation*}
\lim_{t\rightarrow \infty}e^{-(2a+\varepsilon)t}\|u(t)-T(t)u_{0}\|_{L^{\infty}}=0
\end{equation*}
\item[(iii)] If in addition $a=0$ then
\begin{equation*}
\sup_{t\geq 0}\|u(t)\|_{L^{\infty}}<\infty.
\end{equation*}
\end{description}
\end{rk}

\subsection{Proof of Theorem \ref{Main Theorem 2}}

In this subsection,
 we extend the results of the previous section to more initial data set and prove Theorem \ref{Main Theorem 2}. Note that the choice of the initial data $u_{0}\in X^{\alpha}$ in Theorem \ref{Main Theorem 1} depends on $N, p$ and $\alpha\in(\frac{1}{2}, 1).$ Since $X^{\beta}$ is continuously imbedded in $X^{\alpha}$ for $\beta\geq \alpha,$ then Theorem \ref{Main Theorem 1}  covers any nonnegative initial data in $X^{\alpha}$ with $\alpha\geq 1.$

\begin{proof}[Proof of Theorem \ref{Main Theorem 2}]
Let $\varphi$ be a nonnegative smooth mollifier function with $\|\varphi\|_{L^{1}}=1.$ For every $\varepsilon>0$ let $\varphi_{\varepsilon}(x)=\frac{1}{\varepsilon^{N}}\varphi(\frac{1}{\varepsilon}x)$ for every $x\in\R^{N}.$ Next, we define $u_{0n}=\varphi_{\frac{1}{n}}\ast u_{0}$ for every $n\geq 1.$ We have that $u_{0n}\in C^{\infty}(\R^N)\cap W^{k,q}(\R^N)$ for $n\in \mathbb{N}, k\geq 1$ and $q\geq 1.$ Furthermore, we have that $u_{0n}\geq 0$ for every $n$ with
\begin{equation}\label{Eq_11}
\|u_{0n}\|_{L^{q}(\R^N)}\leq \|\varphi_{n}\|_{L^{1}(\R^N)}\|u_{0}\|_{L^{q}(\R^N)}=\|u_{0}\|_{L^{q}(\R^N)} \ \ \quad \forall\ p\geq q\geq 1\, \ n\geq 1
\end{equation}
and
$$
\lim_{n\rightarrow \infty}\|u_{0n}-u_{0}\|_{L^{q}(\R^N)}=0 \ \quad \text{for all}\ q\in[1, p].
$$
Let us choose $\alpha\in(\frac{1}{2},\ 1)$ and $0< \delta\leq 2\alpha -1$ satisfying
$$
\frac{(2\alpha -1-\delta)p}{N}>\frac{1}{2}>\frac{1}{2p}=\frac{1}{p}-\frac{1}{2p}.
$$
Hence, $X^{\alpha}$ is continuously imbedded in $C^{1+\delta}.$  We have that $u_{0n}\in X^{\alpha}$ for all $n\geq 1$.  Thus according to Theorem  \ref{Main Theorem 1}, for every $n\geq 1,$ there is a global in time unique solution $(u_{m}(x,t),v_{m}(x,t))=(u(x,t;u_{0m}),v(x,t;u_{0m}))$ of \eqref{IntroEq1} with initial data $u_{0m}.$

By the arguments of Theorem \ref{Local Existence3},
$$
\lim_{m\to\infty}\big[ \|u(\cdot,t;u_{0m})-u(\cdot,t;u_0)\|_{L^p(\R^N)\cap L^{1}(\R^N)}+\|v(\cdot,t;u_{0m})-v(\cdot,t;u_0)\|_{L^p(\R^N)\cap L^{1}(\R^N)}\big]=0
$$
for any $t$ in the maximal existence interval $[0,T_{\max}^p(u_0))$ of $(u(x,t;u_0),v(x,t;u_0))$. Choose $q_{1}\in \left(\max\{1, \frac{N}{2}\}, \min\{p,\ \frac{\chi}{(\chi-b)_{+}}\}\right).$
By Lemma \ref{Lp bound} and \eqref{Step 1:}, we have
\begin{equation*}
\|u(\cdot,t;u_{0m})\|_{L^{p}(\R^N)\cap L^{1}(\R^N)}\leq \left[\|u_{0m}\|_{L^{p}(\R^N)}+ K_{p}^{\frac{1}{p}}t^{\frac{1}{p}}\|u_{0m}\|_{L^{q_{1}}(\R^N)}^{\frac{\lambda_{p}}{p}}e^{(\lambda_{p} -1)at} \right]e^{at} \ \ \ \forall \ t\in[0, \ T_{max}^p(u_0)).
\end{equation*}
This implies that
 \begin{equation*}
\|u(\cdot,t;u_{0m})\|_{L^{p}(\R^N)\cap L^{1}(\R^N)}\leq \left[\|u_{0}\|_{L^{p}(\R^N)}+ K_{p}^{\frac{1}{p}}t^{\frac{1}{p}}\|u_{0}\|_{L^{q_{1}}(\R^N)}^{\frac{\lambda_{p}}{p}}e^{(\lambda_{p} -1)at} \right]e^{at} \ \ \ \forall \ t\in[0, \ T_{max}^p(u_0))
\end{equation*}
and hence $T_{\max}^p(u_0)=\infty$. Since \eqref{Eq2_Main Theorem} holds for every $u(\cdot,t;u_{0m})$, letting $m\to \infty$, we obtain that $u(\cdot,t:u_{0})$ also satisfies \eqref{Eq2_Main Theorem}  This completes the proof.
\end{proof}
As an immediate consequence of the Theorem \ref{Main Theorem 0} and Theorem \ref{Main Theorem 2} we have the following result.
\begin{coro}
Suppose that assumptions of Theorem \ref{Main Theorem 2} hold and $\chi <b$. Let $u_{0}\in L^{1}(\R^N)\cap L^{p}(\R^N)$ and $(u(\cdot,\cdot:u_{0}),v(\cdot,\cdot;u_{0}))$ be the global classical solution of \eqref{IntroEq1} given by Theorem \ref{Main Theorem 2}. Then for every $T>0$, we have that
\begin{equation}\label{coro4.5}
\sup_{t\geq T}\left[\|u(\cdot,t,u_{0})\|_{\infty}+\|v(\cdot,t,u_{0})\|_{\infty}\right]<\infty.
\end{equation}
\end{coro}
\begin{proof}We have that $u(\cdot,\cdot,u_{0})\in C^{2,1}(\R^N\times(0, \infty))$ and  satisfies
\begin{eqnarray}
\partial_{t}u \leq \Delta u -\chi \nabla v(\cdot,\cdot;u_{0}) +(a-(b-\chi)u)u .
\end{eqnarray}
Since $u(\cdot,T,u_{0})\in C^{b}_{\rm unif}(\R^N)$, same arguments as in the proof of Theorem \ref{Main Theorem 0} imply that inequality \eqref{coro4.5} holds.
\end{proof}

\section{Asymptotic behavior of  solutions}
In this section, we discuss the asymptotic behaviors of global bounded  classical solutions of \eqref{IntroEq1} under the assumption that $b>2\chi.$  This will be done in two subsections. The first subsection is devoted for strictly positive initial data. Hence its results apply only for $u_{0}\in C_{\rm unif}^b(\R^N).$ While in the second part we shall deal with initial data with compact supports. Whence there is no restriction on the space $X$ in this case. Again, throughout this section,  $C$  denotes a constant independent of the initial functions and the solutions under consideration, unless specified otherwise.

\subsection{Asymptotic behavior of solutions with strictly positive initial data }

 We shall assume that $\inf{u_{0}}>0$ for $u_{0}\in C_{\rm unif}^b(\R^N)$ with $b>2\chi.$ Following the ideas given in \cite{TeWi} and \cite{Yilong}, we consider the asymptotic behavior of the solution $(u(x,t),v(x,t):=(u(x,t;u_0)$, $v(x,t;u_0))$ of  \eqref{IntroEq1}
 with $u(x,0;u_0)=u_0(x)$. Clearly, the sufficient conditions required for the existence of a unique bounded classical solution $(u,v)$ in  Theorem \ref{Main Theorem 0} are satisfied  and according to \eqref{Eq_Cor1} and \eqref{Eq_Cor2}, it holds that
\begin{equation*}
0\leq u(x,t)\leq u(t,\|u_{0}\|_{\infty}) \ \ \ \forall \ x\in \R^N, \ t\geq0 .
\end{equation*}
with
\begin{equation*}
\lim_{t\rightarrow \infty}u(t,\|u_{0}\|_{\infty})=\frac{a}{b-\chi}.
\end{equation*}
where $u(t, \|u_{0}\|_{\infty})$ is the solution of \eqref{ODE1}.
 Define
\begin{equation}\label{A01}
\overline{u}=\limsup_{t\rightarrow\infty}\left(\sup_{x\in \R^N}u(x,t)\right),\quad
\underline{u}=\liminf_{t\rightarrow\infty}\left(\inf_{x\in \R^N}u(x,t)\right).
\end{equation}
Clearly $0\leq \underline{u}\leq \overline{u}\leq \frac{a}{b-\chi}.$ Our goal is to prove that $\overline{u}=\underline{u}.$  Observe that this will imply that $\|u(t)-\overline{u}\|_{L^{\infty}}\rightarrow 0$ as $t\rightarrow\infty.$ Note that if $a=0,$ then $\underline{u}=\overline{u}=0.$ Hence we shall suppose that $a>0$ in this section.

By  comparison principle for elliptic equations, we have that
\begin{equation}\label{A03}
\inf_{x\in\R^N}u(\cdot,t)\leq v(x,t)\leq \sup_{x\in\R^N}u(\cdot,t)\quad \forall \ x\in\R^N\ ,\ \forall \ t\geq 0.
\end{equation}
Hence, it follows that $\sup_{x\in \R^{N}}v(x,t)<\infty.$
Using definition of limsup and liminf, for every $\varepsilon>0,$ there is $t_{\varepsilon}>0$ such that
\begin{equation}\label{A04}
\underline{u}-\varepsilon\leq u(x,t)\leq \overline{u}+\varepsilon\ \ \forall\ x\in\R^N, t\geq t_{\varepsilon}.
\end{equation}
Combining this with  \eqref{A03} we have
\begin{equation}\label{A06}
\underline{u}-\varepsilon\leq v(x,t)\leq \overline{u}+\varepsilon\ \ \forall\ x\in\R^N, t\geq t_{\varepsilon}.
\end{equation}

Let us define
\begin{equation*}
Lu=\Delta u-\chi\nabla v\nabla u.
\end{equation*}
Since $(u,v)$ solves \eqref{IntroEq1}, we have
\begin{equation}\label{A07}
u_{t}-Lu=-\chi uv+u(a-(b-\chi)u)
        = u\left[a-\chi v -(b-\chi)u \right].
\end{equation}
Note that $0\leq u$. By \eqref{A06}  and \eqref{A07}, for $t\geq t_{\varepsilon}$, we have
\begin{equation}\label{A08}
u_{t}-Lu \leq u\left[a-\chi(\underline{u}-\varepsilon) -(b-\chi)u \right]
\end{equation}
and
\begin{equation}\label{A09}
u_{t}-Lu \geq u\left[a-\chi(\overline{u}+\varepsilon) -(b-\chi)u \right].
\end{equation}
The following lemmas will be helpful in the proof of the main theorem of this section.
\begin{lem}\label{Asymp1Lem0}
Under the foregoing  assumptions, we have that
\begin{equation}\label{Eq_Asym001}
\inf_{x\in\R^N}u(x,t)> 0\ ,\ \ \forall\ t>0
\quad {\rm and}\quad
a-\chi\overline{u}>0.
\end{equation}
\end{lem}
\begin{proof}
Let $u(t, \inf_{x\in\R^N}u_{0})$ be the solution of the following ordinary differential equation,
\begin{equation*}
\begin{cases}
U_{t}=-(b-\chi)U^{2} +(a-\chi v_{\infty})U\\
U_{0}=\inf_{x\in\R^N}u_{0}.
\end{cases}
\end{equation*}
where $v_{\infty}:=\sup_{x\in\R^{N},t\geq 0}v(x,t)$. Since $b-\chi>0,$ then $u(t, \inf_{x\in\R^N}u_{0})$ is globally defined and bounded with $0< u(t, \inf_{x\in\R^N}u_{0})$ for all $t\geq 0.$ Note that if $ a-\chi v_{\infty}\leq 0,$ then $u(t, \inf_{x\in\R^N}u_{0})$ decreases to $0,$ while if $a-\chi v_{\infty} >0,$  by Lemma \ref{global-existence-lemma},  we have $u(t, \inf_{x\in\R^N}u_{0})\rightarrow \frac{a-\chi v_{\infty}}{b-\chi}.$ Since  $u_{0}\geq u(0, \inf_{x\in\R^N}u_{0}),$ by the comparison principle for parabolic equations,  we conclude that
$$
u(t, \inf_{x\in\R^N}u_{0})\leq u(x,t)
$$
for all $x\in\R^N$ and $t\ge 0$. Hence the first inequality in \eqref{Eq_Asym001} follows. On the other hand, if we suppose by contradiction that  $a-\chi\overline{u}\leq  0,$ then we would have that
\begin{eqnarray*}
\frac{a}{\chi} \leq \overline{u}\leq  \frac{a}{b-\chi},
\end{eqnarray*}
which contradicts the fact that $b>2\chi.$ Hence the second inequality in \eqref{Eq_Asym001} holds.
\end{proof}
Since $\underline{u}\leq \overline{u}, $ according to Lemma \ref{Asymp1Lem0}, we may suppose that $0<a-\chi(\overline{u}+\varepsilon) <a-\chi(\underline{u}-\varepsilon)$ for $\varepsilon$ very small.
\begin{lem}\label{Asymp1Lem1}
Under the forgoing assumptions , it holds that
\begin{equation}\label{A30}
(b-\chi)\overline{u}\leq a-\chi\underline{u}\quad {\rm and}\quad
a-\chi\overline{u} \leq (b-\chi)\underline{u}.
\end{equation}
\end{lem}
\begin{proof}
Let $\overline{w}(t)$ denote the solution of the initial value problem
\begin{equation}\label{supSol}
\begin{cases}
\overline{w}_{t}=\overline{w}\left[a-\chi(\underline{u}-\varepsilon) -(b-\chi)\overline{w} \right] \ \forall t>t_{\varepsilon}\\
\overline{w}(t_{\varepsilon})=\sup_{x\in\R^N}u(x,t_{\varepsilon}).
\end{cases}
\end{equation}
By \eqref{A08}, \eqref{supSol}, and the comparison principle for parabolic equations, we obtain that
\begin{equation}\label{A10}
u(x,t)\leq \overline{w}(t) \quad \forall \ x\in\R^N,\ t\geq t_{\varepsilon}.
\end{equation}
According to Lemma \ref{Asymp1Lem0} we have that  $\sup_{x\in\R^N}u_{0}(x,t)>0$ for all $t>0.$ In particular we have that $\overline{w}(t_{\varepsilon})>0.$
On the other hand, by Lemma \ref{global-existence-lemma}
\begin{equation*}
\overline{w}(t) \to \frac{a-\chi(\underline{u}-\varepsilon)}{b-\chi} \ \ \text{as} \ \ t\rightarrow \infty.
\end{equation*}
Combining this with inequality \eqref{A10}, we obtain that
\begin{equation*}
\overline{u}\leq \frac{a-\chi(\underline{u}-\varepsilon)}{b-\chi}\ \ \ \forall\ \varepsilon>0.
\end{equation*}
By letting $\varepsilon\rightarrow 0,$ we obtain the first inequality in  \eqref{A30}.

Similarly, let $\underline{w}(t)$ be solution of
\begin{equation}\label{subSol}
\begin{cases}
\underline{w}_{t}=\underline{w}\left[a-\chi(\overline{u}+\varepsilon) -(b-\chi)\underline{w} \right] \ \forall t>t_{\varepsilon}\\
\underline{w}(t_{\varepsilon})=\inf_{x\in\R^N}u(x,t_{\varepsilon}).
\end{cases}
\end{equation}
By \eqref{A08}, \eqref{subSol}, and the  comparison principle for parabolic equations, we have
\begin{equation*}
u(x,t)\geq \underline{w}(t) \quad \forall \ x\in\R^N,\ t\geq t_{\varepsilon}.
\end{equation*}
Same arguments as in above yield that $\underline{w}(t_{\varepsilon})>0$, and

\begin{equation*}
\underline{w}(t)\rightarrow \frac{a-\chi(\overline{u}+\varepsilon)}{b-\chi} \ \ \text{as} \ \ t\rightarrow \infty.
\end{equation*}
Combining this with inequality \eqref{A04}, we obtain that
\[
\underline{u}\geq  \frac{a-\chi(\overline{u}+\varepsilon)}{b-\chi}\ \ \ \forall\ \varepsilon>0.
\]
By letting $\varepsilon\rightarrow 0,$ we obtain  the second inequality in \eqref{A30}.
 Lemma \ref{Asymp1Lem1} is thus proved.
\end{proof}

From these Lemmas, we can easily present the proof of Theorem \ref{Main Theorem 3}.

\begin{proof}[Proof of Theorem \ref{Main Theorem 3}]

It follows from Lemma \ref{Asymp1Lem1} that
\begin{align*}
(b-2\chi)\overline{u}&= (b-\chi)\overline{u}-\chi\overline{u}=(b-\chi)\overline{u} +\left(a -\chi\overline{u}\right) -a\nonumber\\
 &\leq  a -\chi\underline{u} +(b-\chi)\underline{u}  -a= (b-2\chi)\underline{u}.
\end{align*}
Combining this with the fact that $\underline{u}\leq \overline{u}$ and $ (b-2\chi)>0$ we obtain that
\begin{equation}\label{A11}
\underline{u}= \overline{u}.
\end{equation}
Equality \eqref{A11} combining with Lemma \ref{Asymp1Lem1} imply that
$$ (b-\chi)\overline{u}=a-\chi \overline{u}. $$ Solving for $ \overline{u} $ in the last equality, we obtain that $ \underline{u}= \overline{u}=\frac{a}{b}.$
The conclusion of the theorem follow ready from the last equality and   \eqref{A01}, and \eqref{A03} .
\end{proof}

When the initial data $u_{0}$ is not bounded away from zero. The uniform convergence on $\R^N$ of $u(x,t)$ as $t\to\infty$ to the constant steady solution $\frac{a}{b}$ does not hold. However we have a uniform local convergence of $u(x,t)$ as $t\to\infty$  to the steady solution under an additional hypothesis. We establish these in the last subsection.

\subsection{Asymptotic behavior of solutions with  non-negative initial data}

Throughout this section we suppose that $u_{0}\in C_{\rm unif}^b(\R^N)$ is nonnegative and not identically zero with $ 2\chi<b$. We shall also denote by $(u(x,t), v(x,t))$, the global bounded classical solution of \eqref{IntroEq1} associated  with initial data $u_{0}.$

In order to study the asymptotic behavior of $u(\cdot,t)$,  we first need to get some estimate on $\|\nabla v(x,t)\|.$
Since $\Delta v=v-u$ and $\|v(\cdot, t)\|_{\infty}\leq \|u(\cdot,t)\|_{\infty}$ for every $t> 0,$  it follows from Lemma \ref{lemma2} that
\begin{equation}\label{A_EE001}
\| \nabla v(x,t)\|\leq \sqrt{N}\|u(\cdot, t)\|_{\infty}  \ \ \forall \ x\in \mathbb{R}^{N}, \  t>0.
\end{equation}

Let us define for $U\in C^{2,1}(\R^N\times\R)$
\begin{equation}\label{Definition of L}
LU:=\partial_{t}U - \Delta U -\chi\nabla v\nabla U.
\end{equation}
 We have that
 \begin{equation}
 Lu=\underbrace{u(a-\chi v - (b-\chi)u)}_{F^{1}(x,t,u)},\ \ \ x\in \mathbb{R}^{N},\ \ t>0.
 \end{equation}
 Hence, since $v\geq 0,$ it follows that
\begin{equation}\label{F2}
Lu\leq \underbrace{u(a-(b-\chi)u)}_{F^{2}(u)}.
\end{equation}
 By  the comparison principle for parabolic equations, we have that
\begin{equation}\label{A_Eq1}
u(x,t)\leq U(t, \|u_{0}\|_{\infty})\ \ \forall\ x\in\mathbb{R}^{N},\ \ t\geq 0,
\end{equation}
where $U(t,\|u_{0}\|_{\infty})$ is the solution of the ODE
\begin{equation}\label{U infty}
\begin{cases}
LU=F^{2}(U)\\
U(0)=\|u_{0}\|_{\infty}.
\end{cases}
\end{equation}
By Lemma \ref{global-existence-lemma}, $U(t,\|u_{0}\|_{\infty})\to \frac{a}{b-\chi}$ as $t\rightarrow \infty.$

\medskip

Next, we prove some lemmas.

\begin{lem}
\label{aux-asymptotic-lemma1}
Assume that
$0<\chi<\frac{2 b}{3+\sqrt{1+Na}}.$
Then
\begin{equation}\label{A_Eq02}
\lim_{R\rightarrow \infty }\inf_{t>R, |x|>R}\left(  4(a-\chi v(x,t))-\chi^{2}\|\nabla v(x,t)\|^{2} \right) >0.
\end{equation}
\end{lem}

\begin{proof}

From \eqref{A_EE001}, \eqref{A_Eq1},  and the fact that $U(t,\|u_{0}\|_{\infty})\rightarrow \frac{a}{b-\chi}$ as t goes to infinity, for \eqref{A_Eq02} to hold, it is enough to have
\begin{equation}\label{A_Eq4}
 4(a-\frac{\chi a}{b-\chi} )-\frac{N\chi^{2} a^2}{(b-\chi)^{2}}> 0.
\end{equation}
Let $\mu=\frac{\chi}{b-\chi}$. \eqref{A_Eq4} is equivalent to
$4(1-\mu)-N a\mu^2>0.$
This implies that
$$
0<\mu=\frac{\chi}{b-\chi}<\frac{2}{1+\sqrt{1+Na}}
$$
and then
$$
0<\chi<\frac{2 b}{3+\sqrt {1+Na}}.
$$
The lemma is thus proved.
 \end{proof}

\begin{lem}\label{MainLemma4} Let $u_{0}\in C_{\rm unif}^b(\R^N)$ be a nonnegative and non-zero function. Let $(u,v)$ be the classical bounded solution of \eqref{IntroEq1} associated with $u_{0}.$  If
\begin{equation}\label{A_Eqtm1}
\chi< \frac{2b}{3+\sqrt{1+Na}},
\end{equation}
then
\begin{equation}\label{A_Eqtm001}
\liminf_{t\rightarrow\infty}\inf_{|x|\leq ct}u(x,t)>0
\end{equation}
for every $0\leq c < c^{\ast}(\le 2\sqrt a)$, where
\begin{equation}\label{A_Eq6}
c^{\ast}=\lim_{R\rightarrow\infty}\inf_{|x|\geq R, t\geq R}(2\sqrt{a-\chi v(x,t)}-\chi\|\nabla v(x,t)\|).
\end{equation}
\end{lem}
\begin{proof} First, we know that $\|v(\cdot,t)\|_{\infty}\leq \|u(\cdot,t)\|_{\infty}\leq U(t,\|u_{0}\|_{\infty})$ for all $t$ with $ U(t,\|u_{0}\|_{\infty})\rightarrow \frac{a}{b-\chi}$ as $t\rightarrow \infty.$ Hence  for every $\varepsilon>0$, there is $T_{\varepsilon}>0$ such that
\begin{equation}\label{Eq00100}
\|v(\cdot,t)\|_{\infty}\leq \frac{a}{b-\chi}+ \varepsilon,  \quad \forall\ t\geq T_{\varepsilon}
\end{equation}
and $T_\varepsilon\to\infty$ as $\epsilon\to 0$.
Define $\tilde{u}(x,t)=u(x,T_{\varepsilon}+t)$ and $\tilde{v}(x,t)=v(x,T_{\varepsilon}+t)$ for all $t\geq 0$ and $x\in\R^N.$
Then $\tilde u(x,t)$ satisfies
$$
\tilde u_t=\Delta\tilde u-\chi\nabla \tilde v\nabla \tilde u+\tilde F(t,x,\tilde u),
$$
where $\tilde F(t,x,u)=\tilde u(a-\chi \tilde v-(b-\chi)\tilde u)$. Note  that
$$
4\partial_{u}\tilde F^{1}(x,t,0)-\| \chi\nabla \tilde v(x,t)\|^{2}=4(a-\chi \tilde v(x,t))-\chi^{2}\|\nabla v(x,t)\|^{2}.
$$
By Lemma \ref{aux-asymptotic-lemma1},
$$
\lim_{R\rightarrow \infty }\inf_{t>R, |x|>R}\left( 4\partial_{u}\tilde F^{1}(x,t,0)-\| \chi\nabla \tilde v(x,t)\|^{2}\right )>0.
$$

Next, we introduce the linear operator
\begin{equation*}
\mathcal{L}w:=\partial_{t}w -\Delta w+q(x,t)\cdot \nabla w -p(x,t)w
\end{equation*} for every $w\in C^{2,1}(\R^N\times\R)$, where
 $$
 q(x,t)=\begin{cases} \chi \nabla \tilde v(x,t),\quad t\ge 0\cr
 \chi \nabla\tilde v(x,0),\quad t<0
 \end{cases}
 \quad {\rm and}\quad
 p(x,t)=\begin{cases}
 a-\chi\tilde v(x,t),\quad t\ge 0\cr
 a-\chi\tilde v(x,0),\quad t<0.
 \end{cases}
 $$
 Following \cite{Berestycki1}, the generalized principal eigenvalue associated to the operator $\mathcal{L}$ is defined to be
\begin{equation*}
\lambda_{1}':=\inf\{ \lambda\in\R^N\ : \ \exists \phi\in C^{2,1}\cap W^{1,\infty}(\R^{N}\times\R) , \inf_{(x,t)}\phi>0, \mathcal{L}\phi\leq \lambda\phi \}.
\end{equation*}
We  show that ${\lambda}_{1}'<0$ for small values of $\varepsilon$. Indeed, for $w(x,t)=1$, the constant function, using definition of $\tilde{v}$ and inequality \eqref{Eq00100}, we obtain that
\begin{align*}
\mathcal{L}(w)&=\begin{cases} -a +\chi\tilde{v}(x,t),\quad t\ge 0\cr
-a+\chi \tilde v(x,0),\quad  t<0
\end{cases}\\
&\leq -a + \frac{\chi  a}{b-\chi}+\chi \varepsilon=\chi \varepsilon- \frac{a(b-2\chi)}{b-\chi}.
\end{align*}
Hence ${\lambda}_{1}'<0 $ whenever $\varepsilon< \frac{a(b-2\chi)}{\chi(b-\chi)}$.

Now, by  Theorem 1.5  in  \cite{Berestycki1}, it holds that
\begin{equation}\label{A_Eqtm01000}
\liminf_{t\rightarrow\infty}\inf_{|x|\leq ct}\tilde{u}(x,t)>0
\end{equation}
for every $0\leq c < c^{\ast}_{\varepsilon}$ where
\[
c^{\ast}_{\varepsilon}=\liminf_{|x|\rightarrow\infty}\inf_{t\geq T_{\varepsilon}}(2\sqrt{a-\chi v(x,t)}-\chi\|\nabla v(x,t)\|).
\]
By the definition of $\tilde{u}$ and \eqref{A_Eqtm01000},  we deduce that
\begin{equation}\label{A_Eqtm01001}
\liminf_{t\rightarrow\infty}\inf_{|x|\leq ct}u(x,t+T_{\varepsilon})>0\quad \forall\,\, 0\leq c < c^{\ast}_{\varepsilon}.
\end{equation}

We claim that $\lim_{\varepsilon\to 0}c^*_{\varepsilon}=c^*$. In fact, recall that
$$
c^{\ast}=\lim_{R\to \infty}\inf_{|x|\geq R, t\geq R}\underbrace{\left( 2\sqrt{a-\chi v(x,t)}-\chi\|\nabla v(x,t)\| \right)}_{f(x,t)}.
$$
Using the fact that
$$
\inf_{|x|\geq R}\inf_{t\geq T_{\varepsilon}}f(x,t)=\inf_{|x|\geq R, t\geq T_{\varepsilon}}f(x,t)\leq \inf_{|x|\geq R, t\geq R}f(x,t), \ \ \forall\ R\geq T_{\varepsilon},
$$
we have
\begin{equation}
\label{c-star-eq1}
c^{\ast}_{\varepsilon} =  \liminf_{|x|\to \infty}\inf_{t\geq T_{\varepsilon}}f(x,t)=\lim_{R\to \infty}\inf_{|x|\geq R}\inf_{t\geq T_{\varepsilon}}f(x,t)\leq  \lim_{R\to \infty}\inf_{|x|\geq R, t\geq R}f(x,t)= c^{\ast}.
\end{equation}
Using the fact that for given $\delta >0$, there is $R_{\delta}>0$ such that
$$
c^{\ast}-\delta< f(x,t)\ \ \ \forall\ |x|,t\geq R_{\delta}
$$
and that there is $\varepsilon_{0}$ such
$$ T_{\varepsilon}\geq R_{\delta} \ \ \ \forall \ \varepsilon< \varepsilon_{0},$$
we have
$$
c^{\ast}-\delta\leq \inf_{|x|\geq R, \ t\geq T_{\varepsilon}}f(x,t)\quad \forall \ R\geq R_{\delta}, \ \forall \ \varepsilon<\varepsilon_{0}.
$$
 Thus, for every $0<\varepsilon<\varepsilon_0$, we have that  %Letting $R$ goes infinity, we obtain that
\begin{equation}
\label{c-star-eq2}
c^{\ast}-\delta \leq \lim_{R\to\infty}\inf_{|x|\geq R,t\geq T_{\varepsilon}}f(x,t)= c^{\ast}_{\varepsilon}. %\ \ \forall \ \varepsilon<\varepsilon_{0}.
\end{equation}

By \eqref{c-star-eq1} and \eqref{c-star-eq2}, we obtain
$$
\lim_{\varepsilon\to 0}c^{\ast}_{\varepsilon}=c^{\ast}.
$$

Finally, let $0\leq c < c^{\ast}$ be fixed. There is some $\varepsilon>0$ small enough such that $c< c^{\ast}_{\varepsilon}$. Choose $\tilde{c}\in ( c ,\ c^{\ast}_{\varepsilon})$. Observe that
\begin{equation}\label{212121}
 ct= \tilde{c}(t-T_{\varepsilon})-(\tilde{c}-c)(t-\frac{\tilde{c}T_{\varepsilon}}{\tilde{c}-c})\leq \tilde{c}(t-T_{\varepsilon})
\end{equation}
whenever $ t\geq \frac{\tilde{c}T_{\varepsilon}}{\tilde{c}-c}. $
Hence, since $u(x,t)=u(x,t-T_{\varepsilon} +T_{\varepsilon}),$ we obtain that
$$
\inf_{\|x\|\leq ct}u(x,t)\geq \inf_{\|x\|\leq \tilde{c}t}u(x,t+T), \ \ \forall \ \ t\geq \frac{\tilde{c}T}{\tilde{c}-c}.
$$
Combining the last inequality with inequality \eqref{212121}, we conclude that inequality \eqref{A_Eqtm001} hold.
\end{proof}

The next step to the proof of Theorem \ref{Main Theorem 4}  is the following result. This result asserts that, under some conditions,  the asymptotic behavior of the function $v(x,t)$ is quit similar to the one of the function $u(x,t).$

\begin{lem}\label{Asym of v}
\begin{description}
\item[(i)] If there is a positive constant $c^{\ast}_{\rm low}$ such that
\begin{equation}\label{T00}
\lim_{t\rightarrow\infty}\sup_{|x|\leq ct}|u(x,t)-\frac{a}{b}|=0\quad \forall\,\, 0\leq c < c^{\ast}_{\rm low},
\end{equation}
 then
\begin{equation*}
\lim_{t\rightarrow\infty}\sup_{|x|\leq ct}|v(x,t)-\frac{a}{b}|=0\quad \forall\,\, 0\leq c < c^{\ast}_{\rm low}.
\end{equation*}

\item[(ii)]If there is a positive constant $c^{\ast}_{\rm up}$ such that
\begin{equation}\label{T01}
\lim_{t\rightarrow\infty}\sup_{|x|\geq ct}u(x,t)=0\quad \forall\,\, c > c^{\ast}_{\rm up},
\end{equation}
then
\begin{equation*}
\lim_{t\rightarrow\infty}\sup_{|x|\geq ct}v(x,t)=0\quad \forall\,\, c > c^{\ast}_{\rm up}.
\end{equation*}
\end{description}
\end{lem}

\begin{proof}
We first recall that
\begin{equation}\label{T1}
v(x,t)=\int_{0}^{\infty}\int_{\R^N}\frac{e^{-s}}{(4\pi s)^{\frac{N}{2}}}e^{-\frac{|x-y|^{2}}{4s}}u(y,t)dyds=\frac{1}{\pi^{\frac{N}{2}}}\int_{0}^{\infty}\int_{\R^N}e^{-s}e^{-|z|^{2}}u(x-2\sqrt{s}z,t)dyds.
\end{equation}
Let $\varepsilon>0.$ be fixed.
Since
$$
\frac{1}{\pi^{\frac{N}{2}}}\int_{0}^{\infty}\int_{\R^N}e^{-s}e^{-|z|^{2}}dyds=\underbrace{\left[\int_{0}^{\infty}e^{-s}ds\right]}_{=1}\underbrace{\left[\frac{1}{\pi^{\frac{N}{2}}}\int_{\R^N}e^{-|z|^{2}}dy\right]}_{=1} =1
$$
and
$$
 \sup_{t\geq 0}\|u(\cdot,t)+1\|_{\infty}<\infty,
$$ there is $R>0$ large enough such that
\begin{equation}\label{T2}
\frac{1}{\pi^{\frac{N}{2}}}\int\int_{\{s \geq R\}\cup\{ |z|\geq R\}}e^{-s}e^{-|z|^{2}}|u(x-2\sqrt{s}z,t)+1|dyds<\frac{\varepsilon}{2}
\end{equation}
for all $x\in\R^N,\ t\geq 0.$
\medskip

\textbf{(i)}  Let $0\leq c<c^{\ast}_{\rm low}$ be fixed. Choose $c<\tilde{c}<c^{\ast}_{\rm low}.$ From \eqref{T00}, there is $t_{0}>0$ such that
\begin{equation}\label{i0}
|u(x,t)-\frac{a}{b}|\leq \frac{\varepsilon}{2}
\end{equation}
for every $|x|\leq \tilde{c}t,\ t\geq t_{0}.$  Using \eqref{T2},  for every $x\in\R^N$ and $t>0,$ we have
\begin{eqnarray}\label{i1}
|v(x,t)-\frac{a}{b}| & \leq & \frac{1}{\pi^{\frac{N}{2}}}\int_{0}^{\infty}\int_{\R^N}e^{-s}e^{-|z|^{2}}|u(x-2\sqrt{s}z,t)-\frac{a}{b}|dyds \nonumber \\
 &\leq & \frac{1}{\pi^{\frac{N}{2}}}\int\int_{\{s\leq R \ \text{and} \ |z|\leq R\}}e^{-s}e^{-|z|^{2}}|u(x-2\sqrt{s}z,t)-\frac{a}{b}|dyds \nonumber \\
 & & +\frac{1}{\pi^{\frac{N}{2}}}\int\int_{\{s > R\}\cup \{|z|> R\}}e^{-s}e^{-|z|^{2}}|u(x-2\sqrt{s}z,t)+\frac{a}{b}|dyds \nonumber \\
 &\leq & \frac{1}{\pi^{\frac{N}{2}}}\int\int_{\{s\leq R \ \text{and} \ |z|\leq R\}}e^{-s}e^{-|z|^{2}}|u(x-2\sqrt{s}z,t)-\frac{a}{b}|dyds +\frac{\varepsilon}{2}.
\end{eqnarray}
On the other hand we have that
\begin{equation}\label{i2}
|x-2\sqrt{s}z| \leq  |x|+ 2\sqrt{s}|z|\leq  ct +2R^{2} =  \tilde{c}t -(\tilde{c}-c)(t-\frac{2R^{2}}{\tilde{c}-c} ) \leq  \tilde{c}t
\end{equation}
whenever $|x|\leq ct$,  $s\leq R$, $|z|\leq R$,  and $t\geq \frac{2R^{2}}{\tilde{c}-c} .$ Hence combining inequalities \eqref{i0}, \eqref{i1} and \eqref{i2}, we obtain that
\[
\sup_{|x|\leq ct}|v(x,t)-\frac{a}{b}|\leq \varepsilon
\]
whenever $t\geq \max\{t_{0}, \ \frac{2R^{2}}{\tilde{c}-c}\}.$ This complete the proof of $(i).$

\medskip

\textbf{(ii)}
Let $c>c^{\ast}_{\rm up}$ be fixed. Choose $c^{\ast}_{\rm up}<\tilde{c}<c.$  Since $\tilde{c}>c^{\ast}_{\rm up},$ according to \eqref{T01}, there is $t_{1}>0$ such that
\begin{equation}\label{T3}
\sup_{|y|\geq \tilde{c}t}u(y,t)<\frac{\varepsilon}{2} ,\quad \forall\ t\geq t_{1}.
\end{equation}
Using Triangle inequality,  for every $(z,s)\in B(0,R)\times [0, R]$  and $|x|\geq ct,$ it hold that
$$ct\leq  |x|\leq |x-2\sqrt{s}z|+2\sqrt{s}|z|\leq |x-2\sqrt{s}z|+2R^{2}. $$
Which implies that
\begin{equation}\label{T4}
ct-2R^{2}\leq |x-2\sqrt{s}z|
\end{equation}
for every $(z,s)\in B(0,R)\times [0, R]$  and $|x|\geq ct.$ But
\begin{equation} \label{T5}
\tilde{c}t\leq ct-2R^{2}\Leftrightarrow t\geq \frac{2R^{2}}{c-\tilde{c}}.
\end{equation}
 Hence combining inequalities \eqref{T3},\eqref{T4} and \eqref{T5} we have that
 \begin{equation}\label{T6}
 \sup_{(z,s)\in B(0,R)\times[0, R] }u(x-2\sqrt{s}z,t)\leq \varepsilon
 \end{equation}
 for $|x|\geq ct$ and $t\geq \max\{t_{1},\ \frac{2R^2}{c-\tilde{c}}\}.$ This implies that
\begin{equation}\label{T7}
\frac{1}{\pi^{\frac{N}{2}}}\int\int_{\{s\leq R, |z|\leq R\}}e^{-s}e^{-|z|^{2}}u(x-2\sqrt{s}z,t)dzds<\frac{\varepsilon}{2}
\end{equation}for $|x|\geq ct$ and $t\geq \max\{t_{1},\ \frac{2R^2}{c-\tilde{c}}\}.$
From \eqref{T1} and inequalities \eqref{T2} and \eqref{T7}, it follows that
\begin{equation*}
\lim_{t\rightarrow\infty}\sup_{|x|\geq ct}v(x,t)=0
\end{equation*}
for every $c>c^{\ast}_{up}.$
\end{proof}

We now prove Theorem \ref{Main Theorem 4}.

\begin{proof}[Proof of Theorem \ref{Main Theorem 4}]
(1) From Lemma \ref{MainLemma4} we know that
$$
\liminf_{t\rightarrow\infty}\inf_{|x|\leq ct}u(x,t)>0
$$
for all $0\leq c < c^{\ast}.$ We first prove that \eqref{A_Eq8} holds for $0\le c<c^*$, where $c^*$ is as in Lemma \ref{MainLemma4}.

 Assume that there are constants $0\leq c< c^{\ast},$ $\delta>0$ and a sequence $\{(x_{n},t_{n})\}_{n\in\mathbb{N}}$ such $t_{n}\rightarrow \infty,\ \|x_{n}\|\leq ct_{n} $ and
\begin{equation}\label{aux-eq1}
|u(x_{n},t_{n})-\frac{a}{b}|\geq \delta, \quad \ \ \forall\ n\geq 1.
\end{equation}
For every $n\geq 1,$ let us define
\begin{equation*}
u_{n}(x,t)=u(x+x_{n},t+t_{n}) , \quad \text{and} \quad v_{n}(x,t)=v(x+x_{n},t+t_{n})
\end{equation*}
for every $x\in\R^N,\ t\geq -t_{n}.$ Choose $0<\alpha<\frac{1}{2}.$
Using the facts that
$$
u(\cdot,t)=T(t)u_{0}-\chi\int_{0}^{t}T(t-s)\nabla(u(s)\nabla v(s))ds+\int_{0}^{t}T(t-s)(au(s)-bu^{2}(s))ds.
$$
For every $n\geq 1,$ we have
\begin{eqnarray*}
\|u_{n0}\|_{X^{\alpha}}&\leq & \|T(t_{n})u_{0}(.+x_{n})\|_{X^{\alpha}}+\chi\int_{0}^{t_n}\|T(t_n-s)\nabla(u(s)\nabla v(s))\|_{X^{\alpha}}ds\nonumber\\
&\quad & +\int_{0}^{t_n}\|T(t_n-s)(au(s)-bu^{2}(s))\|_{X^{\alpha}}ds\nonumber\\
&\leq & C_{\alpha}t_{n}^{-\alpha}\|u_{0}\|_{\infty}+\chi\int_{0}^{t_n}\|T(t_n-s)\nabla(u(s)\nabla v(s))\|_{X^{\alpha}}ds \nonumber\\
&\quad & + C_{\alpha}\int_{0}^{t_n} e^{-(t_n-s)}(t_n-s)^{-
\alpha}\|au(s)-bu^{2}(s))\|_{\infty}ds.
\end{eqnarray*}
Next, using Lemma \ref{L_Infty bound 2}, the last inequality can be improved as
\begin{eqnarray}\label{L01}
\|u_{n0}\|_{X^{\alpha}}
&\leq & C_{\alpha}t_{n}^{-\alpha}\|u_{0}\|_{\infty}+C_{\alpha}\chi\int_{0}^{t_n}e^{-(t_n-s)}(t_n-s)^{-\frac{1}{2}-\alpha}\|u(s)\nabla v(s)\|_{\infty}ds \nonumber\\
&+ & C_{\alpha}\int_{0}^{t_n} e^{-(t_n-s)}(t_n-s)^{-
\alpha}\|au(s)-bu^{2}(s))\|_{\infty}ds.
\end{eqnarray}
Combining inequality \eqref{A_EE001} and the fact that  $\sup_{t}\|u(\cdot,t)\|_{\infty}<\infty$, inequality \eqref{L01} becomes,
\begin{eqnarray}\label{L02}
\|u_{n0}\|_{X^{\alpha}}
&\leq & C_{\alpha}t_{n}^{-\alpha}\|u_{0}\|_{\infty} + {C}\left(\underbrace{\int_{0}^{\infty}e^{-s}s^{-\frac{1}{2}-\alpha} ds}_{\Gamma(\frac{1}{2}-\alpha)} +\underbrace{\int_{0}^{\infty}e^{-s}s^{-\alpha}ds}_{\Gamma(1-\alpha)}\right)\nonumber\\
&= & C_{\alpha}t_{n}^{-\alpha}\|u_{0}\|_{\infty} + {C}
\end{eqnarray}
Since $t_n\rightarrow\infty$ as $n\rightarrow\infty,$ then $\sup_{n}\|u_{n0}\|_{X^{\alpha}}<\infty.$ Furthermore, similar arguments as in the proof  of Theorem \ref{Local existence1} show that the functions $u_{n}: [-T, T]\rightarrow X^{\alpha}$ are equicontinuous for every $T>0$. Hence, Arzela-Ascoli's Theorem and Theorem 15 (page 80 of \cite{Amann}) imply that there is a function $(\tilde{u},\tilde{v})\in C^{2,1}(\R^N\times\R)$ and  a subsequence $\{(u_{n'},v_{n'})\}_{n}$ of $\{(u_{n},v_{n})\}_{n}$  such that $(u_{n'},v_{n'})\rightarrow (\tilde{u},\tilde{v})$ in $C^{1+\delta',\delta'}_{loc}(\R^{N}\times(-\infty, \infty))$ for some $\delta'>0.$ Moreover, $\tilde{v}=(I-\Delta)^{-1}\tilde{u}$ and $(\tilde{u},\tilde{v})$ solves \eqref{IntroEq1} in classical sense. Note that
$$
\tilde{u}(x,t)=\lim_{n\rightarrow \infty}u(x+x_{n'},t+ t_{n'})
$$
for every $x\in\R^N,\ t\in\R.$ Next, choose $\tilde{c}\in (c, c^{\ast}).$ For every $x\in\R^N$ and $t\in \R,$ we have
\begin{eqnarray*}
\|x+x_{n'}\| & \leq & \|x\|+\|x_{n'}\|\leq  \|x\|+ ct_{n'} \nonumber\\
            & = &    \tilde{c}(t_{n'}+t)- (\tilde{c}-c)(t_{n'}-\frac{\|x\|-\tilde{c}t}{\tilde{c}-c}) \leq   \tilde{c}(t_{n'}+t)
\end{eqnarray*}
whenever $t_{n'}\geq \frac{\|x\|+\tilde{c}|t|}{\tilde{c}-c}.$ Thus, it follows that
$$
\tilde{u}(x,t)=\lim_{n\rightarrow \infty}u(x+x_{n'},t+t_{n'})\geq \liminf_{s\rightarrow \infty}\inf_{\|y\|\leq \tilde{c}s}u(y,s)>0
$$ for every $(x,t)\in\R^N\times\R.$  Hence $\inf_{(x,t)\in\R^N\times\R}\tilde{u}(x,t)>0.$
Using Theorem \ref{Main Theorem 3}, we must have
\begin{equation}\label{Q0}
\lim_{t\rightarrow\infty}\|\tilde{u}(\cdot ,t)-\frac{a}{b}\|_{\infty}=0.
\end{equation}

\textbf{Claim.} $\tilde{u}(x,t)=\frac{a}{b}$ for every $(x,t)\in\R^{N+1}.$

The proof of this claim is inspired from the ideas used  to prove Theorem \ref{Main Theorem 3}. Let us set $\underline{u}_{0}=\inf_{(x,t)\in\R^{N+1}}\tilde{u}$ and $\overline{u}_{0}=\sup_{(x,t)\in\R^{N+1}}\tilde{u}.$ Since $(\tilde u, \tilde v)$
 solves \eqref{IntroEq1} in the classical sense and $b>2\chi,$ it follows from Lemma \ref{Asymp1Lem0} that
 \begin{equation}\label{Q1}
 a-\chi\underline {u}_0\geq a-\chi \overline{u}_0>0.
 \end{equation}
 For every $t_{0}\in \R,$ let $\underline{u}(\cdot,t_{0})$ and $\overline{u}(\cdot,t_{0})$ be the solutions of
 $$
 \begin{cases}
 \frac{d}{dt}\overline{u}=\overline{u}(a-\chi\underline{u}_{0}-(b-\chi)\overline{u}),\ \ t>t_0\\
 \overline{u}(t_{0},t_{0})=\overline{u}_{0}
 \end{cases}
 $$
 and
 $$
 \begin{cases}
 \frac{d}{dt}\underline{u}=\underline{u}(a-\chi\overline{u}_{0}-(b-\chi)\underline{u}),\ \ t>t_0\\
 \underline{u}(t_{0},t_{0})=\underline{u}_{0},
 \end{cases}
 $$
 respectively.
 Since $0<\underline{u}_{0}\leq \tilde{v}(x,t)\leq \overline{u}_{0}$ for every $(x,t)\in\R^{N+1},$ following the same arguments used to prove Lemma \ref{Asymp1Lem1}, we obtain that
 \begin{equation}\label{Q3}
 \underline{u}(t-t_0,0)=\underline{u}(t,t_{0})\leq \tilde{u}(x,t)\leq\overline{u}(t,t_{0})=\overline{u}(t-t_0,0)\ \ \forall\ x\in\R^N,\ t\geq t_{0}.
 \end{equation}
 By Lemma \ref{global-existence-lemma}, we have
 \begin{equation}\label{Q5}
\lim_{t_0\to -\infty}\overline {u}(t-t_0,0)=\frac{a-\chi\underline{u}_{0}}{b-\chi}\quad \text{and}\quad  \lim_{t_0\to -\infty} \underline{u}(t-t_0,0)=\frac{a-\chi\overline{u}_{0}}{b-\chi}.
\end{equation}
Combining \eqref{Q3} and \eqref{Q5} we obtain that
\begin{equation*}
\frac{a-\chi\overline{u}_{0}}{b-\chi}\leq \tilde{u}(x,t)\leq \frac{a-\chi\underline{u}_{0}}{b-\chi},\ \ \ \forall\ \ (x,t)\in \R^{N+1}.
\end{equation*}
This together with $\underline{u}_{0}\leq \tilde{u}(x,t)\leq \overline{u}_{0}$ implies that
\begin{equation*}
a-\chi\overline{u}_{0}\leq (b-\chi)\underline{u}_{0} \quad \text{and} \quad (b-\chi)\overline{u}_{0}\leq a-\chi\underline{u}_{0}.
\end{equation*}
These last inequalities are exactly  the ones established in Lemma \ref{Asymp1Lem1}. Therefore, following the arguments of Theorem \ref{Main Theorem 3}, we obtain  that $\underline{u}_{0}=\overline{u}_{0}=\frac{a}{b}$. This complete the proof of the claim.\\
It follows from above that $\tilde{u}(0,0)=\frac{a}{b}.$ But by \eqref{aux-eq1},
 $$
 |\tilde u(0,0)-\frac{a}{b}|\ge \delta,
 $$
 which is a contradiction. Thus
\[
\lim_{t\rightarrow\infty}\sup_{|x|\leq ct}|u(x,t)-\frac{a}{b}|=0
\]
 for all $0\leq c< c^{\ast}.$  This  together with Lemma \ref{Asym of v} (i) implies \eqref{A_Eq8} with any $0<c_{\rm low}^*\le c^*$.

(2) We prove \eqref{A_Eq9}.   Let us denote by $\overline{U}(x,t)$ the classical solution of the Initial Value Problem
\begin{equation*}
\begin{cases}
L\overline{U}= \bar F^{2}(\overline{U}) \ \ \ x\in \R^N, \ t>0\nonumber\\
\overline{U}(x,0)=u_{0}(x)\ \ \ x\in \R^N
\end{cases}
\end{equation*}
where $L\overline{U}$ is given by \eqref{Definition of L} and $\bar F^2(u)=u(a+d-(b-\chi)u)$ where $d\gg 1$ is chosen such that $\|u_{0}\|_{\infty}<\frac{a+d}{b-\chi}$. By the comparison principle for parabolic equations,
$$
u(x,t)\leq \overline{U}(x,t)\leq U(t,\|u_{0}\|_{\infty})
$$
for all $x\in\R^{N}$ and $t\geq 0$.

 It follows from \eqref{A_Eq02} and the fact that $v\ge0$  that
$$
 \lim_{R\rightarrow \infty }\sup_{t>R, |x|>R}\left(  4\partial_{u}\bar F^{2}(0)-\chi^{2}\|\nabla v(x,t)\|^{2} \right) \geq
\lim_{R\rightarrow \infty }\sup_{t>R, |x|>R}\left(4\partial_{u}F^{1}(x,t,0)-\chi^{2}\|\nabla v(x,t)\|^{2} \right) >0.
$$

Since $ \bar F^{2} $ is of KPP type with $F^{2}(0)=F^{2}(\frac{a+d}{b-\chi})=0,$ by Theorem 1 in \cite{Henri1},  there exist two  compact sets $  {\underline{S}}\subset{\overline{S}}$ with non-empty interiors such that
\begin{equation*}
\begin{cases}
\text{for all compact set } K\subset \text{int}\underline{S}, \ \lim_{t\rightarrow\infty}\{ \sup_{x\in tK}|\overline{U}(x,t)- \frac{a+d}{b-\chi}|\}=0,\\
\text{for all closed set } F\subset  \mathbb{R}^{N}\setminus\overline{S}, \ \lim_{t\rightarrow\infty}\{ \sup_{x\in tF}|\overline{U}(x,t)|\}=0.
\end{cases}
\end{equation*}
Take $c_{\rm up}^{\ast}$ to be the diameter of $\overline{S}.$ For every $c>c_{\rm up}^{\ast}$ we have that $F:=\{x \ : \ |x|\geq c\}\subset\R^{N}\setminus\overline{S}$ and closed. Hence
\[
\lim_{t\rightarrow \infty}\sup_{|x|>ct}u(x,t)=0
\] whenever $c > c^{\ast}_{\rm up}.$
This together with  Lemma \ref{Asym of v} implies \eqref{A_Eq9}. By \eqref{A_Eq8} and \eqref{A_Eq9}, it is clear that $c_{\rm up}^*(u_0)\ge c_{\rm low}^*(u_0)$.
\end{proof}

\end{document}